\documentclass[10pt]{amsart}
\usepackage{amsfonts}
\usepackage{amssymb, latexsym, amsmath}
\usepackage{amsmath}
\usepackage{amssymb}
\usepackage{graphicx}

\renewcommand{\see}{\mathcal{S}_e}
\renewcommand{\geq}{\geqslant}
\renewcommand{\leq}{\leqslant}
\theoremstyle{plain}
\newtheorem{theorem}{Theorem}[section]
\newtheorem{lemma}[theorem]{Lemma}
\newtheorem{proposition}[theorem]{Proposition}
\newtheorem{definition}[theorem]{Definition}
\newtheorem{corollary}[theorem]{Corollary}

\newtheorem*{lemma*}{Lemma}
\theoremstyle{remark}
\newtheorem{remark}[theorem]{Remark}
\newtheorem*{remark*}{Remark}
\theoremstyle{remark}

\newtheorem*{notation*}{Notations}
\numberwithin{equation}{section}

\begin{document}
\title{Unstable surface waves in running water}
\author{Vera~Mikyoung~Hur}
\address{Massachusetts Institute of Technology, Department of Mathematics\\
77 Massachusetts Avenue, Cambridge, MA 02139-4307, USA}
\email{verahur@math.mit.edu}
\author{Zhiwu~Lin}
\address{University of Missouri-Columbia, Department of Mathematics, \\
Columbia, MO 65203-4100, USA}
\email{lin@math.missouri.edu}

\begin{abstract}
We consider the stability of periodic gravity free-surface water waves
traveling downstream at a constant speed over a shear flow of finite depth.
In case the free surface is flat, a sharp criterion of linear instability is
established for a general class of shear flows with inflection points and
the maximal unstable wave number is found. Comparison to the rigid-wall
setting testifies that the free surface has a destabilizing effect. For a
class of unstable shear flows, the bifurcation of nontrivial periodic
traveling waves is demonstrated at all wave numbers. We show the linear
instability of small nontrivial waves that appear after bifurcation at an
unstable wave number of the background shear flow. The proof uses a new
formulation of the linearized water-wave problem and a perturbation
argument. An example of the background shear flow of unstable
small-amplitude periodic traveling waves is constructed for an arbitrary
vorticity strength and for an arbitrary depth, illustrating that vorticity
has a subtle influence on the stability of free-surface water waves.
\end{abstract}

\maketitle

\section{Introduction}

The water-wave problem in its simplest form concerns two-dimensional motion
of an incompressible inviscid liquid with a free surface, acted on only by
gravity. Suppose, for definiteness, that in the $(x,y)$-Cartesian
coordinates gravity acts in the negative $y$-direction and that the liquid
at time $t$ occupies the region bounded from above by the free surface $%
y=\eta (t;x)$ and from below by the flat bottom $y=0$. In the fluid region $%
\{(x,y):0<y<\eta (t;x)\}$, the velocity field $(u(t;x,y),v(t;x,y))$
satisfies the incompressibility condition 
\begin{equation}
\partial _{x}u+\partial _{y}v=0  \label{E:Euler1}
\end{equation}%
and the Euler equation 
\begin{equation}
\begin{cases}
\partial _{t}u+u\partial _{x}u+v\partial _{y}u=-\partial _{x}P \\ 
\partial _{t}v+u\partial _{x}v+v\partial _{y}v=-\partial _{y}P-g,%
\end{cases}
\label{E:Euler2}
\end{equation}%
where $P(t;x,y)$ is the pressure and $g>0$ denotes the gravitational
constant of acceleration. The flow is allowed to have rotational motions and
characterized by the vorticity $\omega =v_{x}-u_{y}$. The kinematic and
dynamic boundary conditions at the free surface $\{y=\eta (t;x)\}$ 
\begin{equation}
v=\partial _{t}\eta +u\partial _{x}\eta \quad \text{and}\quad P=P_{\text{atm}%
}  \label{E:top}
\end{equation}%
express, respectively, that the boundary moves with the velocity of the
fluid particles at the boundary and that the pressure at the surface equals
the constant atmospheric pressure $P_{\text{atm}}$. The impermeability
condition at the flat bottom states that 
\begin{equation}
v=0\qquad \text{at}\quad \{y=0\}.  \label{E:bottom}
\end{equation}

It is a matter of common experience that waves which may be observed on the
surface of the sea or on the river are approximately periodic and
propagating of permanent form at a constant speed. In case $\omega \equiv 0$%
, namely in the irrotational setting, waves of this kind are referred to as
Stokes waves, whose mathematical treatment was initiated by formal but
far-reaching considerations of Stokes \cite{Sto} himself. The existence
theory of Stokes waves dates back to the construction due to Levi-Civita 
\cite{lc} and Nekrasov \cite{nek} in the infinite-depth case and due to
Struik \cite{struik} in the finite-depth case of small-amplitude waves, and
it includes the global theory due to Krasovskii \cite{kra} and Keady and
Norbury \cite{keno}. Stokes waves of greatest height exist \cite{toland}, 
\cite{mcleod} and are shown to have stagnation at wave crests \cite%
{aft-stokes}. A nice survey on the existence of Stokes waves include \cite%
{tol} and \cite[Chapter 10, 11]{buto}. In the finite-depth case, it is
recently shown by Constantin \cite{con} that there are no closed paths in
the Stokes waves and each particle experiences a slight forward drift.

While the irrotational assumption may serve as an approximation under
certain circumstances and has been used in a majority of the existing
research, surface water-waves typically carry vorticity, e.g. shear currents
on a shallow channel and wind-drift boundary layers. Moreover, the governing
equations for water waves allow for rotational steady motions. Gerstner \cite%
{gersner} early in 1802 found an explicit formula for a family of periodic
traveling waves on deep waters with a particular nonzero vorticity. An
extensive existence theory of periodic traveling water waves with vorticity
appeared in the construction due to Dubreil-Jacotin \cite{dubreil-local} of
small-amplitude waves. Recently, for a general class of vorticity
distributions, Constantin and Strauss \cite{cost} in the finite-depth case
and Hur \cite{hur1} in the infinite-depth case accomplished the bifurcation
analysis for periodic traveling waves of large amplitude. Partial results on
the location of possible stagnation are found in \cite{cost2, var}.

Waves of Stokes' kind is one of the few exact solutions of the free-surface
water-wave problem, and as such it is important to understand the stability
of these solutions. In this paper, we investigate the linear instability of
periodic gravity water-waves with vorticity. 
%This is the subject of the present investigation.

The stability of water waves in case of zero vorticity has been under
research much by means of numerical computations and formal analysis,
especially in the works of Longuet-Higgins and his coworkers. Numerical
studies of stability of Stokes waves under perturbations of the same period,
namely the superharmonic perturbations, indicate that \cite{longet-78-super}%
, \cite{tanaka83} instability sets in only when the wave amplitude is large
enough to link with wave breaking (\cite{longuet-domm}) and small-amplitude
Stokes waves are found to be linearly stable under the same-period
perturbations. The instability of large Stokes waves and solitary waves is
recently proved by Lin (\cite{lin-stokes}, \cite{lin-soli-water}), under the
assumption of no secondary bifurcation which is confirmed numerically.
MacKay and Saffman \cite{mc-saff} considered linear stability of
small-amplitude Stokes waves by the general results of the Hamiltonian
system. The Hamiltonian formulation in terms of the velocity potential,
however, does not avail in the presence of vorticity except when the
vorticity is constant \cite{wah}. The analysis of Benjamin and Feir \cite%
{benjamin-feir} showed that there is a \textquotedblleft
sideband\textquotedblright\ instability for small Stokes waves, meaning that
the perturbation has a different period than the steady wave. The
Benjamin-Feir instability was made mathematically rigorous by Bridges and
Mielke \cite{bm}.

Analytical works on the stability of water waves with vorticity, on the
other hand, are quite sparse. A recent contribution due to Constantin and
Strauss \cite{cost1} concerns two different kinds of formal stability of
periodic traveling water-waves with vorticity under perturbations of the
same period. First, in case when the vorticity decreases with depth an
energy-Casimir functional $\mathcal{H}$ is constructed as a temporal
invariant of the nonlinear water-wave problem, whose first variation gives
the exact equations for steady waves \cite{css}. In \cite{cost1}, the second
variation of $\mathcal{H}$ is shown to be positive for some special
perturbations and the water-wave system is called $\mathcal{H}$-formally
stable under these perturbations. The use of such energy-Casimir functional
in studying stability of ideal fluids is pioneered by Arnold \cite{arn} for
the fixed boundary case.%However, it remains unclear how to pass
%from the first kind of formal stability to the genuine stability
%since these perturbations may not stay in the same class at later time.
The second approach of \cite{cost1} uses another functional $\mathcal{J}$,
which is essentially the dual of $\mathcal{H}$ in the transformed variables
but not an invariant. Its first variation gives the exact equations for
steady waves in the transformed variables (\cite{css}), which served as the
basis in \cite{cost} for the existence theory of traveling waves. The $%
\mathcal{J}$-formal stability then means the positivity of its second
variation. The \textquotedblleft exchange of stability\textquotedblright\
theorem due to Crandall and Rabinowitz \cite{cr-ra-stability} applies to
conclude that the $\mathcal{J}$-formal stability of the trivial solutions
switches exactly at the bifurcation point and that steady waves along the
curve of local bifurcation are $\mathcal{J}$-formally stable provided that
both the depth and the vorticity strength are sufficiently small. 
%These conclusions are, however, in contradiction to our results (Remark 5.2)
%obtained from analyzing the exact linearized equations, indicating that
%the formal stability of second kind may be unrelated to linear stability.

\ 

\textbf{The main results.} As a preliminary step toward the stability and
instability of nontrivial periodic waves, we examine the linear stability
and instability of flat-surface shear flows. The linear stability of shear
flows in the rigid-wall setting is a classical problem, whose theories date
back to the necessary condition for linear instability due to Rayleigh \cite%
{ray}. We refer to \cite{cclin, dr81, dh66, frho, lin1} and references
therein for historic and recent results on this problem. In \cite{lin1}, Lin
obtained linear instability criteria for several classes of shear flows in a
channel with rigid walls, and in this paper we generalize these to the
free-surface setting. More specifically, our conclusions include: (1) The
linear stability of shear flows with no inflection points (Theorem \ref%
{T:stable}), which generalizes Rayleigh's criterion in the rigid-wall
setting \cite{ray} to the free-surface setting; (2) A sharp criterion of
linear instability for a class of shear flows with one inflection value
(Theorem \ref{class-k}); and (3) A sufficient condition of linear
instability for a class of shear flows with multiple inflection values
(Theorem \ref{classF}) including any monotone flows. Our result testifies
that free surface has a destabilizing effect compared to rigid walls.

Our next step is to understand the local bifurcation of small-amplitude
periodic traveling waves in the physical space. While our setting is similar
to \cite{cost1} in that it hinges on the existence results of periodic waves
in \cite{cost} via the local bifurcation, the choice of the bifurcation
parameter and the dependence of other parameters on the bifurcation
parameter and free parameters in the description of the background shear
flow are different. In our setting, it is natural to consider that the shear
profile and the channel depth are given and that the speed of wave
propagation is chosen to ensure the local bifurcation. The relative flux and
the vorticity-stream function relation are then computed. In contrast, in
the bifurcation analysis \cite{cost} in the transformed variables, the wave
speed, as well as the relative flux and the vorticity-stream function
relation are held fixed. In turn, the shear profile and the channel depth
vary along the bifurcation curve. Lemma \ref{lemma-equi-disper} establishes
the equivalence between the bifurcation equation (\ref{E:bifur-CS})
(equivalently \cite[(3.8)]{cost}) in the transformed variables and the
Rayleigh system (\ref{rayleigh-disper})--(\ref{bc-disper}) to obtain the
bifurcation results for a large class of shear flows. In addition, our
result helps to clarify the nature of the local bifurcation of periodic
traveling water-waves that it does not necessarily involve the exchange of
stability of trivial solutions (Remark \ref{remark-neutral}).

Our third step is to show under some technical assumptions that the linear
instability of the background shear flow persists along the local curve of
bifurcation of small-amplitude periodic traveling waves (Theorem \ref%
{T:unstable}). An example of such an unstable shear flow is 
\begin{equation*}
U(y)=a\sin b(y-h/2)\qquad \text{for}\quad y\in \lbrack 0,h],
\end{equation*}%
where $h,b>0$ satisfy $hb\leq \pi $ and $a>0$ is arbitrary (Remark \ref%
{R:example}). In particular, by choosing $a$ and $h$ to be arbitrarily
small, we can construct linearly unstable small periodic traveling
water-waves with an arbitrarily small vorticity strength and an arbitrarily
small channel depth. This indicates that the formal stability of the second
kind in \cite{cost1} (see discussions above) is quite different from the
linear stability of the physical water wave problem. Our example also shows
that adding an arbitrarily small vorticity to the water-wave system may
affect the superharmonic stability of small-amplitude periodic irrotational
waves in a water of arbitrary depth. Thus, it is important to take into
account of the effects of vorticity in the study of the stability of water
waves.

Temporal invariants of the linearized water-wave problem are derived and
their implications for the stability of the water-wave system are discussed
(Section 3.3). In case when the vorticity-stream function relation is
monotone, the energy functional $\partial ^{2}\mathcal{H}$ in \cite{cost1}
is indeed an invariant of the linearized water-wave problem. Other
invariants are also derived. However, even with these additional invariants
as constraints, the quadratic form $\partial ^{2}\mathcal{H}$ is in general
indefinite, indicating that that a steady (pure gravity) water-wave may be
an energy saddle. A similar observation was made by Bona and Sachs \cite%
{bona-sachs} in the irrotational case. Therefore, a successful proof of the
stability for the full water-wave problem would require to use the full
equations instead of just a few invariants.

\textbf{Ideas of the proofs.} Our approach in the proof of the linear
instability of free-surface shear flows uses the Rayleigh system (\ref%
{rayleigh})--(\ref{bc-rayleigh}), which is related to that in the rigid-wall
setting \cite{lin1}. The main difference from \cite{lin1} lies in the
complicated boundary condition (\ref{bc-rayleigh}) on the free surface,
which renders the analysis more involved. The instability property depends
on the wave number, which is considered as a parameter. As in the rigid-wall
setting \cite{lin1}, the key to a successful instability analysis is to
locate the neutral limiting modes, which are neutrally stable solution of
the Rayleigh system and contiguous to unstable modes. For certain classes of
flows, neutral limiting modes in the free-surface setting are characterized
by the inflection values. This together with the local bifurcation of
unstable modes from each neutral limiting wave number gives a complete
knowledge on the instability at all wave numbers.

The instability analysis of small-amplitude nontrivial waves taken here is
based on a new formulation which directly linearizes the Euler equation and
the kinematic and dynamic boundary conditions on the free surface around a
periodic traveling wave. Its growing-mode problem then is written as an
operator equation for the stream function perturbation restricted on the
steady free-surface. The mapping by the action-angle variables is employed
to prove the continuity of the operator with respect to the amplitude
parameter. In addition, in the action-angle variables, the equation of the
particle trajectory takes a very simple form. The persistence of instability
along the local curve of bifurcation is established by means of Steinberg's
eigenvalue perturbation theorem \cite{steinberg}. In addition, growing-mode
solutions is proved to acquire regularity up to that of the steady
profiles.\ 

This paper is organized as follows. Section 2 is the discussion on the local
bifurcation of periodic traveling water-waves when a background shear flow
in the physical space is given. Section 3 includes the formulation of the
linearized periodic water-wave problem and the derivation of its invariants.
Section 4 is devoted to the linear instability of shear flows with one
inflection value, and subsequently, Section 5 is to the linear instability
of small-amplitude periodic waves over an unstable shear flow. Section 6
revisits the linear instability of shear flows for a more general class.

\section{Existence of small-amplitude periodic traveling water-waves}

\label{S:steady}

We consider a traveling-wave solution of (\ref{E:Euler1})--(\ref{E:bottom}),
that is, a solution for which the velocity field, the wave profile and the
pressure have space-time dependence $(x-\underline{c} t, y)$, where $%
\underline{c}>0$ is the speed of wave propagation. With respect to a frame
of reference moving with the speed $\underline{c}$, the wave profile appears
to be stationary and the flow is steady. The traveling-wave problem for (\ref%
{E:Euler1})--(\ref{E:bottom}) is further supplemented with the periodicity
condition that the velocity field, the wave profile and the pressure are $%
2\pi/\alpha$-periodic in the $x$-variable, where $\alpha>0$ is the wave
number.

It is traditional in the traveling-wave problem to introduce the relative
stream function $\psi (x,y)$ such that 
\begin{equation}
\psi _{x}=-v,\qquad \psi _{y}=u-\underline{c}  \label{D:stream}
\end{equation}%
and $\psi (0,\eta (0))=0$. This reduces the traveling-wave problem for (\ref%
{E:Euler1})--(\ref{E:bottom}) to a stationary elliptic boundary value
problem \cite[Section 2]{cost}:

For a real parameter $B$ and a function $\gamma\in C^{1+\beta}([0,|p_0|])$, $%
\beta\in (0,1)$, find $\eta(x)$ and $\psi(x,y)$ which are $2\pi/\alpha$%
-periodic in the $x$-variable, $\psi_y(x,y)<0$ in $\{(x,y): 0<y<\eta(x)\}$%
\footnote{%
In other words, there is no stagnation in the fluid region. Field
observations \cite{lighthill} as well as laboratory experiments \cite{tk}
indicate that for wave patterns which are not near the spilling or breaking
state, the speed of wave propagation is in general considerably larger than
the horizontal velocity of any water particle.}, and 
\begin{subequations}
\label{stream}
\begin{alignat}{2}
-\Delta\psi & =\gamma(\psi) \qquad & & \text{in}\quad 0<y<\eta(x){,} \\
\psi & =0\quad & & \text{on}\quad y=\eta(x){,} \\
|\nabla\psi|^{2}+ & 2gy=B\qquad & & \text{on}\quad y=\eta(x){,} \\
\psi & =-p_0\quad & & \text{on}\quad y=0,
\end{alignat}
where 
\end{subequations}
\begin{equation}
p_{0}=\int_{0}^{\eta(x)}\psi_{y}(x,y)dy  \label{E:flux}
\end{equation}
is the relative total flux\footnote{$p_{0}<0$ is independent of $x$.}.

The vorticity function $\gamma $ gives the vorticity-stream function
relation, that is, $\omega =\gamma (\psi )$. The assumption of no
stagnation, i.e. $\psi _{y}(x,y)<0$ in the fluid region $\{(x,y):0<y<\eta
(x)\}$, guarantees that such a function is well-defined globally; See \cite%
{cost}. Furthermore, under this physically motivated stipulation,
interchanging the roles of the $y$-coordinates and $\psi $ offers an
alternative formulation to (\ref{stream}) in a fixed strip, which serves as
the basis of the existence theories in \cite{dubreil-local}, \cite{cost}, 
\cite{hur1}. The nonlinear boundary condition (\ref{stream}c) at the free
surface $y=\eta (x)$ expresses Bernoulli's law. The steady hydrostatic
pressure in the fluid region is given by 
\begin{equation}
P(x,y)=B-\tfrac{1}{2}|\nabla \psi (x,y)|^{2}-gy-\int_{0}^{\psi (x,y)}\gamma
(-p)dp.  \label{eqn-steady-pressure}
\end{equation}

In this setting, $\alpha$ and $B$ are considered as parameters whose values
form part of the solution. The wave number $\alpha$ in the existence theory
is independent of other physical parameters and hence is held fixed, while
in the stability analysis in Section \ref{S:instability0} it serves as
parameter. The Bernoulli constant $B$ measures the total mechanical energy
of the flow and varies along a solution branch.

\subsection{The local bifurcation theorem in \protect\cite{cost}}

This subsection contains a summary of the existence result in \cite{cost}
via the local bifurcation theorem of small-amplitude travelling-wave
solutions to (\ref{stream}), provided that the total flux $p_{0}$ and the
vorticity-stream function relation $\gamma $ are given.

A preliminary result for the local bifurcation is to find a curve of trivial
solutions, which correspond to horizontal shear flows under a flat surface.
As in \cite[Section 3.1]{cost}, let 
\begin{equation*}
\Gamma (p)=\int_{0}^{p}\gamma (-p^{\prime })dp^{\prime },\qquad \Gamma
_{\min }=\min_{[p_{0},0]}\Gamma (p)\leq 0.
\end{equation*}

\begin{lemma}[\protect\cite{cost}, Lemma 3.2]
\label{L:trivial} Given $p_0<0$ and $\gamma \in C^{1+\beta}([0, |p_0|])$, $%
\beta \in (0,1)$, for each $\mu \in( -2\Gamma_{\min}, \infty)$ the system (%
\ref{stream}) has a solution 
\begin{equation*}
y(p)=\int^{p}_{p_0} \frac{dp^{\prime }}{\sqrt{\mu+2\Gamma(p^{\prime })}},
\end{equation*}
which corresponds to a parallel shear flow in the horizontal direction 
\begin{equation}  \label{E:trivialu}
u(t;x,y)=U(y;\mu)=\underline{c}-\sqrt{\mu+2\Gamma(p(y))}
\end{equation}
and $v(t;x,y)\equiv 0$ in the channel $\{(x,y): 0<y<h(\mu)\}$, where 
\begin{equation*}
h(\mu) =\int_{p_0}^0\frac{dp}{\sqrt{\mu+2\Gamma(p)}}.
\end{equation*}
The hydrostatic pressure is $P(y)=-gy$ for $y\in [0, h(\mu)]$. Here, $p(y)$
is the inverse of $y=y(p)$ and determines the stream function $%
\psi(y;\mu)=-p(y;\mu)$; $\underline{c}>0$ is arbitrary.
\end{lemma}

In the statement of Theorem \ref{T:smallE} below, instead of $B$ the squared
(relative) upstream flow speed $\mu =(U(h)-\underline{c})^{2}$ of a trivial
shear flow (\ref{E:trivialu}) serves as the bifurcation parameter. For each $%
\mu \in (-2\Gamma _{\min },\infty )$ the Bernoulli constant $B$ is
determined uniquely in terms of $\mu $ by 
\begin{equation}
B=\mu +2g\int_{p_{0}}^{0}\frac{dp}{\sqrt{\mu +2\Gamma (p)}}.  \label{E:ber}
\end{equation}

The following theorem \cite[Theorem 3.1]{cost} states the existence result
of a one-parameter curve of small-amplitude periodic water-waves for a
general class of vorticities and their properties, in a form convenient for
our purposes.

\begin{theorem}[Existence of small-amplitude periodic water-waves]
\label{T:smallE} Let the speed of wave propagation $\underline{c}>0$, the
flux $p_{0}<0$, the vorticity function $\gamma\in C^{1+\beta}([0,|p_{0}|])$, 
$\beta\in(0,1)$, and the wave number $\alpha>0$ be given such that the
system 
\begin{equation}
\begin{cases}
(a^{3}(\mu)M_{p})_{p}=\alpha^{2}a(\mu)M\qquad\text{for}\quad p\in (p_{0},0)
\\ 
\mu^{3/2}M_{p}(0)=gM(0) \\ 
M(p_{0})=0%
\end{cases}
\label{E:bifur-CS}
\end{equation}
admits a nontrivial solution for some $\mu_{0} \in (-2\Gamma_{\min}, \infty)$%
, where $a(\mu)=a(\mu;p)=\sqrt{\mu+2\Gamma(p)}$.

Then, for $\epsilon\geq0$ sufficiently small there exists a one-parameter
curve of steady solution-pair $\mu_{\epsilon}$ of (\ref{E:ber}) and $%
(\eta_{\epsilon }(x),\psi_{\epsilon}(x,y))$ of \textrm{(\ref{stream}) }such
that $\eta_{\epsilon}(x)$ and $\psi_{\epsilon}(x,y)$ are $2\pi/\alpha $%
-periodic in the $x$-variable, of $C^{3+\beta}$ class, where $\beta \in
(0,1) $, and $\psi_{\epsilon y}(x,y)<0$ throughout the fluid region.

At $\epsilon=0$ the solution corresponds to a trivial shear flow under a
flat surface: \renewcommand{\labelenumi}{{\normalfont (\roman{enumi}0)}}

\begin{enumerate}
\item The flat surface is given by $\eta_{0}(x)\equiv h(\mu_0)=:h_0$ and the
velocity field is 
\begin{equation*}
(\psi_{0y}(x,y),-\psi_{0x}(x,y))=(U(y)-\underline{c},0),
\end{equation*}
where $U(y)$ is determined in (\ref{E:trivialu});

\item The pressure is given by the hydrostatic law $P_{0}(x,y)=-gy$ for $%
y\in [0,h_{0}]$.
\end{enumerate}

At each $\epsilon>0$ the corresponding nontrivial solution enjoys the
following properties: 
\renewcommand{\labelenumi}{{\normalfont
(\roman{enumi}$\epsilon$)}}

\begin{enumerate}
\item The bifurcation parameter has the asymptotic expansion 
\begin{equation*}
\mu_{\epsilon}=\mu_{0}+O(\epsilon)\qquad\text{as}\quad\epsilon\rightarrow0
\end{equation*}
and the wave profile is given by 
\begin{equation*}
\eta_{\epsilon}(x)=h_{\epsilon}+\alpha^{-1}\delta_{\gamma}\epsilon\cos\alpha
x+O(\epsilon^{2})\qquad\text{as}\quad\epsilon\rightarrow0,
\end{equation*}
where $h_{\epsilon}=h(\mu_{\epsilon})$ is given in Lemma \ref{L:trivial} and 
$\delta_{\gamma}$ depends only on $\gamma$ and $p_{0}$; The mean height
satisfies 
\begin{equation*}
h_\epsilon=h_0+O(\epsilon) \qquad \text{as} \quad \epsilon \to 0;
\end{equation*}
Furthermore, the wave profile is of mean-zero; That is, 
\begin{equation}
\int_{0}^{2\pi/\alpha}(\eta_{\epsilon}(x)-h_{\epsilon})dx=0;
\label{mean-zero}
\end{equation}

\item The velocity field $(\psi_{\epsilon y}(x,y),-\psi_{\epsilon x}(x,y))$
in the steady fluid region $\{(x,y):0<x<2\pi/\alpha\,,\,0<y<\eta_{%
\epsilon}(x)\} $ is given by 
\begin{align*}
\psi_{\epsilon x}(x,y) & =\epsilon\psi_{\ast x}(y)\sin\alpha x+O(\epsilon
^{2}), \\
\psi_{\epsilon y}(x,y) & =U(y)-\underline{c}+\epsilon\psi_{\ast
y}(y)\cos\alpha x+O(\epsilon^{2})
\end{align*}
as $\epsilon\rightarrow0$, where $\psi_{\ast x}$ and $\psi_{\ast y}$ are
determined from the linear theory;

\item The hydrostatic pressure has the asymptotic expansion 
\begin{equation*}
P_{\epsilon}(x,y)=-gy+O(\epsilon)\qquad\text{as}\quad\epsilon\rightarrow0.
\end{equation*}
\end{enumerate}
\end{theorem}

The condition that the system (\ref{E:bifur-CS}) admits a nontrivial
solution for some $\mu _{0}\in (-2\Gamma _{\min },\infty )$ is necessary and
sufficient for the local bifurcation \cite[Section 3]{cost}. A sufficient
condition (\cite{cost}) for the solvability of (\ref{E:bifur-CS}) and
therefore the local bifurcation is 
\begin{equation*}
\int_{p_{0}}^{0}\left( \alpha ^{2}(p-p_{0})^{2}(2\Gamma (p)-2\Gamma _{\min
})^{1/2}+(2\Gamma (p)-2\Gamma _{\min })^{3/2}\right) dp<gp_{0}^{2},
\end{equation*}%
which is satisfied when $p_{0}$ is sufficiently small. 
%, and it can be checked when $p_{0}$ is sufficiently small \cite{cost}$.$

\subsection{The bifurcation condition for a given shear flow}

Our instability analyses in Section 4 and Section 5 are carried out in the
physical space, where a shear-flow profile and the water depth are held
fixed. The bifurcation analysis in the proof of Theorem \ref{T:smallE}, on
the other hand, is carried out in the space of transformed variables, where
the travel spped $\underline{c}$, as well as the relative flux $p_{0}$ and
the vorticity-stream function relation $\gamma $ are held fixed, and the
shear flow $U(y)-\underline{c}$ and the water depth $h$ vary along the curve
of local bifurcation. In this subsection, we study the local bifurcation in
the physical space, with a given shear flow $U\left( y\right) $ and the
water depth $h$, which is relevant to the later instability analyses. The
natural choice for parameters is the speed of wave propagation $\underline{c}%
>\max U$ and the wave number $k$.

Our first task is to relate the bifurcation equation (\ref{E:bifur-CS}) in
transformed variables with the Rayleigh system in the physical variables.

\begin{lemma}
\label{lemma-equi-disper} For the shear flow $U(y)$ with $y\in \lbrack 0,h]$
which is defined via Lemma \ref{L:trivial}, the bifurcation equation (\ref%
{E:bifur-CS}) is equivalent to the following Rayleigh equation 
\begin{equation}
(U-c)(\phi ^{\prime \prime }-k^{2}\phi )-U^{\prime \prime }\phi =0\qquad 
\text{for}\quad y\in (0,h)  \label{rayleigh-disper}
\end{equation}%
with the boundary conditions 
\begin{equation}
\phi ^{\prime }(h)=\left( \frac{g}{(U(h)-c)^{2}}+\frac{U^{\prime }(h)}{U(h)-c%
}\right) \phi (h)\quad \text{and}\quad \phi (0)=0,  \label{bc-disper}
\end{equation}%
where $(c,k)=(\underline{c},\alpha )$, $\underline{c}>\max U$, and $\phi
(y)=(\underline{c}-U(y))M(p(y))$. Here and in the sequel, the prime denotes
the differentiation in the $y$-variable.
\end{lemma}

\begin{proof}
Notice that $\underline{c}>\max U$. Indeed, 
\begin{equation*}
a(\mu ;p)=\sqrt{\mu +2\Gamma (p)}=-(U(y(p))-\underline{c})>0.
\end{equation*}%
Since $\frac{\partial p}{\partial y}=-\psi ^{\prime }(y)=-(U(y)-\underline{c}%
),$ it follows that $\partial _{p}=\partial _{y}\frac{\partial y}{\partial p}%
=-\frac{1}{U-\underline{c}}\partial _{y}$. Let $M(p(y))=\Phi (y)$, then (\ref%
{E:bifur-CS}) is written as 
\begin{gather*}
((U-\underline{c})^{2}\Phi ^{\prime })^{\prime }-\alpha ^{2}(U-\underline{c}%
)^{2}\Phi =0\qquad \text{for}\quad y\in (0,h), \\
(U-\underline{c})^{2}\Phi ^{\prime }(h)=g\Phi (h)\quad \text{and}\quad \Phi
(0)=0.
\end{gather*}%
Let $\phi (y)=(\underline{c}-U(y))\Phi (y)$, and the above system becomes (%
\ref{rayleigh-disper})-(\ref{bc-disper}).
\end{proof}

\begin{remark}
\label{remark-neutral-bifur}We illustrate how to construct
downstream-traveling periodic waves of small-amplitude bifurcating from a
fixed background shear-flow $U(y)$ for $y\in \lbrack 0,h]$. First, one finds
the parameter values $(c,k)=(\underline{c},\alpha )$ with $\underline{c}%
>\max U$ and $\alpha >0$ such that the Rayleigh system (\ref{rayleigh-disper}%
)-- (\ref{bc-disper}) admits a nontrivial solution. The wave speed then
determines the bifurcation parameter via $\mu _{0}=(U(h)-\underline{c})^{2}$%
, and the flux and the vorticity function are determined by 
\begin{equation}
p_{0}=\int_{0}^{h}(U(y)-\underline{c})dy\quad \text{and}\quad \gamma
(p)=U^{\prime }(y(p)),  \label{eqn-flux-vorticity}
\end{equation}%
respectively. By Lemma \ref{lemma-equi-disper}, the bifurcation equation (%
\ref{E:bifur-CS}) with $\mu _{0}$, $p_{0}$ and $\gamma $ as above has a
nontrivial solution. Moreover, each shear flow $U(h)-\underline{c}$ for $%
\underline{c}>\max U$ corresponds to a trivial solution in Lemma \ref%
{L:trivial}. Indeed, $\mu $ and $\underline{c}$ has a one-to-one
correspondence via $\mu =(U(h)-\underline{c})^{2}$; The (relative) stream
function defined as 
\begin{equation*}
\psi (y)=-\int_{y}^{h}(U(h)-\underline{c})dy
\end{equation*}%
is monotone and its inverse $y=y(-\psi )$ is well-defined. Then, Theorem \ref%
{T:smallE} applies in the setting above to give a local curve of bifurcation
of periodic waves.
\end{remark}

The lemma below obtains for a large class of shear flows the local
bifurcation by showing that the Rayleigh system (\ref{rayleigh-disper})--(%
\ref{bc-disper}) has a nontrivial solution.

\begin{lemma}
\label{lemma-bifur}If 
\begin{equation}
U\in C^{2}([0,h]),\quad U^{\prime\prime}(h)<0\quad \text{and} \quad
U(h)>U(y) \quad \text{for} \quad y\neq h,  \label{condition-U-neutral}
\end{equation}
then for any wave number $k>0$ there exists $c(k)>U(h)=\max U$ such that the
system (\ref{rayleigh-disper})--(\ref{bc-disper}) has a nontrivial solution $%
\phi$ with $\phi>0$ in $(0,h]$.
\end{lemma}

\begin{proof}
For $c\in (U(h),\infty )$ and $k>0$, let $\phi _{c}$ be the solution of (\ref%
{rayleigh-disper}), or equivalently, 
\begin{equation}
\left( (U-c)\phi _{c}^{\prime }-U^{\prime }\phi _{c}\right) ^{\prime
}-k^{2}(U-c)\phi _{c}=0\qquad \text{for}\quad y\in (0,h)  \label{eqn-Vc-phi}
\end{equation}%
with $\phi _{c}(0)=0$ and $\phi _{c}^{\prime }(0)=1$. An integration of the
above equation on the interval $[0,h]$ yields that 
\begin{equation*}
(U(h)-c)\phi _{c}^{\prime }(h)-(U(0)-c)-\phi _{c}(h)U^{\prime }\left(
h\right) -k^{2}\int_{0}^{h}(U-c)\phi _{c}dy=0
\end{equation*}%
Note that the bifurcation condition (\ref{bc-disper}) is fulfilled if and
only if the function 
\begin{equation}
f(c)=c-U(0)+k^{2}\int_{0}^{h}(c-U)\phi _{c}dy-\frac{g}{c-U(h)}\phi _{c}(h)
\label{defn-fc}
\end{equation}%
has a zero at some $c(k)>U(h)$. It is easy to see that $f$ is a continuous
function of $c$ for $c>U(h)$.

First, we claim that $\phi _{c}(y)>0$ for $y\in (0,h]$. Suppose, on the
contrary, that $\phi _{c}(y_{0})=0$ for some $y_{0}\in (0,h]$. Note that (%
\ref{rayleigh-disper}) can be written as a Sturm-Liouville equation 
\begin{equation}
\phi _{c}^{\prime \prime }-k^{2}\phi _{c}-\frac{U^{\prime \prime }}{U-c}\phi
_{c}=0.  \label{E:rayleigh-sturm}
\end{equation}%
Since 
\begin{equation*}
(c-U)^{\prime \prime }+\frac{U^{\prime \prime }}{c-U}(c-U)=0\qquad \text{for}%
\quad y\in (0,h),
\end{equation*}%
by Sturm's first comparison theorem the function $c-U(y)$ must have a zero
on the interval $(0,y_{0})$. A contradiction then proves the claim.

Our goal is to show that $f(c)>0$ for $c$ large enough and $f(c)<0$ as $%
c\rightarrow U(h)+$. Then by continuity, $f$ vanishes at some $c>U\left(
h\right) $. First, when $c\rightarrow \infty $ the sequence of solutions $%
\phi _{c}$ of (\ref{rayleigh-disper}), or equivalently (\ref%
{E:rayleigh-sturm}) converges in $C^{2}$, that is, $\phi _{c}\rightarrow
\phi _{\infty }$. The limit $\phi _{\infty }$ satisfies the boundary value
problem 
\begin{equation*}
\phi _{\infty }^{\prime \prime }-k^{2}\phi _{\infty }=0\qquad \text{for}%
\quad y\in (0,h)
\end{equation*}%
with $\phi _{\infty }(0)=0$ and $\phi _{\infty }^{\prime }(0)=1$. Therefore, 
$\phi _{\infty }$ is bounded, continuous, and positive on $(0,h]$. By the
definition (\ref{defn-fc}), then it follows that $f(c)\rightarrow \infty $
as $c\rightarrow \infty $.

Next is to examine $f(c)$ as $c\rightarrow U(h)+$. Denote $\varepsilon
=c-U(h)>0$. We claim that: 
\begin{equation}
\phi _{c}(h)\geq C_{1}>0,\ \ \text{for }\varepsilon >0\ \text{sufficiently
small,}  \label{claim-lower}
\end{equation}%
where $C_{1}>0$ is independent of $\varepsilon $. To see this, it is
convenient to write (\ref{eqn-Vc-phi}) as 
\begin{equation*}
\left( (c-U)\phi _{c}^{\prime }-(c-U)^{\prime }\phi _{c}\right) ^{\prime
}=k^{2}(c-U)\phi _{c}>0\qquad \text{for}\quad y\in (0,h),
\end{equation*}%
whence 
\begin{equation}
(c-U(y))\phi _{c}^{\prime }(y)-(c-U(y))^{\prime }\phi _{c}(y)>c-U(0)>0\qquad 
\text{for}\quad y\in (0,h).  \label{inter9}
\end{equation}%
Since 
\begin{equation}
(c-U)\phi _{c}^{\prime }-(c-U)^{\prime }\phi _{c}=(c-U)^{2}\left( \frac{\phi
_{c}}{c-U}\right) ^{\prime },  \label{inter6}
\end{equation}%
it follows from (\ref{inter9}) that 
\begin{equation*}
\left( \frac{\phi _{c}}{c-U}\right) ^{\prime }>\frac{c-U(0)}{(c-U)^{2}}
\end{equation*}%
and an integration of the above on $\left[ 0,h\right] \ $yields that 
\begin{equation*}
\phi _{c}(h)>(c-U(h))\int_{0}^{h}\frac{c-U(0)}{(\underline{c}-U(y))^{2}}dy.
\end{equation*}%
Our assumption on $U(y)$ asserts that $0\leq U(h)-U(y)\leq \beta (h-y)$ for $%
y\in \lbrack h-\delta ,h]$ and $\delta >0$ sufficiently small, where $\beta
>0$ is a constant. Thus, 
\begin{align*}
\phi _{c}(h)& >\varepsilon (c-U(0))\int_{h-\delta }^{h}\frac{1}{(\varepsilon
+\beta (h-y))^{2}}dy \\
& =(c-U(0))\int_{0}^{(h-\delta )/\varepsilon }\frac{1}{(1+\beta x)^{2}}%
dx\geq C_{1}>0,
\end{align*}%
where $C_{1}>0$ is independent of $\varepsilon >0$. This proves the claim (%
\ref{claim-lower}).

To prove that $f(c)<0$ as $c\rightarrow U(h)+$, we consider the following
two cases.

Case1: $\max \phi _{c}(y)=\phi _{c}(h).$ It follows that 
\begin{equation*}
f(c)\leq c-U(0)+C_{1}\left( k^{2}\int_{0}^{h}(c-U)dy-\frac{g}{c-U(h)}\right)
<0
\end{equation*}%
provided that $c-U(h)=\varepsilon >0$ is small enough.

Case 2: $\max \phi _{c}=\phi _{c}(y_{c})$, where $y_{c}\in (0,h).$ Since $%
\phi _{c}^{\prime \prime }(y_{c})\leq 0$ and $\phi _{c}(y_{c})>0$, by (\ref%
{E:rayleigh-sturm}) we have $U^{\prime \prime }(y_{c})>0$. On the other
hand, $U^{\prime \prime }(h)\leq 0$ and hence $y_{c}\in \lbrack 0,h-\delta ]$
for some $\delta >0$. Note that $c-U(y)$ is bounded away from zero on the
interval $y\in \lbrack 0,h-\delta ]$. Since the coefficients of (\ref%
{E:rayleigh-sturm}) are uniformly bounded for $c$ on $y\in \lbrack
0,h-\delta ]$, the solution $\phi _{c}$ is uniformly bounded on $y\in
\lbrack 0,h-\delta ]$. In particular, $0<\phi _{c}(y_{c})\leq C_{2}$
independently for $c$. Therefore, 
\begin{equation*}
f(c)\leq c-U(0)+k^{2}hC_{2}\max_{[0,h]}(c-U(y))-\frac{gC_{1}}{c-U(h)}<0,
\end{equation*}%
when $\varepsilon =c-U(h)$ is small enough. This completes the proof.
\end{proof}

Lemma \ref{lemma-equi-disper} and Remark \ref{remark-neutral-bifur} ensure
the local bifurcation from a shear flow satisfying (\ref{condition-U-neutral}%
) at any wave number $k>0$, as stated below.

\begin{theorem}
\label{T:bifur-fixed U}If $U\in C^{2}([0,h])$, $U^{\prime \prime}(h)<0$ and $%
U(h)>U(y)$ for $y\neq h$ then for an arbitrary wavelength $2\pi/k$, where $%
k>0$, there exist small-amplitude periodic waves bifurcating in the sense as
in Theorem \ref{T:smallE} from the flat-surface shear flow $U(y)$, where $%
c(k) >\max U$.
\end{theorem}

In the irrotational setting, i.e. $U\equiv 0$, the parameter values $c$ and $%
k$ for which the Rayleigh system (\ref{rayleigh-disper})-- (\ref{bc-disper})
is solvable give the dispersion relation (i.e. \cite{crapper}) 
\begin{equation*}
c^{2}=\frac{g\tanh (kh)}{k}.
\end{equation*}%
In the case of a nonzero background shear flow, such an explicit algebraic
relation is in general unavailable. Still, the solvability of the Rayleigh
problem (\ref{rayleigh-disper})--(\ref{bc-disper}) may be considered to give
a generalized dispersion relation. Moreover, we have the following
quantitative information about $c(k)$.

\begin{lemma}
\label{lemma-dispersion-property}Given a shear flow $U(y)$ in $[0,h]$, let $%
k $ and $c(k)>\max U$ be such that (\ref{rayleigh-disper})--(\ref{bc-disper}%
) has a nontrivial solution. Then,

\textrm{{(a)} $c(k)$ is bounded for $k>0$; }

\textrm{{(b)} If $k_{1}\neq k_{2}$ then $c(k_{1})\neq c(k_{2})$; }

\textrm{{(c)} In the long wave limit $k\rightarrow 0+$, the limit of the
wave speed $c(0)$ satisfies Burns condition \cite{burns} 
\begin{equation}
\int_{0}^{h}\frac{dy}{(U-c(0))^{2}}=\frac{1}{g}.  \label{burns}
\end{equation}%
}
\end{lemma}

\begin{proof}
As in the proof of Lemma \ref{lemma-bifur}, the solvability of (\ref%
{rayleigh-disper})--(\ref{bc-disper}) is equivalent to the vanishing of the
function $f\left( c\right) $ defined by (\ref{defn-fc}).

(a) The proof of Lemma \ref{lemma-bifur} implies that for each $A>0$ there
exists $C_{A}$ such that $f(c)$ is positive when $c>C_{A},\ k<A$. In
interpretation, $c(k)\leq C_{A}$ for $k<A$. Thus, it suffices to show that $%
f(c)>0$ when $c$ and $k$ are large enough. Indeed, let $c>\max U$ be large
enough so that $\left\vert \frac{U^{\prime \prime }}{c-U}\right\vert \leq 1$
and let us denote by $\phi _{1}$ and $\phi _{2}$ the solutions of 
\begin{equation*}
\phi _{1}^{\prime \prime }+\left( 1-k^{2}\right) \phi _{1}=0\quad \text{and}%
\quad \phi _{2}^{\prime \prime }+\left( -1-k^{2}\right) \phi _{2}=0,
\end{equation*}%
respectively, with $\phi _{j}(0)=0$ and $\phi _{j}^{\prime }(0)=1$, where $%
j=1,2$. It is straightforward to see that 
\begin{equation*}
\phi _{1}(y)=\frac{1}{\sqrt{k^{2}-1}}\sinh (\sqrt{k^{2}-1})y\quad \text{and}%
\quad \phi _{2}(y)=\frac{1}{\sqrt{k^{2}+1}}\sinh (\sqrt{k^{2}+1})y.
\end{equation*}%
Sturm's second comparison theorem \cite{hartman} then asserts that the
solution $\phi _{c,k}$ of (\ref{E:rayleigh-sturm}) and $\phi _{1}$ $\phi
_{2} $ satisfy that 
\begin{equation*}
\frac{\phi _{1}^{\prime }}{\phi _{1}}\leq \frac{\phi _{c,k}^{\prime }}{\phi
_{c,k}}\leq \frac{\phi _{2}^{\prime }}{\phi _{2}},
\end{equation*}%
and thus $\phi _{1}\leq \phi _{c,k}\leq \phi _{2}$. It is then easy to see
that $f(c)>0$ when $k$ is big enough.

(b) Suppose on the contrary that $c(k_{1})=c(k_{2})=c$ for $k_{1}<k_{2}$.
Let us denote by $\phi _{c,k_{1}}$ and $\phi _{c,k_{2}}$ the corresponding
nontrivial solutions of (\ref{rayleigh-disper})--(\ref{bc-disper}). By
Sturm's second comparison theorem \cite{hartman}, it follows that 
\begin{equation*}
\frac{\phi _{c,k_{1}}^{\prime }(h)}{\phi _{c,k_{1}}(h)}<\frac{\phi
_{c,k_{2}}^{\prime }(h)}{\phi _{c,k_{2}}(h)}.
\end{equation*}%
This contradicts the boundary condition (\ref{bc-disper}).

(c) The Rayleigh equation (\ref{eqn-Vc-phi}) for $k=0$ implies that 
\begin{equation*}
(c-U(y))\phi _{c}^{\prime }(y)+U^{\prime }(y)\phi _{c}(y)=m,
\end{equation*}%
where $m$ is a constant. So by (\ref{inter6}), 
\begin{equation*}
\left( \frac{\phi _{c}}{c-U}\right) ^{\prime }=\frac{m}{(c-U)^{2}}
\end{equation*}%
and an integration of above from $0$ to $h$ gives 
\begin{equation*}
\frac{\phi _{c}\left( h\right) }{c-U\left( h\right) }=m\int_{0}^{h}\frac{1}{%
(c-U(y))^{2}}dy.
\end{equation*}%
On the other hand, by (\ref{bc-disper}), 
\begin{equation*}
m=(c-U(h))\phi _{c}^{\prime }(h)+U^{\prime }(h)\phi _{c}(h)=\frac{g}{(c-U(h))%
}\phi _{c}(h).
\end{equation*}%
A combination of above gives (\ref{burns}).
\end{proof}

The limiting parameter value $\mu $ which corresponds to the limiting wave
speed $c(0)$ gives the lowest hydraulic head $B$ defined in (\ref{E:ber});
see \cite[Section 3]{cost} for detail. The limiting wave speed $c(0)$ is the
critical parameter near which small solitary waves of elevation exist \cite%
{ter-krikorov}, \cite{hur2}.

\section{Linearization of the periodic gravity water-wave problem}

This section includes the detailed account of the linearization of the
water-wave problem (\ref{E:Euler1})--(\ref{E:bottom}) around a periodic
traveling wave which solves (\ref{stream}). The growing-mode problem is
formulated as a set of operator equations. Invariants of the linearized
problem are derived, and their implications in the stability of water waves
are discussed.

\subsection{Derivation of the linearized problem of periodic water waves}

A periodic traveling-wave solution of (\ref{stream}) is held fixed, and it
serves as the undisturbed state about which the system (\ref{E:Euler1})--(%
\ref{E:bottom}) is linearized. The derivation is performed in the moving
frame of references, in which the wave profile appears to be stationary and
the flow is steady. 
%Consider a periodic water wave with a wavelength $2\pi/\alpha$ and travelling speed $\underline{c}$.
%Throughout this section, we work on the moving wave
%frame $\left(  x-\underline{c}t,y,t\right)  $ and thus $\underline{c}$ does
%not appear in our derivation below.
Let us denote the undisturbed wave profile and (relative) stream function by 
$\eta _{e}(x)$ and $\psi _{e}(x,y)$, respectively, which satisfy the system (%
\ref{stream}). The steady (relative) velocity field $(u_{e}(x,y)-\underline{c%
},v_{e}(x,y))=$ $(\psi _{ey}(x,y),-\psi _{ex}(x,y))$ is given by (\ref%
{D:stream}), and the hydrostatic pressure $P_{e}(x,y)$ is determined by (\ref%
{eqn-steady-pressure}). Let 
\begin{equation*}
\mathcal{D}_{e}=\{(x,y):0<x<2\pi /\alpha \,,\,0<y<\eta _{e}(x)\}\quad \text{%
and}\quad \mathcal{S}_{e}=\{(x,\eta _{e}(x)):0<x<2\pi /\alpha \}
\end{equation*}%
denote, respectively, the undisturbed fluid domain of one period and the
steady wave profile. The steady vorticity $\omega _{e}(x,y)$ is given by $%
\omega _{e}=-\Delta \psi _{e}=\gamma (\psi _{e})$.

The linearization concerns a slightly-perturbed time-dependent solution of
the nonlinear problem (\ref{E:Euler1})--(\ref{E:bottom}) near the steady
state $(\eta _{e}(x),\psi _{e}(x,y))$. Let us denote the small perturbation
of the wave profile, the velocity field and the pressure by $\eta
_{e}(x)+\eta (t;x)$, $(u_{e}(x,y)-\underline{c}%
+u(t;x,y),v_{e}(x,y)+v(t;x,y)) $ and $P_{e}(x,y)+P(t;x,y)$, respectively. We
expand the nonlinear equations (\ref{E:Euler1})--(\ref{E:bottom}) around the
steady state in the order of small perturbations and restrict the
first-order terms to the steady domain and the boundary to obtain linearized
equations for the the deviations $\eta (t;x)$, $(u(t;x,y),v(t;x,y))$, and $%
P(t;x,y)$ in the wave profile, the velocity field and the pressure from
those of the undisturbed state. %which we describe in details now.

In the steady fluid domain $\mathcal{D}_{e}$, the velocity deviation $(u,v)$
satisfies the incompressibility condition 
\begin{equation}
\partial _{x}u+\partial _{y}v=0  \label{E:Euler1L}
\end{equation}%
and the linearized Euler equation 
\begin{equation}
\begin{cases}
\partial _{t}u+(u_{e}-\underline{c})\partial _{x}u+u_{ex}u+v_{e}\partial
_{y}u+v_{ey}v=-\partial _{x}P \\ 
\partial _{t}v+(u_{e}-\underline{c})\partial _{x}v+v_{ex}u+v_{e}\partial
_{y}v+v_{ey}v=-\partial _{y}P,%
\end{cases}
\label{E:velocityL}
\end{equation}%
where $P$ is the pressure deviation. Equation (\ref{E:Euler1L}) allows us to
introduce the stream function $\psi (t;x,y)$ for the velocity deviation $%
(u(t;x,y),v(t;x,y))$: 
\begin{equation*}
\partial _{x}\psi =-v\quad \text{and}\quad \partial _{y}\psi =u.
\end{equation*}%
Let us denote by $\omega (t;x,y)$ the deviation of vorticity. By definition, 
$\omega =-\Delta \psi $. The linearized vorticity equation is 
\begin{equation}
\partial _{t}\omega +(\psi _{ey}\partial _{x}\omega -\psi _{ex}\partial
_{y}\omega )+(\omega _{ex}\partial _{y}\psi -\omega _{ey}\partial _{x}\psi
)=0.  \label{E:vorticityL}
\end{equation}%
Since $\omega _{e}=\gamma (\psi _{e})$, the last term can be written as $%
-\gamma ^{\prime }(\psi _{e})(\psi _{ey}\partial _{x}\psi -\psi
_{ex}\partial _{y}\psi )$.

The linearized kinematic and dynamic boundary conditions restricted to the
steady free boundary $\mathcal{S}_{e}$ are 
\begin{equation}
v+v_{ey}\eta=\partial_{t}\eta+(u_{e}-\underline{c})\partial_{x}%
\eta+(u+u_{ey}\eta)\eta_{ex}  \label{E:kinematicL}
\end{equation}
and 
\begin{equation*}
P+P_{ey}\eta=0,  \label{E:dynamicL}
\end{equation*}
respectively. In terms of the stream function, (\ref{E:kinematicL}) is
further written as 
\begin{align*}
0 & =\partial_{t}\eta+\psi_{ey}\partial_{x}\eta+(\partial_{x}\psi+\eta
_{ex}\partial_{y}\psi)+(\psi_{exy}+\eta_{ex}\psi_{eyy})\eta \\
&
=\partial_{t}\eta+\psi_{ey}\partial_{x}\eta+\partial_{\tau}\psi+\partial_{%
\tau}\psi_{ey}\eta \\
& =\partial_{t}\eta+\partial_{\tau}( \psi_{ey}\eta) +\partial_{\tau}\psi,
\end{align*}
where 
\begin{equation*}
\partial_{\tau}f=\partial_{x}f+\eta_{ex}\partial_{y}f
\end{equation*}
denotes the tangential derivative of a function $f$ defined on the curve $%
\{y=\eta_{e}(x)\}$. Alternatively, $\partial_{\tau}f(x)=\partial_{x}
f(x,\eta_{e}(x))$. The bottom boundary condition of the linearized motion is 
\begin{equation*}
\partial_{x}\psi=0 \qquad \text{on}\quad \{y=0\}.
\end{equation*}

Our next task is to examine the time-evolution of $\psi$ on the steady free
surface $\mathcal{S}_{e}$. This links the tangential derivative of the
pressure deviation $P$ on the steady free surface $\mathcal{S}_{e}$ with $%
\psi$ and $\eta$ on $\mathcal{S}_{e}$. 
%To get a closed system for $\left(  \eta,\psi,P\right)  ,$ we need to relate
%$P$ to other quantities on the steady free surface $\mathcal{S}_{e}.\ $
For a function $f$ defined on $\mathcal{S}_{e}$, let us denote by 
\begin{equation*}
\partial_{n}f=\partial_{y}f-\eta_{ex}\partial_{x}f
\end{equation*}
the normal derivative of $f$ on the curve $\{y=\eta_{e}(x)\}$.

\begin{lemma}
On the steady free surface $\mathcal{S}_{e}$, the normal derivative $%
\partial_n\psi$ satisfies 
\begin{equation*}
\partial_{t}\partial_{n}\psi+\partial_{\tau}(\psi
_{ey}\partial_{n}\psi)+\Omega\partial_{\tau}\psi+\partial_{\tau}P=0,
\end{equation*}
where $\Omega=\omega_{e}(x,\eta_{e}(x))$ is the (constant) value of steady
vorticity on $\mathcal{S}_{e}$.
\end{lemma}

\begin{proof}
The linearized Euler equation (\ref{E:velocityL}) can be rewritten in the
form 
\begin{equation}
-\bigtriangledown P=\partial _{t}\left( 
\begin{array}{c}
u \\ 
v%
\end{array}%
\right) +\bigtriangledown \left( (u_{e}-\underline{c})u+v_{e}v\right)
-\omega _{e}\left( 
\begin{array}{c}
v \\ 
-u%
\end{array}%
\right) -\omega \left( 
\begin{array}{c}
v_{e} \\ 
-(u_{e}-\underline{c})%
\end{array}%
\right) ,  \label{equality-P}
\end{equation}%
by linearizing the nonlinear convection term according to the identity 
\begin{equation*}
\left( u,v\right) \cdot \bigtriangledown \left( 
\begin{array}{c}
u \\ 
v%
\end{array}%
\right) =\bigtriangledown \left( \frac{1}{2}\left( u^{2}+v^{2}\right)
\right) -\omega \left( 
\begin{array}{c}
v \\ 
-u%
\end{array}%
\right) .
\end{equation*}

Taking dot product of (\ref{equality-P}) with the vector $\left( 1,\eta
_{ex}\right) $ and restricting the result to $\mathcal{S}_{e}$, we have 
\begin{eqnarray*}
-\partial _{\tau }P &=&\partial _{t}\left( u+v\eta _{ex}\right) +\partial
_{\tau }\left( (u_{e}-\underline{c})u+v_{e}v\right) -\Omega \left( v-u\eta
_{ex}\right) \\
&=&\partial _{t}\left( \psi _{y}-\eta _{ex}\psi _{x}\right) +\partial _{\tau
}\left( (u_{e}-\underline{c})\psi _{y}-(u_{e}-\underline{c})\eta _{ex}\psi
_{x}\right) -\Omega \left( -\psi _{x}-\eta _{ex}\psi _{y}\right) \\
&=&\partial _{t}\partial _{n}\psi +\partial _{\tau }((u_{e}-\underline{c}%
)\partial _{n}\psi )+\Omega \partial _{\tau }\psi ,
\end{eqnarray*}%
where in the above derivation we use the steady kinematic equation 
\begin{equation*}
v_{e}=(u_{e}-\underline{c})\eta _{ex},\text{ on }\mathcal{S}_{e}\text{. }
\end{equation*}
\end{proof}

In summary, there results in the linearized water-wave problem: 
\begin{subequations}
\label{eqn-L}
\begin{gather}
\partial _{t}\omega +(\psi _{ey}\partial _{x}\omega -\psi _{ex}\partial
_{y}\omega )=\gamma ^{\prime }(\psi _{e})(\psi _{ey}\partial _{x}\psi -\psi
_{ex}\partial _{y}\psi )\qquad \text{in}\quad \mathcal{D}_{e},
\label{eqn-L-vor} \\
\partial _{t}\eta +\partial _{\tau }(\psi _{ey}\eta )+\partial _{\tau }\psi
=0\qquad \text{on}\quad \mathcal{S}_{e};  \label{eqn-L-eta} \\
P+P_{ey}\eta =0\qquad \text{on}\quad \mathcal{S}_{e};  \label{eqn-L-P-eta} \\
\partial _{t}\partial _{n}\psi +\partial _{\tau }(\psi _{ey}\partial
_{n}\psi )+\Omega \partial _{\tau }\psi +\partial _{\tau }P=0\qquad \text{on}%
\quad \mathcal{S}_{e};  \label{eqn-L-phi-P} \\
\partial _{x}\psi =0\qquad \text{on}\quad \{y=0\}.  \label{eqn-L-bottom}
\end{gather}%
Note that the above linearized system may be viewed as one for $\psi (t;x,y)$
and $\eta (t;x)$. Indeed, $P(t;x,\eta _{e}(x))$ is determined through (\ref%
{eqn-L-P-eta}) in terms of $\eta (t;x)$ and other physical quantities are
similarly determined in terms of $\psi (t;x,y)$ and $\eta (t;x)$.

\subsection{The growing-mode problem}

A \emph{growing mode} refers to an exponentially growing solution to the
linearized water-wave problem (\ref{eqn-L}) of the form 
\end{subequations}
\begin{equation*}
(\eta (t;x),\psi (t;x,y))=(e^{\lambda t}\eta (x),e^{\lambda t}\psi (x,y))
\end{equation*}%
and $P(t;x,\eta _{e}(x))=e^{\lambda t}P(x,\eta _{e}(x))$ with ${\rm Re}%
\,\lambda >0$. For such a solution, the linearized vorticity equation (\ref%
{eqn-L-vor}) further reduces to 
\begin{equation}
\lambda \omega +(\psi _{ey}\partial _{x}\omega -\psi _{ex}\partial
_{y}\omega )-\gamma ^{\prime }(\psi _{e})(\psi _{ey}\partial _{x}\psi -\psi
_{ex}\partial _{y}\psi )=0\quad \text{in }\mathcal{D}_{e},
\label{E:vorticityG}
\end{equation}%
where $\omega =-\Delta \psi $. Let $(X_{e}(s;x,y),Y_{e}(s;x,y))$ be the
particle trajectory of the steady flow 
\begin{equation}
\begin{cases}
\dot{X_{e}}=\psi _{ey}(X_{e},Y_{e}) \\ 
\dot{Y_{e}}=-\psi _{ex}(X_{e},Y_{e})%
\end{cases}
\label{E:char}
\end{equation}%
with the initial position $\left( X_{e}(0),Y_{e}(0)\right) =\left(
x,y\right) $. Here, the dot above a variable denotes the differentiation in
the $s$-variable. Integration of (\ref{E:vorticityG}) along the particle
trajectory $(X_{e}(s;x,y),Y_{e}(s,x,y))$ for $s\in (-\infty ,0)$ yields \cite%
[Lemma 3.1]{lin2} 
\begin{subequations}
\label{growing mode}
\begin{equation}
\Delta \psi +\gamma ^{\prime }(\psi _{e})\psi -\gamma ^{\prime }(\psi
_{e})\int_{-\infty }^{0}\lambda e^{\lambda s}\psi
(X_{e}(s),Y_{e}(s))ds=0\quad \text{in $\mathcal{D}_{e}$}.  \label{eqn-G-phi}
\end{equation}

For a growing mode the boundary conditions (\ref{eqn-L-eta}), (\ref%
{eqn-L-P-eta}) and (\ref{eqn-L-phi-P}) on $\mathcal{S}_{e}$ become 
\begin{gather}
\lambda\eta(x)+\frac{d}{dx}\big(\psi_{ey}(x,\eta _{e}(x))\eta(x)\big)=-\frac{%
d}{dx}\psi(x,\eta_{e}(x)),  \label{eqn-G-eta} \\
P(x,\eta_{e}(x))+P_{ey}(x,\eta_{e}(x))\eta (x)=0,  \label{eqn-G-P-eta} \\
\lambda\psi_{n}(x)+\frac{d}{dx}\big(\psi_{ey}(x,\eta _{e}(x))\psi_{n}(x)\big)%
=-\frac{d}{dx}P(x,\eta_{e}(x))-\Omega\frac{d}{dx}\psi(x,\eta_{e}(x)).
\label{eqn-G-phi-n}
\end{gather}

The kinematic boundary condition (\ref{eqn-L-bottom}) at the flat bottom $%
\{y=0\}$ implies that $\psi (x,0)$ is a constant. Observe that (\ref%
{eqn-G-phi})-(\ref{eqn-G-phi-n}) remain unchanged by adding a constant to $%
\psi $. So we can impose the bottom boundary condition by 
\begin{equation}
\psi (x,0)=0.  \label{eqn-G-bottom}
\end{equation}

In summary, the growing-mode problem for periodic traveling water-waves is
to find a nontrivial solution of (\ref{eqn-G-phi})-(\ref{eqn-G-bottom}) with 
${\rm Re}\,\lambda>0$.

\subsection{Invariants of the linearized water-wave problem}

In this subsection, the invariants of the linearized water-wave problem (\ref%
{eqn-L-vor})--(\ref{eqn-L-bottom}) are derived, and their implications in
the stability of water waves are discussed.

With the introduction of the Poisson bracket, defined as 
\end{subequations}
\begin{equation*}
\lbrack
f_{1},f_{2}]=\partial_{x}f_{1}\partial_{y}f_{2}-\partial_{y}f_{1}%
\partial_{x}f_{2},
\end{equation*}
the linearized vorticity equation (\ref{eqn-L-vor}) is further written as 
\begin{equation}
\partial_{t}\omega-[\psi_{e},\omega]+\gamma^{\prime}(\psi_{e})[\psi_{e},%
\psi]=0\quad\text{in }\mathcal{D}_{e},  \label{E:vorticityI}
\end{equation}
where $\omega=-\Delta\psi$. Recall that 
\begin{equation*}
\partial_{\tau}f=\partial_{x}f+\eta_{ex}\partial_{y}f\quad\text{and}\quad
\partial_{n}f=\partial_{y}f-\eta_{ex}\partial_{x}f,
\end{equation*}
where $f$ is a function defined on $\see$.

For the case when the vorticity-stream relation is monotone, we show that
the following energy functional is an invariant of the linearized system.

\begin{lemma}
\label{L:E} Provided that either $\gamma ^{\prime }(p)<0$ or $\gamma
^{\prime }(p)>0$ on $[0,|p_{0}|]$, then for any solution $(\eta ,\psi )$ of
the linearized water-wave problem (\ref{eqn-L}), we have $\frac{d}{dt}%
\mathcal{E}(\eta ,\psi )=0$, where 
\begin{equation}
\begin{split}
\mathcal{E}(\eta ,\psi )=& \iint_{\mathcal{D}_{e}}\tfrac{1}{2}|\nabla \psi
|^{2}dydx-\iint_{\mathcal{D}_{e}}\tfrac{1}{2}\gamma ^{\prime }(\psi
_{e})^{-1}|\omega |^{2}dydx \\
& -\int_{\see}\tfrac{1}{2}P_{ey}|\eta |^{2}dx+{\rm Re}\,\int_{\see}\psi
_{ey}\partial _{n}\psi \eta ^{\ast }dx-\Omega \int_{\see}\tfrac{1}{2}\psi
_{ey}|\eta |^{2}dx.
\end{split}
\label{energy-gamma}
\end{equation}%
In the irrotational setting, i.e. $\gamma \equiv 0$, the invariant
functional becomes 
\begin{equation}
\mathcal{E}(\eta ,\psi )=\iint_{\mathcal{D}_{e}}\tfrac{1}{2}|\nabla \psi
|^{2}dydx-\int_{\see}\tfrac{1}{2}P_{ey}|\eta |^{2}dx+\mathrm{Re}\,\int_{\see%
}\psi _{ey}\partial _{n}\psi \eta ^{\ast }dx.  \label{energy-0}
\end{equation}
\end{lemma}

\begin{proof}
By integration by parts, 
\begin{align*}
\frac{d}{dt}\iint_{\mathcal{D}_{e}}\tfrac{1}{2}|\nabla \psi |^{2}dydx& =%
{\rm Re}\,\iint_{\mathcal{D}_{e}}\partial _{t}(\nabla \psi )\cdot \nabla
\psi ^{\ast }dydx \\
& ={\rm Re}\,\iint_{\mathcal{D}_{e}}\partial _{t}\omega \psi ^{\ast
}dydx+\int_{\see}\partial _{t}(\partial _{n}\psi )\psi ^{\ast
}dx+\int_{\{y=0\}}\partial _{t}\partial _{y}\psi \psi ^{\ast }dx \\
& :=(I)+(II)+(III).
\end{align*}

%\begin{equation}\label{E:(I)}
%(I)=\frac{d}{dt} \iint_{\dee} \tfrac{1}{2} \gamma'(\psi_e)^{-1}|\omega|^2 dydx.\end{equation}
A substitution of $\partial _{t}\omega $ by the linearized vorticity
equation (\ref{E:vorticityI}) yields that 
\begin{align*}
(I)& ={\rm Re}\,\iint_{\mathcal{D}_{e}}\left( -\gamma ^{\prime }(\psi
_{e})[\psi _{e},\psi ]+[\psi _{e},\omega ]\right) \psi ^{\ast }dydx \\
& ={\rm Re}\iint_{\mathcal{D}_{e}}\left( -\tfrac{1}{2}[\psi _{e},\gamma
^{\prime }(\psi _{e})|\psi |^{2}]+[\psi _{e},\omega \psi ^{\ast }]-[\psi
_{e},\psi ^{\ast }]\omega \right) dxdy.
\end{align*}%
The second equality in the above uses that 
\begin{align*}
\tfrac{1}{2}[\psi _{e},& \gamma ^{\prime }(\psi _{e})|\psi |^{2}]=\tfrac{1}{2%
}[\psi _{e},\gamma ^{\prime }(\psi _{e})]|\psi |^{2}+\tfrac{1}{2}\gamma
^{\prime }(\psi _{e})([\psi _{e},\psi ]\psi ^{\ast }+[\psi _{e},\psi ^{\ast
}]\psi ) \\
& =\gamma ^{\prime }(\psi _{e}){\rm Re}[\psi _{e},\psi ]\psi ^{\ast },
\end{align*}%
which follows since $[\psi _{e},g(\psi _{e})]=0$ for any $g$. With the
observation that 
\begin{align*}
\iint_{\mathcal{D}_{e}}[\psi _{e},f]dydx& =\iint_{\mathcal{D}_{e}}\nabla
\cdot ((\psi _{ey},-\psi _{ex})f)dydx \\
& =\int_{\see}(\psi _{ey},-\psi _{ex})\cdot (-\eta
_{ex},1)fdx+\int_{y=0}f\psi _{ex}dx=0
\end{align*}%
for any function $f$, the integral $(I)$ further reduces to 
\begin{equation*}
(I)={\rm Re}\,\iint_{\mathcal{D}_{e}}-[\psi _{e},\psi ]^{\ast }\omega dydx.
\end{equation*}%
A simple substitution of $[\psi _{e},\psi ]$ by the linearized vorticity
equation (\ref{E:vorticityI}) then yields that 
\begin{align}
(I)& ={\rm Re}\iint_{\mathcal{D}_{e}}\gamma ^{\prime }(\psi
_{e})^{-1}(\partial _{t}\omega ^{\ast }-[\psi _{e},\omega ]^{\ast })\omega
dydx  \notag \\
& ={\rm Re}\iint_{\mathcal{D}_{e}}(\tfrac{1}{2}\gamma ^{\prime }(\psi
_{e})^{-1}\partial _{t}|\omega |^{2}-\tfrac{1}{2}[\psi _{e},\gamma ^{\prime
}(\psi _{e})^{-1}|\omega |^{2}])dydx  \label{E:E1} \\
& =\frac{d}{dt}\iint_{\mathcal{D}_{e}}\tfrac{1}{2}\gamma ^{\prime }(\psi
_{e})^{-1}|\omega |^{2}dydx.  \notag
\end{align}%
This uses that 
\begin{equation*}
\tfrac{1}{2}[\psi _{e},\gamma ^{\prime }(\psi _{e})^{-1}|\omega
|^{2}]=\gamma ^{\prime }(\psi _{e})^{-1}{\rm Re}[\psi _{e},\omega ]^{\ast
}\omega .
\end{equation*}

Next is to examine the surface integral $(II)$. Simple substitutions by the
linearized boundary conditions (\ref{eqn-L-eta}), (\ref{eqn-L-P-eta}) and (%
\ref{eqn-L-phi-P}) into $(II)$ and an integration by parts yield that 
\begin{align*}
(II) & ={\rm Re}\int_{\see}\partial_{\tau}(P_{ey}\eta-\psi
_{ey}\partial_{n}\psi-\Omega\psi)\psi^{\ast}dx \\
&={\rm Re}\,\int _{\see}(P_{ey}\eta-\psi_{ey}\partial_{n}\psi)(-\partial_{%
\tau}\psi^{\ast})dx \\
& ={\rm Re}\,\int_{\see}(P_{ey}\eta-\psi_{ey}\partial_{n}\psi)(\partial_{t}%
\eta^{\ast}+\partial_{\tau}(\psi_{ey}\eta^{\ast}))dx \\
& =\frac{d}{dt}\int_{\see}\tfrac{1}{2}P_{ey}|\eta|^{2}dx-{\rm Re}\,\int_{%
\see}\psi_{ey}\partial_{n}\psi\partial_{t}\eta^{\ast}dx+{\rm Re}\,\int_{\see%
}\partial_{\tau}(P+\psi_{ey}\partial_{n}\psi)\psi_{ey}\eta^{\ast}dx.
\end{align*}
The second equality uses that 
\begin{equation*}
{\rm Re}\int_{\see}(\partial_{\tau}\psi)\psi^{\ast}dx=\frac{1}{2}\int_{\see}%
\frac{d}{dx}\left\vert \psi(x,\eta_{e}(x))\right\vert ^{2}dx=0.
\end{equation*}
More generally, ${\rm Re}\int_{\see}(\partial_{\tau}f)f^{\ast}dx=0$ for any
function $f$ defined on $\see$. With another simple substitution by the
boundary condition (\ref{eqn-L-phi-P}), the last term in the computation of $%
(II)$ is written as 
\begin{align*}
{\rm Re}\,\int_{\see}\partial_{\tau} &
(P+\psi_{ey}\partial_{n}\psi)\psi_{ey}\eta^{\ast}dx \\
& =-{\rm Re}\,\int_{\see}(\partial_{t}\partial_{n}\psi
+\Omega\partial_{\tau}\psi)\psi_{ey}\eta^{\ast}dx \\
& =-{\rm Re}\,\int_{\see}(\partial_{t}(\partial_{n}\psi)\psi
_{ey}\eta^{\ast}dx+\Omega {\rm Re}\,\int_{\see}(\partial_{t}\eta+\partial_{%
\tau}(\psi_{ey}\eta))\psi_{ey}\eta^{\ast})dx \\
& =-{\rm Re}\,\int_{\see}\psi_{ey}\partial_{t}(\partial_{n}\psi
)\eta^{\ast}dx+\Omega\frac{d}{dt}\int_{\see}\tfrac{1}{2}\psi_{ey}|%
\eta|^{2}dx.
\end{align*}
The last equality uses that ${\rm Re}\int_{\see}\partial_{\tau}(\psi_{ey}%
\eta)(\psi_{ey}\eta)^{\ast}dx=0$. Therefore, 
\begin{equation}
(II)=\frac{d}{dt}\int_{\see}\tfrac{1}{2}P_{ey}|\eta|^{2}dx-{\rm Re}\,\frac{d%
}{dt}\int_{\see}\psi_{ey}\partial_{n}\psi\eta^{\ast}dx+\Omega\frac {d}{dt}%
\int_{\see}\tfrac{1}{2}\psi_{ey}|\eta|^{2}dx.  \label{E:E2}
\end{equation}

Finally, it is straightforward to see that 
\begin{equation*}
(III)=\psi ^{\ast }(x,0)\int_{\{y=0\}}\partial _{t}(\partial _{y}\psi )dx=0.
\end{equation*}%
This, together with (\ref{E:E1}) and (\ref{E:E2}) proves that $\mathcal{E}%
(\eta ,\psi )$ is an invariant. In the irrotational setting, i.e. $\gamma =0$%
, the area integral $(I)$ is zero, $\Omega =0$, and the other terms remain
the same. This completes the proof.
\end{proof}

\begin{remark}
\label{remark-E-invariant} Our energy functional $\mathcal{E}$ agrees with
the second variation $\partial ^{2}\mathcal{H}$ of the energy-Casimir
functional in \cite{cost1}. Recall that the hydrostatic pressure of the
steady solution is given as in (\ref{eqn-steady-pressure}) by 
\begin{equation*}
P_{e}(x,y)=B-\tfrac{1}{2}|\nabla \psi _{e}(x,y)|^{2}-gy+\int_{0}^{\psi
_{e}(x,y)}\gamma (-p)dp,
\end{equation*}%
where $B$ is the Bernoulli constant. Differentiation of the above and
restriction on $\mathcal{S}_{e}$ then yield that 
\begin{equation*}
-P_{ey}-\Omega \psi _{ey}=\tfrac{1}{2}\partial _{y}|\nabla \psi _{e}|^{2}+g,
\end{equation*}%
where $\Omega =\gamma (\psi =0)=\omega _{e}(x,\eta _{e}(x))$. Thus, when $%
\gamma \ $is monotone, it follows that 
\begin{multline*}
2\mathcal{E}(\eta ,\psi )=\iint_{\mathcal{D}_{e}}|\nabla \psi
|^{2}dydx-\iint_{\mathcal{D}_{e}}\gamma ^{\prime }(\psi _{e})^{-1}|\omega
|^{2}dydx \\
+\int_{\mathcal{S}_{e}}\left( \tfrac{1}{2}\partial _{y}|\nabla \psi
_{e}|^{2}+g\right) |\eta |^{2}dx+2{\rm Re}\,\int_{\mathcal{S}_{e}}\psi
_{ey}\partial _{n}\psi \eta ^{\ast }dx.
\end{multline*}%
This is exactly the expression of the second variation $\partial ^{2}%
\mathcal{H}$ in \cite{cost1} of the following enery-Casimir functional (in
our notations) 
\begin{equation}
\mathcal{H}(\eta ,\psi )=\iint_{\mathcal{D}_{\eta }}\left( \frac{|\nabla
(\psi -\underline{c}y)|^{2}}{2}+gy-B-F(\omega )\right) dydx,  \label{ec-H}
\end{equation}%
around the steady state $(\eta _{e},\psi _{e})$. Here, $\mathcal{D}_{\eta
}=\{(x,y):0<y<\eta (t;x)\}$ and $(F^{\prime })^{-1}=\gamma $. The quadratic
form $\partial ^{2}\mathcal{H}$ is used in \cite{cost1} to study a formal
stability.
\end{remark}

Our next invariant is the linearized horizontal momentum. The result is free
from restrictions on $\gamma$.

\begin{lemma}
For any solution $(\eta,\psi)$ of the linearized problem of \textrm{(\ref%
{eqn-L})}, the identity 
\begin{equation*}
\frac{d}{dt}\int_{\mathcal{S}_{e}}(\psi+\psi_{ey}\eta)dx=0
\end{equation*}
holds true.
\end{lemma}

\begin{proof}
We integrate over the steady fluid region $\mathcal{D}_{e}$ of the
linearized equation for the horizontal velocity 
\begin{equation*}
\partial _{t}\partial _{y}\psi +(\psi _{ey},-\psi _{ex})\cdot \nabla
(\partial _{y}\psi )+(\partial _{y}\psi ,-\partial _{x}\psi )\cdot \nabla
\psi _{ey}=-P_{x}
\end{equation*}%
and apply the divergence theorem to arrive at 
\begin{equation}
\frac{d}{dt}\iint_{\mathcal{D}_{e}}\partial _{y}\psi \,dydx+\int_{\mathcal{S}%
_{e}}(\partial _{y}\psi ,-\partial _{x}\psi )\cdot (-\eta _{ex},1)\psi
_{ey}dx=-\iint_{\mathcal{D}_{e}}P_{x}dydx.  \label{E:V}
\end{equation}%
It is straightforward to see that 
\begin{equation*}
\iint_{\mathcal{D}_{e}}\partial _{y}\psi dydx=\int_{\mathcal{S}_{e}}\psi
\,dx.
\end{equation*}%
In view of (\ref{eqn-L-eta}), the second term on the left hand side of (\ref%
{E:V}) is written as 
\begin{align*}
\int_{\mathcal{S}_{e}}(\partial _{\tau }\psi )\psi _{ey}dx& =\int_{\mathcal{S%
}_{e}}(\partial _{t}\eta +\partial _{\tau }(\psi _{ey}\eta ))\psi _{ey}dx \\
& =\frac{d}{dt}\int_{\mathcal{S}_{e}}\psi _{ey}\eta \,dx-\int_{\mathcal{S}%
_{e}}\psi _{ey}\partial _{\tau }(\psi _{ey})\eta dx.
\end{align*}%
With the use of Stokes' theorem and the dynamic boundary condition (\ref%
{E:dynamicL}), the right side of (\ref{E:V}) becomes 
\begin{equation*}
\iint_{\mathcal{D}_{e}}P_{x}dydx=\int_{\mathcal{S}_{e}}P\eta _{ex}dx=\int_{%
\mathcal{S}_{e}}P_{ey}\eta \eta _{ex}dx.
\end{equation*}%
On the other hand, the steady Euler equation restricted to the steady
free-surface $\mathcal{S}_{e}$ yields that 
\begin{equation*}
-P_{ex}=\psi _{ey}\psi _{exy}-\psi _{ex}\psi _{eyy}=\psi _{ey}(\psi
_{exy}+\eta _{ex}\psi _{eyy})=\psi _{ey}\partial _{\tau }\psi _{ey}
\end{equation*}%
Since $P_{ex}+P_{ey}\eta _{ex}=0$ on $\see$, a simple substitution then
proves the assertion.
\end{proof}

Next, the integration of (\ref{eqn-L-eta}) on $\mathcal{S}_{e}$ shows that $%
\int_{\mathcal{S}_{e}}\eta \,dx$ is an invariant. Finally, multiplication on
the linearized vorticity equation (\ref{E:vorticityI}) by $\xi $ and then
integration yield that 
\begin{equation*}
\iint_{\mathcal{D}_{e}}\omega \xi dydx
\end{equation*}%
is an invariant, for any function 
\begin{equation*}
\xi \in \ker (\psi _{ey}\partial _{x}-\psi _{ex}\partial _{y})\subset L^{2}(%
\mathcal{D}_{e}).
\end{equation*}

We summarize our results.

\begin{proposition}
\label{P:invariants} The linearized problem \textrm{{(\ref{eqn-L})} has the
energy invariant: 
\begin{align*}
\mathcal{E}(\eta ,\psi )=& \iint_{\mathcal{D}_{e}}\tfrac{1}{2}|\nabla \psi
|^{2}dydx-\iint_{\mathcal{D}_{e}}\tfrac{1}{2}\gamma ^{\prime }(\psi
_{e})^{-1}|\omega |^{2}dydx \\
& +\int_{\mathcal{S}_{e}}\tfrac{1}{2}\left( \partial _{y}(\tfrac{1}{2}%
|\nabla \psi _{e}|^{2})+g\right) |\eta |^{2}dx+{\rm Re}\,\int_{\mathcal{S}%
_{e}}\psi _{ey}\partial _{n}\psi \eta ^{\ast }dx
\end{align*}%
when $\gamma $ is monotone and 
\begin{align*}
\mathcal{E}(\eta ,\psi )& =\iint_{\mathcal{D}_{e}}\tfrac{1}{2}|\nabla \psi
|^{2}dydx \\
& +\int_{\see}\tfrac{1}{2}\left( \partial _{y}(\tfrac{1}{2}|\nabla \psi
_{e}|^{2})+g\right) |\eta |^{2}dx+\mathrm{Re}\,\int_{\see}\psi _{ey}\partial
_{n}\psi \eta ^{\ast }dx
\end{align*}%
when $\gamma \equiv 0$ (irrotational). In addition, {(\ref{eqn-L})} has the
following invariants: 
\begin{align*}
\mathcal{M}(\eta ,\psi )& =\int_{\mathcal{S}_{e}}(\psi +\psi _{ey}\eta )dx,
\\
\mathnormal{m}(\eta ,\psi )& =\int_{\mathcal{S}_{e}}\eta dx \\
\mathcal{F}(\eta ,\psi )& =\iint_{\mathcal{D}_{e}}\omega \xi dydx\mathrm{,\ }%
\text{for any }\mathrm{\xi \in \ker (\psi _{ey}\partial _{x}-\psi
_{ex}\partial _{y}).\ }
\end{align*}%
}
\end{proposition}

The nonlinear water-wave problem (\ref{E:Euler1})--(\ref{E:bottom}) has the
following invariants \cite[Section 2]{cost1}: 
\begin{alignat*}{2}
\mathfrak{E}& =\iint_{\mathcal{D}(t)}\left( \tfrac{1}{2}(u^{2}+v^{2})+gy%
\right) dydx\quad & & \text{(energy),} \\
\mathfrak{M}& =\iint_{\mathcal{D}(t)}u\,dydx\quad & & \text{(horizontal
momentum),} \\
\mathfrak{m}& =\iint_{\mathcal{D}(t)}dydx\quad & & \text{(mass),} \\
\mathfrak{F}& =\iint_{\mathcal{D}(t)}f(\omega )dydx & & \text{(Casimir
invariant)},
\end{alignat*}%
where 
\begin{equation*}
\mathcal{D}(t)=\{(x,y):0<x<2\pi /\alpha \,,\,0<y<\eta (t;x)\}
\end{equation*}%
is the fluid domain at time $t$ of one wave length, and the function $f$ is
arbitrary such that the integral $\mathfrak{F}$ exists.

The invariants of the linearized problem $\mathcal{E}$, $\mathcal{M}$, $%
\mathnormal{m}$, and $\mathcal{F}$ in Proposition \ref{P:invariants} can be
obtained by expanding the invariants of the nonlinear problem $\mathfrak{E}$%
, $\mathfrak{M}$, $\mathfrak{m}$ and $\mathfrak{F}$, respectively, around
the steady state $(\eta _{e}(x),\psi _{e}(x,y))$. The quadratic form $%
\mathcal{E}(\eta ,\psi )$ is the second variation of the energy functional $%
\mathcal{H}$ defined in (\ref{ec-H}) (see Remark \ref{remark-E-invariant}),
which is a combination of $\mathfrak{E}$, $\mathfrak{M}$, $\mathfrak{m}$ and 
$\mathfrak{F}$. The invariants $\mathcal{M}$ and $\mathnormal{m}$ of the
linear problem are the first variations of $\mathfrak{M}$ and $\mathfrak{m}$%
, respectively. We note that $\mathcal{F}$ is the first variation of $%
\mathfrak{F}$, since by the assumptions of no stagnation and the
monotonicity of $\gamma $ it follows that 
\begin{equation*}
\ker (\psi _{ey}\partial _{x}-\psi _{ex}\partial _{y})=\{f(\psi _{e}):\text{$%
f$ is arbitrary}\}=\{f(\omega _{e}):\text{$f$ is arbitrary}\}.
\end{equation*}

We now make some comments on the implications of these invariants on the
stability of water waves. A traditional approach of proving stability for
conservative systems is the energy method, for which one tries to show that
a steady state is an energy minimizer under the constraints of other
invariants such as momentum, mass, etc. This method has been widely used in
the stability analysis of various approximate equations such as the KdV
equation \cite{anjulo-bona}, \cite{benjamin72} and water waves with a large
surface tension \cite{meilke}. However, steady waves of the full pure
gravity water-wave problem in general are \emph{not} (constrainted) energy
minimizers. This was first observed (\cite{bona-sachs}) in the irrotational
case. By our discussions below, even with the favorable vorticity, the
situation remains the same. Note that if a steady gravity water-wave is an
energy minimizer under the constraints of fixed $\mathfrak{M}$, $\mathfrak{m}
$ and $\mathfrak{F}$, then at the linearized level the second variation $%
\mathcal{E}$ should be positive under the constraints that the variations $%
\mathcal{M}$, $\mathnormal{m}$ and $\mathcal{F}$ are zero.

In the irrotational setting, the first two terms in the expression (\ref%
{energy-0}) of $\mathcal{E}(\eta ,\psi )$ yield a positive norm 
\begin{equation}
\iint_{\mathcal{D}_{e}}|\nabla \psi |^{2}dydx+\int_{\mathcal{S}_{e}}|\eta
|^{2}dx.  \label{positive-irro}
\end{equation}%
(Indeed, since $\Delta P_{e}=-2\psi _{exy}^{2}-\psi _{exx}^{2}-\psi
_{eyy}^{2}\leq 0$ and $P_{ey}(x,0)=-g$ by the maximum principle $P_{e}$
attains its minimum at the free surface $\mathcal{S}_{e}$. Furthermore,
since $P_{e}$ takes a constant on $\mathcal{S}_{e}$ by the Hopf lemma $%
P_{ey}<0$ on $\mathcal{S}_{e}$.) The last term $\mathrm{Re}\,\int_{\mathcal{S%
}_{e}}\psi _{ey}\partial _{n}\psi \eta ^{\ast }dx$ of the right side of (\ref%
{energy-0}), however, does not have a definite sign and contains a $1/2$%
-higher derivative than that of (\ref{positive-irro}). Thus, it cannot be
bounded by the norm (\ref{positive-irro}), even with the constraints $%
\mathcal{M}=\mathnormal{m}=0$. Consequently, the quadratic form $\mathcal{E}%
(\eta ,\psi )$ is highly indefinite unless $\psi _{ey}\equiv 0$, that is,
for a trivial flow. In other words, steady gravity water waves in the
irrotational case are in general not (constrainted) energy minimizers.
Indeed, in \cite{garabedian}, \cite{buff-tol-saddle}, steady water-waves
were constructed as energy saddles by variational methods.

With a nonzero vorticity, a control of the mixed-type term in the energy
functional by other positive terms fails for the same reason as in the
irrotational setting. Let us consider the case of a favorable vorticity with 
$\gamma ^{\prime }<0$. Under some additional assumptions, it is shown in 
\cite{cost1} that the first three terms in the expression (\ref{energy-gamma}%
) of $\mathcal{E}(\eta ,\psi )$ gives a positive norm 
\begin{equation}
\iint_{\mathcal{D}_{e}}(|\nabla \psi |^{2}+|\omega |^{2})dydx+\int_{\mathcal{%
S}_{e}}|\eta |^{2}dx.  \label{norm-positive}
\end{equation}%
However, due to the lack of control of the boundary value of $\psi $ on $%
\mathcal{S}_{e}$, one could not use the elliptic regularity to get a better
control for $\psi ,$such as $\left\Vert \psi \right\Vert _{H^{2}}$. Thus,
the mixed-type term ${\rm Re}\,\int_{\mathcal{S}_{e}}\psi _{ey}\partial
_{n}\psi \eta ^{\ast }dx$ is still not controllable by (\ref{norm-positive}%
). In \cite{cost1}, several class of perturbations are introduced to make
this mixed term controllable by (\ref{norm-positive}), and therefore get the
positivity of $\mathcal{E}\left( \eta ,\psi \right) $ and a formal stability
for these special perturbations. However, it is difficult to see that these
special classes are invariant during the evolution of the water-wave
problem. So, it remains unclear how to pass from such formal stability to
the genuine stability, even for these special perturbations. It is not hard
to see that by adding other invariants as constraints, one can relax the
assumptions to get the positive term (\ref{norm-positive}) such as in \cite%
{lin2} for the fixed boundary case, but the mixed-type term remains
uncontrollable. In conclusion, the quadratic form $\mathcal{E}(\eta ,\psi )$
is in general indefinite also in the rotational setting, and steady water
waves with vorticity are expected to be energy saddles.

The above discussions of steady water waves as energy saddles do not imply
that steady water-waves are necessarily unstable. Indeed, as mentioned in
the Introduction, the small Stokes waves are believed to be stable \cite%
{longet-78-super}, \cite{tanaka83} under perturbations of the same period.
For the rotational case, under the assumption of a monotone $\gamma $, the
corresponding trivial solutions with shear flows defined in Lemma \ref%
{L:trivial} can not have an inflection point, since $\gamma ^{\prime }(\psi
_{e}(y))=-U^{\prime \prime }(y)/U(y)\neq 0$ and $U<0$ by the no-stagnation
assumption. Thus by Theorem \ref{T:stable}, such shear flows are linearly
stable to perturbations of any period. Since small-amplitude waves with any
monotone vorticity relation $\gamma $ bifurcate from these strongly stable
shear flows, they are likely to be stable. But, a successful stability
analysis of any nontrivial steady gravity water waves would require to use
the full set of equations instead of a few invariants.

\section{Linear instability of shear flows with free surface}

\label{S:formulation}This section is devoted to the study of the linear
instability of a free-surface shear flow $(U(y),0)$ with $y\in \lbrack 0,h]$%
, which is a steady solution of the water-wave problem (\ref{E:Euler1})--(%
\ref{E:bottom}) with $P(x,y)=-gy$. In the traveling frame of reference with
the speed $\underline{c}>\max U$, this may be recognized as a trivial
solution of the traveling-wave problem (\ref{stream}): 
\begin{equation*}
\eta _{e}(x)\equiv h\quad \text{and}\quad (\psi _{ey}(x,y),-\psi
_{ex}(x,y))=(U(y)-\underline{c},0).
\end{equation*}%
Throughout this section, we write $U$ for $U-\underline{c}$ to simplify our
notations.

We seek for normal mode solutions of the form 
\begin{equation*}
\eta(t;x)=\eta_{h}e^{i\alpha(x-ct)}, \qquad
\psi(t;x,y)=\phi(y)e^{i\alpha(x-ct)}
\end{equation*}
and $P(t;x,y)=P(y)e^{i\alpha(x-ct)}$ to the linearized water-wave problem (%
\ref{eqn-L}). Here, $\alpha>0$ is the wave number and $c$ is the complex
phase speed. It is equivalent to find solutions to the growing-mode problem (%
\ref{growing mode}) of the form $\lambda=-i\alpha c$ and 
\begin{equation*}
\eta(x)=\eta_{h}e^{i\alpha x}, \qquad \psi(x,y)=\phi(y)e^{i\alpha x}
\end{equation*}
and $P(x, \eta_{e}(x))=P_{h}e^{i\alpha x}$. Note that ${\rm Re}\lambda>0$
if and only if ${\rm Im}c>0$.

Since $\gamma ^{\prime }(\psi _{e}(y))=\omega _{ey}(y)/\psi
_{ey}(y)=-U^{\prime \prime }(y)/U(y)$ and $(X_{e}(s),Y_{e}(s))=(x+U(y)s,y)$,
the linearized vorticity equation (\ref{eqn-G-phi}) translates into the
Rayleigh equation 
\begin{equation}
(U-c)(\phi ^{\prime \prime }-\alpha ^{2}\phi )-U^{\prime \prime }\phi =0\ \ 
\text{for}\quad y\in (0,h).  \label{rayleigh}
\end{equation}%
Here and elsewhere the prime denote the differentiation in the $y$-variable.
The boundary conditions (\ref{eqn-G-eta}), (\ref{eqn-G-P-eta}) and (\ref%
{eqn-G-phi-n}) on the free surface are simplified to be 
\begin{equation*}
(c-U(h))\eta _{h}=\phi (h),\qquad P_{h}-g\eta _{h}=0
\end{equation*}%
and 
\begin{equation*}
(c-U(h))\phi ^{\prime }(h)=P_{h}+U^{\prime }(h)\phi (h),
\end{equation*}%
respectively. Eliminating $\eta _{h}$ and $P_{h}$ from the above yields that 
\begin{equation*}
(U(h)-c)^{2}\phi ^{\prime }(h)=\left( g+U^{\prime }(h)(U(h)-c)\right) \phi
(h).
\end{equation*}%
The bottom boundary condition (\ref{eqn-G-bottom}) becomes $\phi (0)=0.$ So
the linear instability is reduced to study the Rayleigh equation (\ref%
{rayleigh}) with the boundary condition 
\begin{equation}
\begin{cases}
(U(h)-c)^{2}\phi ^{\prime }(h)=\left( g+U^{\prime }(h)(U(h)-c)\right) \phi
(h) \\ 
\phi (0)=0.%
\end{cases}
\label{bc-rayleigh}
\end{equation}%
A shear profile $U$ is said to be \emph{linearly unstable} if there exists a
nontrivial solution of (\ref{rayleigh})-(\ref{bc-rayleigh}) with $\mathrm{Im}%
\,c>0$.

The Rayleigh system (\ref{rayleigh})--(\ref{bc-rayleigh}) is alternatively
derived in \cite{yih72} by linearizing directly the water-wave problem (\ref%
{E:Euler1})--(\ref{E:bottom}) around $(U(y),0)$ in $y\in \lbrack 0,h]$. Note
that for a real number $c>\max U$ the Rayleigh system (\ref{rayleigh})--(\ref%
{bc-rayleigh}) becomes the bifurcation equation (\ref{rayleigh-disper})--(%
\ref{bc-disper}).

In case of rigid walls at $y=h$ and $y=0$, that is, the Dirichlet boundary
conditions $\phi (h)=0=\phi (0)$ in place of (\ref{bc-rayleigh}), the
instability of a shear flow is a classical problem, which has been under
extensive research since Lord Rayleigh \cite{ray}. Recently, by a novel
analysis of neutral modes, Lin \cite{lin1} established a sharp criterion for
linear instability in the rigid-wall setting for a general class of shear
flows. Our objective in this section is to obtain analogous results in the
free-surface setting.

Below is the definition of the class of shear flows studied in this section.
By an inflection value we mean the value of $U$ at an inflection point.

\begin{definition}
\label{classK} A function $U\in C^{2}([0,h])$ is said to be in the class $%
\mathcal{K}^{+}$ if $U$ has an unique inflection value $U_{s}$ and 
\begin{equation}
K(y)=-\frac{U^{\prime \prime }(y)}{U(y)-U_{s}}  \label{E:K}
\end{equation}%
is bounded and positive on $[0,h]$. 
%Here $U_{s}$ is the unique inflection value of $U\left(  y\right)  $.
\end{definition}

An example of class $\mathcal{K}^{+}$ flows is $U(y)=\sin my$, for which $%
K(y)=m^{2}$.%A typical example of a flow is $U=\cos\beta y$. \

For $U\left( y\right) $ in the class $\mathcal{K}^{+}$, consider the
Sturm-Liouville equation 
\begin{equation}
\phi ^{\prime \prime }-\alpha ^{2}\phi +K(y)\phi =0\qquad \text{for}\quad
y\in (0,h)  \label{E:SL}
\end{equation}%
with the boundary conditions 
\begin{gather}
\begin{cases}
\label{bc-sturm1}\phi ^{\prime }(h)=g_{r}(U_{s})\phi (h)\quad & \text{if}%
\quad U(h)\neq U_{s} \\ 
\phi (h)=0\quad & \text{if}\quad U(h)=U_{s},%
\end{cases}
\\
\phi (0)=0.  \label{bc-sturm2}
\end{gather}%
%
%
%
%
%
%
%
%
%
%
%
%
%\left\{\begin{array}[c]{cc}%
%\phi\left(  0\right)  =0,\ \phi^{\prime}\left(  d\right)  =f_{s}\phi\left(
%d\right)  , & \text{if }U\left(  d\right)  \neq U_{s}\\
%\phi\left(  0\right)  =0,\phi\left(  d\right)  =0, & \text{if }U\left(
%d\right)  =U_{s}%
%\end{array}
%\label{bc-sturm}%\end{equation}
Here, 
\begin{equation}
g_{r}(c)=\frac{g}{(U(h)-c)^{2}}+\frac{U^{\prime }(h)}{U(h)-c}.
\label{E:defn-fc}
\end{equation}

%A solution triple $(\phi_{s},\alpha_{s}, U_s)$ of the Sturm-Liouville problem
%(\ref{E:SL}--(\ref{bc-sturm2} is a ``neutral mode'' of (\ref{rayleigh}--(\ref{bc-rayleigh}.
%Let $\alpha_{\max}$ be the largest wave number for which a neutral mode exists.
%In what follows, we assume that $-\alpha_{\max}^2$ exists and is negative.
%Otherwise, the shear flow $(U(y),0)$ is linearly stable with the prescribed
%boundary conditions at the free surface.

The following theorem gives a sharp instability criterion for shear flows in
class $\mathcal{K}^{+}$, for the free surface case. 
%if $-\alpha_{\max}^2$ exists and negative then for each $\alpha \in (0,\alpha_{\max})$
%corresponds an unstable solution of (\ref{rayleigh}--(\ref{bc-rayleigh}.

\begin{theorem}[Linear instability of free-surface shear flows in $\mathcal{K%
}^{+}$]
\label{class-k}Consider a flow $U\left( y\right) $ in class \ $\mathcal{K}%
^{+}.$ Denote by $-\alpha _{\max }^{2}$ the lowest eigenvalue of $-\frac{%
d^{2}}{dy^{2}}-K\left( y\right) \ $with the boundary condition (\ref%
{bc-sturm1})--(\ref{bc-sturm2}), which is assumed to be negative. Then for
each $\alpha \in \left( 0,\alpha _{\max }\right) ,$ there exists an unstable
solution-triple $(\phi ,\alpha ,c)$ (with $\mathrm{Im}\,c>0$) of (\ref%
{rayleigh})--(\ref{bc-rayleigh}). The interval of unstable wave numbers $%
(0,\alpha _{\max })$ is maximal in the sense that the flow is linearly
stable if either the operator $-\frac{d^{2}}{dy^{2}}-K(y)$ on $y\in (0,h)$
with (\ref{bc-sturm1})--(\ref{bc-sturm2}) is nonnegative or $\alpha \geq
\alpha _{\max }$.
\end{theorem}

>From the usual variational consideration, the lowest eigenvalue $-\alpha _{%
\mathrm{max}}^{2}$ is characterized as 
\begin{subequations}
\label{E:amax}
\begin{equation}
-\alpha _{\max }^{2}=\inf_{\substack{ \phi \in H^{1}(0,h)  \\ \phi (0)=0}}%
\frac{\int_{0}^{h}(\left\vert \phi ^{\prime }\right\vert ^{2}-K(y)|\phi
|^{2})dy+g_{r}(U_{s})|\phi (h)|^{2}}{\int_{0}^{h}|\phi |^{2}dy}
\label{E:amax1}
\end{equation}%
in case $U(h)\neq U_{s}$, and 
\begin{equation}
-\alpha _{\max }^{2}=\inf_{\substack{ \phi \in H^{1}(0,h)  \\ \phi
(0)=0=\phi (h)}}\frac{\int_{0}^{h}(\left\vert \phi ^{\prime }\right\vert
^{2}-K(y)|\phi |^{2})dy}{\int_{0}^{h}|\phi |^{2}dy}  \label{E:amax2}
\end{equation}%
in case $U(h)=U_{s}$.

\subsection{Neutral limiting modes}

The proof of Theorem \ref{class-k} makes use of \emph{neutral limiting modes}%
, as in the rigid-wall setting~\cite{lin1}.

\begin{definition}[Neutral limiting modes]
A triple $(\phi _{s},\alpha _{s},c_{s})$ with $\alpha _{s}$ positive and $%
c_{s}$ real is called a neutral limiting mode if it is the limit of a
sequence of unstable solutions $\{(\phi _{k},\alpha
_{k},c_{k})\}_{k=1}^{\infty }$ of (\ref{rayleigh})--(\ref{bc-rayleigh}) as $%
k\rightarrow \infty $. The convergence of $\phi _{k}$ to $\phi _{s}$ is in
the almost-everywhere sense. For a neutral limiting mode, $\alpha _{s}$ is
called the neutral limiting wave number and $c_{s}$ is called the neutral
limiting wave speed.
\end{definition}

Lemma \ref{basis} below establish that neutral limiting wave numbers form
the boundary points of the interval of unstable wave numbers, and thus the
stability investigation of a shear flow is reduced to find all neutral
limiting modes and then study the stability properties near them. In
general, it is difficult to locate all neutral limiting modes. For flows in
class $\mathcal{K}^{+}$, nonetheless, neutral limiting modes are
characterized by the inflection value.

\begin{proposition}
\label{prop-neutral-mode}For $U\in \mathcal{K}^{+}$ a neutral limiting mode $%
(\phi _{s},\alpha _{s},c_{s})$ must solve (\ref{E:SL})--(\ref{bc-sturm2})
with $c_{s}=U_{s}$.
\end{proposition}

For the proof of Proposition \ref{prop-neutral-mode}, we need several
properties of unstable solutions. First, Howard's semicircle theorem holds
true in the free-surface setting \cite[Theorem 1]{yih72}. That is, any
unstable eigenvalue $c=c_{r}+ic_{i}$ $\left( c_{i}>0\right) $ of the
Rayleigh equation (\ref{rayleigh})-(\ref{bc-rayleigh}) must lie in the
semicircle 
\end{subequations}
\begin{equation}
(c_{r}-\tfrac{1}{2}(U_{\min }+U_{\max }))^{2}+c_{i}^{2}\leq (\tfrac{1}{2}%
(U_{\min }-U_{\max }))^{2},  \label{semicircle}
\end{equation}%
where $U_{\min }=\min_{[0,h]}U(y)$ and $U_{\max }=\max_{[0,h]}U(y)$.

The identities below are important for future considerations.

\begin{lemma}
\label{rayleigh-eqn} If $\phi $ is a solution (\ref{rayleigh})--(\ref%
{bc-rayleigh}) with $c=c_{r}+ic_{i}$ and $c_{i}\neq 0$ 
%Let us define \begin{equation}
%J_{q}(\phi) =\int_{0}^h\left(| \phi'|^{2}+\alpha^{2}|\phi|^{2}+
%\frac{U''(U-q)}{|U-c|^{2}}|\phi|^{2}\right)  dy.
%\end{equation}
then for any $q$ real the identities 
\begin{align}
\int_{0}^{h}\left( \left\vert \phi ^{\prime }\right\vert ^{2}+\alpha
^{2}|\phi |^{2}+\frac{U^{\prime \prime }(U-q)}{|U-c|^{2}}|\phi |^{2}\right)
dy& =\left( {\rm Re}g_{r}(c)+\frac{c_{r}-q}{c_{i}}{\rm Im}g_{r}(c)\right)
|\phi (h)|^{2},  \label{estimate-Jq-f} \\
\int_{0}^{h}\left( \left\vert \phi ^{\prime }\right\vert ^{2}+\alpha
^{2}|\phi |^{2}+\frac{U^{\prime \prime }(U-q)}{|U-c|^{2}}|\phi |^{2}\right)
dy& =\left( {\rm Re}g_{s}(c)-\frac{c_{r}-q}{c_{i}}{\rm Im}g_{s}(c)\right)
\left\vert \phi ^{\prime }\left( h\right) \right\vert ^{2}
\label{estimate-Jq-g}
\end{align}%
hold true, where $g_{r}$ is defined in (\ref{E:defn-fc}) and 
\begin{equation}
g_{s}(c)=\frac{(U(h)-c)^{2}}{g+U^{\prime }(h)(U(h)-c)}.  \label{E:gs}
\end{equation}
\end{lemma}

\begin{proof}
We rewrite the Rayleigh equation (\ref{rayleigh}) as 
\begin{equation}
-\phi ^{\prime \prime }+\alpha ^{2}\phi +\frac{U^{\prime \prime }}{U-c}\phi
=0.  \label{2nd-rayleigh}
\end{equation}%
Multiplication of above by $\phi ^{\ast }$ and integration by parts with the
boundary condition (\ref{bc-rayleigh}) yield 
\begin{equation*}
\int_{0}^{h}\left( \left\vert \phi ^{\prime }\right\vert ^{2}+\alpha
^{2}|\phi |^{2}+\frac{U^{\prime \prime }}{U-c}|\phi |^{2}\right)
dy=g_{r}(c)|\phi (h)|^{2}.
\end{equation*}%
Its real and imaginary parts read as 
\begin{equation}
\int_{0}^{h}\left( \left\vert \phi ^{\prime }\right\vert ^{2}+\alpha
^{2}\left\vert \phi \right\vert ^{2}+\frac{U^{\prime \prime }\left(
U-c_{r}\right) }{\left\vert U-c\right\vert ^{2}}\left\vert \phi \right\vert
^{2}\right) dy={\rm Re}g_{r}(c)\left\vert \phi \left( h\right) \right\vert
^{2}  \label{zero-real}
\end{equation}%
and 
\begin{equation}
c_{i}\int_{0}^{h}\frac{U^{\prime \prime }}{|U-c|^{2}}|\phi |^{2}dy={\rm Im}%
g_{r}(c)|\phi (h)|^{2},  \label{zero-imag}
\end{equation}%
respectively. Combining (\ref{zero-real}) and (\ref{zero-imag}) then
establishes (\ref{estimate-Jq-f}).

Similarly, combining the real and the imaginary parts of 
%we get (\ref{estimate-Jq-f}%). Similarly,%
\begin{equation}
\int_{0}^{h}\left( \left\vert \phi ^{\prime }\right\vert ^{2}+\alpha
^{2}|\phi |^{2}+\frac{U^{\prime \prime }}{U-c}|\phi |^{2}\right)
dy=g_{s}^{\ast }(c)|\phi ^{\prime }(h)|^{2}  \label{zweo-g-complex}
\end{equation}%
leads to (\ref{estimate-Jq-g}).
\end{proof}

Our next preliminary result is an a priori $H^{2}$-estimate of unstable
solutions near a neutral limiting mode.

\begin{lemma}
\label{lemma-h2-bound} For $U\in \mathcal{K}^{+}$, let $\{(\phi _{k},\alpha
_{k},c_{k})\}_{k=1}^{\infty }$ be a sequence of unstable solutions 
%(that is, $\operatorname{Im}c_{k}>0$)
to (\ref{rayleigh})--(\ref{bc-rayleigh}) such that $\Vert \phi _{k}\Vert
_{L^{2}}=1$. If $\alpha _{k}$ $\rightarrow \alpha _{s}>0$ and ${\rm Im}%
c_{k}\rightarrow 0+$ as $k\rightarrow \infty $ then $\left\Vert \phi
_{k}\right\Vert _{H^{2}}\leq C$, where $C>0$ is independent of $k$.
\end{lemma}

\begin{proof}
By the semicircle theorem (\ref{semicircle}), $c_{k}\in \lbrack U_{\min
},U_{\max }]$ for any $k>1$. Thus $c_{k}\rightarrow c_{s}$ $\in \lbrack
U_{\min },U_{\max }]$, as $k\rightarrow \infty $. The proof is divided into
to the following two cases.

Case 1: $U(h)\neq c_{s}$. The proof is similar to that of \cite[Lemma 3.7]%
{lin1}. It is straightforward to see that 
\begin{equation}
|{\rm Re}g_{r}(c_{k})|+\left\vert \frac{{\rm Im}g_{r}(c_{k})}{{\rm Im}%
c_{k}}\right\vert \leq C_{0},  \label{bound-gr}
\end{equation}%
where $C_{0}>0$ is independent of $k$. For simplicity of notations, the
subscript $k$ is suppressed in the estimates below and $C$ is used to denote
generic constants which are independent of $k$. Let $c=c_{r}+ic_{i}$. We
write (\ref{estimate-Jq-f}) as 
\begin{equation*}
\int_{0}^{h}\left( \left\vert \phi ^{\prime }\right\vert ^{2}+\alpha
^{2}|\phi |^{2}+K(y)\frac{(U-U_{s})(U-q)}{|U-c_{r}|^{2}+c_{i}^{2}}|\phi
|^{2}\right) dy=\left( {\rm Re}g_{r}(c)+\frac{c_{r}-q}{c_{i}}{\rm Im}%
g_{r}(c)\right) \left\vert \phi (h)\right\vert ^{2}.
\end{equation*}%
Note that $g_{r}(c_{s})$ is well defined. Setting $q=U_{s}-2\left(
U_{s}-c_{r}\right) $ in the above identity leads to 
\begin{align*}
\int_{0}^{h}(\left\vert \phi ^{\prime }\right\vert ^{2}+\alpha ^{2}|\phi
|^{2})dy\leq & \int_{0}^{h}K(y)\frac{(U-U_{s})^{2}+2(U-U_{s})(U_{s}-c_{r})}{%
|U-c_{r}|^{2}+c_{i}^{2}}|\phi |^{2}dy+C|\phi (h)|^{2} \\
=& \int_{0}^{h}K(y)\frac{(U-c_{r})^{2}-(U_{s}-c_{r})^{2}}{%
|U-c_{r}|^{2}+c_{i}^{2}}|\phi |^{2}dy+C|\phi (h)|^{2} \\
\leq & \int_{0}^{h}K(y)|\phi |^{2}dy+C(\varepsilon \left\Vert \phi ^{\prime
}\right\Vert _{L^{2}}^{2}+\frac{1}{\varepsilon }\left\Vert \phi \right\Vert
_{L^{2}}^{2}).
\end{align*}%
One chooses $\varepsilon $ sufficiently small to conclude that $\Vert \phi
\Vert _{H^{1}}\leq C$.

Next is the $H^{2}$-estimate. Multiplication of (\ref{2nd-rayleigh}) 
%the Rayleigh equation \[
%\left(  \frac{d^{2}}{dy^{2}}-\alpha^{2}\right)  \phi-\frac{U^{\prime\prime}}{U-c}\phi=0\]
by $-(\phi ^{\ast })^{\prime \prime }$ and integration over the interval $%
[0,h]$ yield 
\begin{equation}
\int_{0}^{h}(\left\vert \phi ^{\prime \prime }\right\vert ^{2}+\alpha
^{2}\left\vert \phi ^{\prime }\right\vert ^{2})dy-\alpha ^{2}g_{r}^{\ast
}(c)|\phi (h)|^{2}=\int_{0}^{h}(\phi ^{\ast })^{\prime \prime }\frac{%
U^{\prime \prime }}{U-c}\phi \,dy.  \label{E:2nd-ray}
\end{equation}%
In view of the Rayleigh equation for $\phi ^{\ast }$, the right side is
written as 
\begin{align*}
\int_{0}^{h}(\phi ^{\ast })^{\prime \prime }\frac{U^{\prime \prime }}{U-c}%
\phi \,dy& =\int_{0}^{h}\left( \alpha ^{2}\phi ^{\ast }+\left( \frac{%
U^{\prime \prime }}{U-c}\phi \right) ^{\ast }\right) \frac{U^{\prime \prime }%
}{U-c}\phi \,dy \\
& =\alpha ^{2}\int_{0}^{h}\frac{U^{\prime \prime }}{U-c}|\phi
|^{2}dy+\int_{0}^{h}\frac{\left( U^{\prime \prime }\right) ^{2}}{|U-c|^{2}}%
|\phi |^{2}dy.
\end{align*}%
The real part of (\ref{E:2nd-ray}) then reads as 
\begin{align*}
\int_{0}^{h}(\left\vert \phi ^{\prime \prime }\right\vert ^{2}+\alpha
^{2}\left\vert \phi ^{\prime }\right\vert ^{2})dy& =\alpha ^{2}{\rm Re}%
g_{r}(c)\left\vert \phi (h)\right\vert ^{2}+\alpha ^{2}\int_{0}^{h}\frac{%
U^{\prime \prime }(U-c_{r})}{|U-c|^{2}}|\phi |^{2}dy+\int_{0}^{h}\frac{%
\left( U^{\prime \prime }\right) ^{2}}{|U-c|^{2}}|\phi |^{2}dy \\
& =\alpha ^{2}\left( 2{\rm Re}g_{r}(c)|\phi
(h)|^{2}-\int_{0}^{h}(\left\vert \phi ^{\prime }\right\vert ^{2}+\alpha
^{2}|\phi |^{2})dy\right) +\int_{0}^{h}\frac{\left( U^{\prime \prime
}\right) ^{2}}{|U-c|^{2}}|\phi |^{2}dy,.
\end{align*}%
where in the last equality we use (\ref{zero-real}). So 
\begin{equation}
\int_{0}^{h}(\left\vert \phi ^{\prime \prime }\right\vert ^{2}+2\alpha
^{2}\left\vert \phi ^{\prime }\right\vert ^{2}+\alpha ^{4}|\phi
|^{2})dy=2\alpha ^{2}{\rm Re}g_{r}(c)|\phi (h)|^{2}+\int_{0}^{h}\frac{%
\left( U^{\prime \prime }\right) ^{2}}{|U-c|^{2}}|\phi |^{2}dy
\label{inter7}
\end{equation}%
By (\ref{estimate-Jq-f}) with $q=U_{s}$, 
\begin{align*}
\int_{0}^{h}\frac{\left( U^{\prime \prime }\right) ^{2}}{|U-c|^{2}}|\phi
|^{2}dy& \leq \Vert K\Vert _{L^{\infty }}\int_{0}^{h}\frac{-U^{\prime \prime
}(U-U_{s})}{|U-c|^{2}}|\phi |^{2}dy\text{ (since }-U^{\prime \prime
}(U-U_{s})\geq 0\text{) } \\
& =\Vert K\Vert _{L^{\infty }}\left( \int_{0}^{h}(\left\vert \phi ^{\prime
}\right\vert ^{2}+\alpha ^{2}|\phi |^{2})dy-\left( {\rm Re}g_{r}(c)+\left(
c_{r}-U_{s}\right) \frac{{\rm Im}g_{r}(c)}{c_{i}}\right) |\phi
(h)|^{2}\right) \\
& \leq C\left\Vert \phi \right\Vert _{H^{1}}^{2},\text{ }
\end{align*}%
where we use the bound (\ref{bound-gr}) and that $||\phi (h)|\leq C_{0}\Vert
\phi \Vert _{H^{1}}$. Therefore, by (\ref{inter7}) and the $\left\Vert \phi
\right\Vert _{H^{1}}$ bound, we get $\left\Vert \phi \right\Vert
_{H^{2}}\leq C$ as desired.

Case 2: $U(h)=c_{s}$. The proof is similar to that of Case 1 except that one
uses $\phi (h)=g_{s}(c)\phi ^{\prime }(h)$ in place of $\phi ^{\prime
}(h)=g_{r}(c)\phi (h)$. It can be check that 
\begin{equation}
|{\rm Re}g_{s}(c_{k})|+\left\vert \frac{{\rm Im}g_{s}(c_{k})}{{\rm Im}%
c_{k}}\right\vert \leq C_{0}\,d(c_{k},U(h)),  \label{bound-gs}
\end{equation}%
where $C_{0}>0$ is independent of $k$ and 
\begin{equation*}
d(c_{k},U(h))=|{\rm Re}c_{k}-U(h)|+\left( {\rm Im}c_{k}\right) ^{2}.
\end{equation*}%
Since $U(h)=c_{s}$, it follows that $d(c_{k},U(h))\rightarrow 0$ as $%
k\rightarrow \infty $. The same computations as in Case 1 establish that 
\begin{equation*}
\int_{0}^{h}(\left\vert \phi _{k}^{\prime }\right\vert ^{2}+\alpha
_{k}^{2}|\phi _{k}|^{2})dy\leq \int_{0}^{h}K(y)|\phi
_{k}|^{2}dy+C_{0}d(c_{k},U(h))\left\vert \phi _{k}^{\prime }\left( h\right)
\right\vert ^{2}
\end{equation*}%
and%
\begin{eqnarray}
&&\int_{0}^{h}(\left\vert \phi _{k}^{\prime \prime }\right\vert ^{2}+2\alpha
_{k}^{2}\left\vert \phi _{k}^{\prime }\right\vert ^{2}+\alpha _{k}^{4}|\phi
_{k}|^{2})dy\newline
\label{estimate-2nd-gs} \\
&=&2\alpha _{k}^{2}{\rm Re}\,g_{s}(c_{k})|\phi _{k}^{\prime
}(h)|^{2}+\int_{0}^{h}\frac{\left( U^{\prime \prime }\right) ^{2}}{|U-c|^{2}}%
|\phi _{k}|^{2}dy  \notag \\
&\leq &C\left( \int_{0}^{h}(\left\vert \phi _{k}^{\prime }\right\vert
^{2}+\alpha _{k}^{2}|\phi _{k}|^{2})dy+d(c_{k},U(h))\left\vert \phi
_{k}^{\prime }\left( h\right) \right\vert ^{2}\right)  \notag \\
&\leq &C\left( \int_{0}^{h}\left( \varepsilon \left\vert \phi _{k}^{\prime
\prime }\right\vert ^{2}+\left( \frac{1}{\varepsilon }+\alpha
_{k}^{2}\right) \left\vert \phi _{k}\right\vert ^{2}\right)
dy+d(c_{k},U(h))\left\vert \phi _{k}^{\prime }\left( h\right) \right\vert
^{2}\right)  \notag
\end{eqnarray}%
where $C>0$ is independent of $k$. Consequently, by choosing $\varepsilon $
small, 
\begin{equation*}
\Vert \phi _{k}\Vert _{H^{2}}^{2}\leq C_{1}(\Vert \phi _{k}\Vert
_{L^{2}}^{2}+d(c_{k},U(h))\left\vert \phi _{k}^{\prime }\left( h\right)
\right\vert ^{2})\leq C_{2}(\Vert \phi _{k}\Vert
_{L^{2}}^{2}+d(c_{k},U(h))\Vert \phi _{k}\Vert _{H^{2}}^{2}),
\end{equation*}%
where $C_{1},C_{2}>0$ are independent of $k$. Since $d(c_{k},U(h))%
\rightarrow 0$ as $k\rightarrow \infty $ it follows from above that $%
\left\Vert \phi _{k}\right\Vert _{H^{2}}^{2}\leq C$.
\end{proof}

For $U\in \mathcal{K}^{+}$, if $c$ is in the range of $U$ then $U(y)=c$
holds at a finite number of points \cite[Remark 3.2]{lin1}, which are
denoted by $y_{1}<y_{2}<\cdots <y_{m_{c}}$. Let $y_{0}=0$ and $y_{m_{c}+1}=h$%
. We state our last preliminary result.

\begin{lemma}[\protect\cite{lin1}, Lemma 3.5]
\label{vanish} Let $\phi$ satisfy (\ref{rayleigh}) with $\alpha$ positive
and $c$ in the range of $U$ and let $U(y)=c$ for $y\in\left\{
y_{1},y_{2},\dots,y_{m_{c}}\right\} $. If $\phi$ is sectionally continuous
on the open intervals $(y_{j},y_{j+1})$, $j=0,1,\dots,m_{c}$, then $\phi$
cannot vanish at both endpoints of any intervals $(y_{j},y_{j+1})$ unless it
vanishes identically on that interval.
\end{lemma}

\begin{proof}[Proof of Proposition \protect\ref{prop-neutral-mode}]
Let $(\phi _{s},\alpha _{s},c_{s})$ be a neutral limiting mode with $\alpha
_{s}>0$ and $c_{s}\in \lbrack U_{\min },U_{\max }]$, and let $\{(\phi
_{k},\alpha _{k},c_{k})\}_{k=1}^{\infty }$ be a sequence of unstable
solutions of (\ref{rayleigh})--(\ref{bc-rayleigh}) such that $(\phi
_{k},\alpha _{k},c_{k})$ converges to $(\phi _{s},\alpha _{s},c_{s})$ as $%
k\rightarrow \infty $. We normalize the sequence by setting $\Vert \phi
_{k}\Vert _{L^{2}}=1$.

First, the result of Lemma \ref{lemma-h2-bound} says that $%
\Vert\phi_{k}\Vert_{H^{2}}\leq C$, where $C>0$ is independent of $k$.
Consequently, $\phi_{k}$ converges to $\phi_{s}$ weakly in $H^{2}$ and
strongly in $H^{1}$. Then $\Vert\phi_{s}\Vert_{H^{2}}\leq C$ 
%,\ \phi_{s}\in C^{1}[0,h] $
and $\Vert\phi_{s}\Vert_{L^{2}}=1$. Let $y_{1},y_{2},\cdots,y_{m_{s}}$ be
the roots of $U(y)-c_{s}$ and let $S_{0}$ be the complement of the set of
points $\{y_{1},y_{2},\dots,y_{m_{s}}\}$ in the interval $[0,h]$. Since $%
\phi_{k}$ converges to $\phi_{s}$ uniformly in $C^{2}$ on any compact subset
of $S_{0}$, it follows that $\phi_{s}^{\prime\prime}$ exists on $S_{0}$.
Since $(U-c_{k})^{-1}$, $\phi_{k}$ and their derivatives up to second order
are uniformly bounded on any compact subset of $S_{0}$, it follows that $%
\phi_{s}$ satisfies 
%Formally, $\left(  c_{s},\alpha_{s},\phi_{s}\right)  $ ought to satisfy the Rayleigh equation%
\begin{equation}
\phi_{s}^{\prime\prime}-\alpha_{s}^{2}\phi_{s}-\frac{U^{\prime\prime}}{%
U-c_{s}}\phi_{s}=0  \label{neutral}
\end{equation}
almost everywhere on $[0,h]$. Moreover, 
\begin{equation}
\phi_{s}^{\prime}(h)=\left( \frac{g}{(U(h)-c_{s})^{2}}+\frac{U^{\prime}(h)}{%
U(h)-c_{s}}\right) \phi_{s}(h)\quad \text{and} \quad \phi_{s}(0)=0
\label{bc-non-Ud}
\end{equation}
in case $U(h)\neq c_{s}$, and 
\begin{equation}
\phi_{s}(h)=0\quad \text{and} \quad \phi_{s}(0)=0  \label{bc-Ud}
\end{equation}
in case $U(h)=c_{s}$. 
%(\ref{neutral}) with the boundary condition (\ref{bc-non-Ud}) or (\ref{bc-Ud}).
%This uses that on any compact subset of $S_{0}$,
%$1/(U(y) -c_{k})$ is uniformly bounded
%and $\phi_{k}$ and their derivatives up to the second-order are uniformly bounded.

Next, we claim that $c_{s}$ is the inflection value $U_{s}$. By Definition %
\ref{classK}, $U^{\prime\prime}(y_{j})=-K(y_{j})(c_{s}-U_{s})$ has the same
sign for $j=1,\cdots,m_{s}$, say positive. Let 
\begin{equation*}
E_{\delta}=\cup_{i=1}^{m_{s}}\{ y\in[0,h] :\left\vert y-y_{j}\right\vert
<\delta\} .
\end{equation*}
Clearly, $E_{\delta}^{c}\subset S_{0}$ and $U^{\prime\prime}(y)>0$ for $y\in
E_{\delta}$ when $\delta>0$ sufficiently small. The proof is again divided
into two cases.

Case 1: $U(h)\neq c_{s}$. Since $\phi _{s}$ is not identically zero, Lemma %
\ref{vanish} asserts that $\phi _{s}(y_{j})\neq 0$ for some $y_{j}$. If $%
c_{s}$ were not an inflection value then near such a $y_{j}$ it must hold
that 
\begin{equation}
\int_{E_{\delta }}\frac{U^{\prime \prime }(U-U_{\min }+1)}{|U-c_{s}|^{2}}%
|\phi _{s}|^{2}dy\geq \int_{|y-y_{j}|<\delta }\frac{U^{\prime \prime }}{%
|U-c_{s}|^{2}}|\phi _{s}|^{2}dy=\infty .  \label{bd-infty}
\end{equation}%
Since%
\begin{align*}
& \int_{0}^{h}\left( \left\vert \phi _{k}^{\prime }\right\vert ^{2}+\alpha
_{k}^{2}|\phi _{k}|^{2}+\frac{U^{\prime \prime }(U-U_{\min }+1)}{%
|U-c_{k}|^{2}}|\phi _{k}|^{2}\right) dy \\
& \geq \int_{E_{\delta }}\frac{U^{\prime \prime }(U-U_{\min }+1)}{%
|U-c_{k}|^{2}}|\phi _{k}|^{2}dy-\sup_{E_{\delta }^{c}}\frac{|U^{\prime
\prime }(U-q)|}{|U-c_{k}|^{2}},
\end{align*}%
by Fatou's Lemma and (\ref{bd-infty}) it follows that 
\begin{equation*}
\liminf_{k\rightarrow \infty }\int_{0}^{h}\left( \left\vert \phi
_{k}^{\prime }\right\vert ^{2}+\alpha _{k}^{2}|\phi _{k}|^{2}+\frac{%
U^{\prime \prime }(U-U_{\min }+1)}{|U-c_{k}|^{2}}|\phi _{k}|^{2}\right)
dy=\infty .
\end{equation*}%
%
%
%
%
%
%
%
%
%
%
%
%
%
%
%Therefore, $\liminf_{k} J_{q}(\phi_{k})=\infty$.
On the other hand, (\ref{estimate-Jq-f}) with $q=U_{\min }-1$ yields 
\begin{multline*}
\int_{0}^{h}\left( \left\vert \phi _{k}^{\prime }\right\vert ^{2}+\alpha
_{k}^{2}|\phi _{k}|^{2}+\frac{U^{\prime \prime }(U-U_{\min }+1)}{%
|U-c_{k}|^{2}}|\phi _{k}|^{2}\right) dy \\
=\left( {\rm Re}g_{r}(c_{k})+\frac{{\rm Re}\,c_{k}-U_{\min }+1}{\mathrm{Im}%
\,c_{k}}\right) |\phi _{k}(h)|^{2}\leq C\Vert \phi _{k}\Vert _{H^{2}}\leq
C_{1},
\end{multline*}%
where $C_{1}>0$ is independent of $k$. A contradiction proves the claim.

Case 2: $U(h)=c_{s}$. The proof is identical to that of Case 1 except that
we use (\ref{estimate-Jq-g}) in place of (\ref{estimate-Jq-f}) and hence is
omitted.
\end{proof}

The following lemma allows us to continuate unstable modes until a neutral
limiting wave number is reached, analogous to \cite[Theorem 3.9]{lin1}.

\begin{lemma}
\label{basis} For $U\in \mathcal{K}^{+}$, the set of unstable wave numbers
is open, whose any boundary point $\alpha $ satisfies that $-\alpha ^{2}$ is
an eigenvalue of the operator $-\frac{d^{2}}{dy^{2}}-K(y)$ on $y\in (0,h)$
with the boundary conditions (\ref{bc-sturm1})--(\ref{bc-sturm2}).
\end{lemma}

\begin{proof}
An unstable solution $(\phi _{0},\alpha _{0},c_{0})$ of (\ref{rayleigh})--(%
\ref{bc-rayleigh}) is held fixed, with $\alpha _{0}>0$ and ${\rm Im}c_{0}>0$%
. For $r_{1},r_{2}>0$, let us define 
\begin{equation*}
I_{r_{1}}=\{\alpha >0:|\alpha -\alpha _{0}|<r_{1}\}\quad \text{and}\quad
B_{r_{2}}=\{c\in \mathbb{C}^{+}:|c-c_{0}|<r_{2}\},
\end{equation*}%
where $\mathbb{C}^{+}$ denotes the set of complex numbers with positive
imaginary part. Our goal is to show that for each $\alpha \in I_{r_{1}}$,
there exists $c(\alpha )\in B_{r_{2}}$ for some $r_{2}>0$ and $\phi _{\alpha
}\in H^{2}$ such that $(\phi _{\alpha },\alpha ,c(\alpha ))$ satisfies (\ref%
{rayleigh})--(\ref{bc-rayleigh}).

For $\alpha \in (0,\infty )$ and $c\in \mathbb{C}^{+}$, let $\phi _{1}\left(
x;\alpha ,c\right) $ and $\phi _{2}\left( x;\alpha ,c\right) $ to be the
solutions of the differential equation 
\begin{equation*}
\phi ^{\prime \prime }-\alpha ^{2}\phi -\frac{U^{\prime \prime }}{U-c}\phi
=0\qquad \text{for}\quad y\in (0,h)
\end{equation*}%
normalized at $h$, that is, 
\begin{equation*}
\begin{cases}
\phi _{1}(h)=1,\quad & \phi _{2}(h)=0, \\ 
\phi _{1}^{\prime }(h)=0,\quad & \phi _{2}^{\prime }(h)=1.%
\end{cases}%
\end{equation*}%
It is standard that $\phi _{1}$ and $\phi _{2}$ are analytic as functions of 
$c$ in $\mathbb{C}^{+}$. Let us consider a function on $(0,\infty )\times 
\mathbb{C}^{+}$, defined as 
\begin{equation}
\Phi (\alpha ,c)=\phi _{1}(0;\alpha ,c)+g_{r}(c)\phi _{2}(0;\alpha ,c),
\label{E:Phi}
\end{equation}%
where $g_{r}$ is defined in (\ref{E:defn-fc}). Clearly, $\Phi $ is analytic
in $c$ and continuous in $\alpha $. Note that $\Phi (\alpha ,c)=0$ if and
only if the system (\ref{rayleigh})--(\ref{bc-rayleigh}) has an unstable
solution 
\begin{equation*}
\phi _{\alpha }(y)=\phi _{1}(y;\alpha ,c)+g_{r}(c)\phi _{2}(y;\alpha ,c).
\end{equation*}

Since $\Phi (\alpha _{0},c_{0})=0$ and the zeros of an analytic function are
isolated, $\Phi (\alpha _{0},c)\neq 0$ on $\{|c-c_{0}|=r_{2}\}$, for $%
r_{2}>0 $ sufficiently small. Then, by the continuity of $\Phi (\alpha ,c)$
in $\alpha $, we have $\Phi (\alpha ,c)\neq 0$ on $\{|c-c_{0}|=r_{2}\}$,
when $\alpha \in I_{r_{1}}$ and $r_{1}$ is sufficiently small. Let us
consider the function 
\begin{equation*}
N(\alpha )=\frac{1}{2\pi i}\oint_{|c-c_{0}|=r_{2}}\frac{\partial \Phi
/\partial c(\alpha ,c)}{\Phi (\alpha ,c)}dc,
\end{equation*}%
where $\alpha \in I_{r_{1}}$. Observe that $N(\alpha )$ counts the number of
zeros of $\Phi (\alpha ,c)$ in $B_{r_{2}}$. Since $N(\alpha _{0})>0$ and $%
N(\alpha )$ is continuous as a function of $\alpha $ in $I_{r_{1}}$ it
follows that $N(\alpha )>0$ for any $\alpha \in I_{r_{1}}$. 
%In interpretation, there exists $c(\alpha)\in B_{r_{2}}$ such that $\Phi(  \alpha,c(\alpha)) =0$.
This proves that the set of unstable wave numbers is open.

By definition, the boundary points of the set of unstable wave numbers must
be neutral limiting wave numbers, say $\alpha _{s}$. Proposition \ref%
{prop-neutral-mode} asserts that $-\alpha _{s}^{2}$ must be a negative
eigenvalue of the operator $-\frac{d^{2}}{dy^{2}}-K(y)$ on $(0,h)$ with (\ref%
{bc-sturm1})--(\ref{bc-sturm2}). This completes the proof.
\end{proof}

%In light of Lemma \ref{basis}, in order to determine the set of unstable wave numbers
%we only need to understand the stability property near a neutral limiting wave number.
%This strategy is employed to find a sharp criterion for instability.

\subsection{Proof of Theorem \protect\ref{class-k}}

The proof of Theorem \ref{class-k} is to examine the bifurcation of unstable
modes from each neutral limiting mode. This reduces to the study of the
bifurcation of zeros of the algebraic equation $\Phi (\alpha ,c)=0$ from $%
(\alpha _{s},c_{s})$ as is pointed out in the proof of Lemma \ref{basis}.
However, the differentiability of $\Phi (\alpha ,c)$ in $c$ at a neutral
limiting mode $(\alpha _{s},c_{s})$ can only be established on half of the $%
c-$neighborhood, and thus the standard implicit function theorem does not
apply. To overcome this difficulty, we construct a contraction mapping in
such a half neighborhood and prove the existence of an unstable mode near $%
(\alpha _{s},c_{s})$ as is done in \cite{lin1} for the rigid-wall case.
Moreover, the existence of an unstable mode can only be established for wave
numbers slightly to the left of $\alpha _{s}$. Therefore, unstable modes
with wave numbers slightly to the left of $\alpha _{s}$ can be continuated
to the zero wave number. Our method is almost the same as in the rigid-wall
setting (\cite[Section 4]{lin1}).

Below we construct unstable solutions for wave numbers slightly to the left
of a neutral limiting wave number.

\begin{proposition}
\label{theorem-unstable} For $U\left( y\right) \ $in the class $\mathcal{K}%
^{+}$, let $U_{s}$ be the inflection value and $y_{1},y_{2},\dots ,y_{m_{s}}$
be the inflection points, as such $U(y_{j})=U_{s}$ for $i=1,2,\dots ,m_{s}$.
If $(\phi _{s},\alpha _{s},U_{s})$ is a neutral limiting mode with $\alpha
_{s}$ positive, then for $\varepsilon \in (\varepsilon _{0},0)$, where $%
|\varepsilon _{0}|$ is sufficiently small, there exist $\phi _{\varepsilon }$
and $c(\varepsilon )$ such that 
\begin{equation}
\phi _{\varepsilon }^{\prime \prime }-\left( \alpha _{s}^{2}+\varepsilon
\right) \phi _{\varepsilon }-\frac{U^{\prime \prime }}{U-U_{s}-c(\varepsilon
)}\phi _{\varepsilon }=0\qquad \text{for}\quad y\in (0,h),
\label{eqn-ray-shift}
\end{equation}%
\begin{equation}
\phi _{\varepsilon }^{\prime }(h)=g_{r}(U_{s}+c(\varepsilon ))\phi
_{\varepsilon }(h)\quad \text{and}\quad \phi _{\varepsilon }(0)=0
\label{bc-rayleigh-shift}
\end{equation}%
with $\mathrm{Im}\,c(\varepsilon )>0$. 
%where $\alpha(\varepsilon)  =\sqrt{\alpha_s^2+ \varepsilon}$.
Moreover, 
%$c\left(  \varepsilon\right)  $ is differentiable in the interval $\left(  \varepsilon_{0},0\right)  $ and%
\begin{equation}
\lim_{\varepsilon \rightarrow 0-}c(\varepsilon )=0,  \label{limit c}
\end{equation}%
\begin{equation}
\lim_{\varepsilon \rightarrow 0-}\frac{dc}{d\varepsilon }(\varepsilon
)=\left( \int_{0}^{h}\phi _{s}^{2}dy\right) \left( i\pi \sum_{j=1}^{m_{s}}%
\frac{K(y_{j})}{|U(y_{j})|}\phi _{s}^{2}(y_{j})+p.v.\int_{0}^{h}\frac{K}{%
U-U_{s}}\phi _{s}^{2}dy+A\right) ^{-1};  \label{limit rate}
\end{equation}%
here $p.v.$ means the Cauchy principal part and 
\begin{equation}
A=%
\begin{cases}
\frac{2g}{(U(h)-U_{s})^{3}}+\frac{U^{\prime }(h)}{(U(h)-U_{s})^{2}}\quad & 
\text{if}\quad U(h)\neq U_{s} \\ 
0 & \text{if}\quad U(h)=U_{s}.%
\end{cases}
\label{definition-A}
\end{equation}
\end{proposition}

\begin{proof}
As in the proof of Lemma \ref{basis}, for $c\in\mathbb{C}^{+}$ and $%
\varepsilon<0,$ let $\phi_{1}(y;\varepsilon,c)$ and $\phi_{2}(y;\varepsilon
,c)$ be the solutions of 
\begin{equation}
\phi^{\prime\prime}-(\alpha_{s}^{2}+\varepsilon)\phi-\frac{U^{\prime\prime}}{%
U-U_{s}-c}\phi=0\quad\text{for}\quad y\in(0,h)  \label{homo-ode}
\end{equation}
normalized at $h$, that is 
\begin{equation*}
\begin{cases}
\phi_{1}(h)=1,\quad & \phi_{2}(h)=0, \\ 
\phi_{1}^{\prime}(h)=0,\quad & \phi_{2}^{\prime}(h)=1.%
\end{cases}%
\end{equation*}
%
%\phi_{1}\left(  d\right)  =1,$ $\phi_{1}^{\prime}\left(  d\right)  =0$
%and $\phi_{2}\left(  d\right)  =0,\phi_{2}^{\prime}\left(  d\right)  =1.$
It is standard that $\phi_{1}$ and $\phi_{2}$ are analytic as a function of $%
c$ in $\mathbb{C}^{+}$ and that $\phi_{1}$ and $\phi_{2}$ are linearly
independent with Wronskian $1$. The neutral limiting mode is normalized so
that $\phi_{s}(h)=1$ and $\phi^{\prime}_{s}(h)=g_{r}(U_{s})$. The proof is
again divided into two cases.

Case 1: $U(h)\neq U_{s}$. Let us define 
\begin{align*}
\phi _{0}(y;\varepsilon ,c)& =\phi _{1}(y;\varepsilon ,c)+g_{r}(U_{s}+c)\phi
_{2}(y;\varepsilon ,c), \\
\Phi (\varepsilon ,c)& =\phi _{1}(0;\varepsilon ,c)+g_{r}(U_{s}+c)\phi
_{2}(0;\varepsilon ,c),
\end{align*}%
where $g_{r}$ is given in (\ref{E:defn-fc}). It is readily seen that $\phi
_{0}$ solves (\ref{homo-ode}) with $\phi _{0}(h)=1$ and $\phi _{0}^{\prime
}(h)=g_{r}(U_{s}+c)$. It is easy to see that $\Phi (\varepsilon ,c)$ is
analytic in $c\in $ $\mathbb{C}^{+}$ and differentiable in $\varepsilon $.
Note that an unstable solution to (\ref{eqn-ray-shift})--(\ref%
{bc-rayleigh-shift}) exists if and only if $\Phi (\varepsilon ,c)=0$ for
some ${\rm Im}c>0$. The Green's function of (\ref{homo-ode}) is written as 
\begin{equation*}
G(y,y^{\prime };\varepsilon ,c)=\bar{\phi}_{1}(y;\varepsilon ,c)\phi
_{2}(y^{\prime };\varepsilon ,c)-\phi _{2}(y;\varepsilon ,c)\bar{\phi}%
_{1}(y^{\prime };\varepsilon ,c),
\end{equation*}%
where $\bar{\phi}_{1}(y;\varepsilon ,c)$ is the solution of (\ref{homo-ode})
with $\bar{\phi}_{1}(h)=1$ and $\bar{\phi}_{1}^{\prime }(h)=g_{r}(U_{s})$. A
similar computation as in \cite[pp. 336]{lin1} for $\phi _{j}(y;\varepsilon
,c)$ ($j=1,2$) yields that 
\begin{equation}
\frac{\partial \Phi }{\partial \varepsilon }(\varepsilon
,c)=-\int_{0}^{h}G(y,0;\varepsilon ,c)\phi _{0}(y;\varepsilon ,c)dy
\label{differ1}
\end{equation}%
and%
\begin{equation}
\frac{\partial \Phi }{\partial c}(\varepsilon
,c)=\int_{0}^{h}G(y,0;\varepsilon ,c)\frac{-U^{\prime \prime }}{%
(U-U_{s}-c)^{2}}\phi _{0}(y;\varepsilon ,c)dy+\frac{d}{dc}g_{r}(U_{s}+c)\phi
_{2}(0;\varepsilon ,c).  \label{differ3}
\end{equation}

Let us define the triangle in $\mathbb{C}^{+}$ as 
\begin{equation*}
\Delta_{(R,b)}=\{ c_{r}+ic_{i} : |c_{r}| <Rc_{i}\,,\,0<c_{i}<b\}
\end{equation*}
and the Cartesian product in $(0,\infty) \times\mathbb{C}^{+}$ as 
\begin{equation*}
E_{(R,b_{1},b_{2}) }=(-b_{2}, 0) \times\Delta_{(R,b_{1})},
\end{equation*}
where $R, b_{1}, b_{2}>0$ are to be determined later.

We claim that:

(a) For fixed $R,$ both $\bar{\phi}_{1}(\cdot;\varepsilon,c)$ and $%
\phi_{0}(\cdot;\varepsilon,c)$ uniformly converge to $\phi_{s}$ in $%
C^{1}[0,h]$ as $(\varepsilon,c)\rightarrow(0,0)$ in $E_{(R,b_{1},b_{2})}.$
That is, for any $\delta>0$ there exists some $b_{0}>0$ such that whenever $%
b_{1},b_{2}<b_{0}$ and $(\varepsilon,c)\in E_{(R,b_{1},b_{2})}$ the
inequalities 
\begin{equation*}
\left\Vert \bar{\phi}_{1}(\cdot;\varepsilon,c) -\phi _{s}\right\Vert
_{C^{1}},\left\Vert \phi_{0}( \cdot;\varepsilon ,c) -\phi_{s} \right\Vert
_{C^{1}}\leq\delta
\end{equation*}
hold.

(b) $\phi _{2}(\cdot ;\varepsilon ,c)$ converges uniformly to $\phi
_{2}\left( y;0,0\right) \ $in the sense of (a). We denote $\phi _{2}\left(
y;0,0\right) =\phi _{z}\left( y\right) ,$ then $\phi _{z}\left( 0\right) =-%
\frac{1}{\phi _{s}^{\prime }\left( 0\right) }$ since the Wronskian of $%
\left( \bar{\phi}_{1},\phi _{2}\right) $ and its limit $\left( \phi
_{s},\phi _{z}\right) $ is $1.$ The proof of (a) and (b) is very similar to 
\cite[pp. 337]{lin1} and we skip it.

In the appendix, we prove that 
\begin{equation}
\frac{\partial \Phi }{\partial \varepsilon }(\varepsilon ,c)\rightarrow 
\frac{1}{\phi _{s}^{\prime }(0)}\int_{0}^{h}\phi _{s}^{2}dy,
\label{i-conver 1}
\end{equation}%
\begin{equation}
\frac{\partial \Phi }{\partial c}(\varepsilon ,c)\rightarrow -\frac{1}{\phi
_{s}^{\prime }(0)}\left( i\pi \sum_{j=1}^{m_{s}}\frac{K(y_{j})}{|U^{\prime
}(y_{j})|}\phi _{s}^{2}(y_{j})+p.v.\int_{0}^{h}\frac{K}{U-U_{s}}\phi
_{s}^{2}dy+A\right)  \label{i-conver2}
\end{equation}%
uniformly as $\varepsilon \rightarrow 0-$ and $c\rightarrow 0$ in $%
E_{(R,b_{1},b_{2})}$, where $A$ is defined by (\ref{definition-A}) and $%
y_{j} $ ($j=1,\dots ,m_{s}$) are the inflection points for $U_{s}$. Let us
denote 
\begin{align}
B& =\frac{1}{\phi _{s}^{\prime }(0)}\int_{0}^{h}\phi _{s}^{2}dy,  \notag
\label{E:BCD} \\
C& =-\frac{1}{\phi _{s}^{\prime }(0)}\left( p.v.\int_{0}^{h}\frac{K(y)}{%
U-U_{s}}\phi _{s}^{2}dy+A\right) , \\
D& =-\frac{\pi }{\phi _{s}^{\prime }(0)}\sum_{j=1}^{m_{s}}\frac{K(y_{j})}{%
|U^{\prime }(y_{j})|}\phi _{s}^{2}(y_{j}).  \notag
\end{align}%
Lemma \ref{vanish} asserts that $\phi _{s}$ is nonzero at one of the
inflection points, say $\phi _{s}(y_{j})\neq 0$. Note that by \cite[Remark
4.2]{lin1}, for class $\mathcal{K}^{+}$ flows, $U^{\prime }(y_{j})\neq 0$
for $j=1,2,\dots ,m_{s}$. Consequently, $D<0$.

The remainder of the proof is identically the same as that of \cite[Theorem
4.1]{lin1} and hence we only sketch it. Define 
\begin{align}
f(\varepsilon ,c)& =\Phi (\varepsilon ,c)-B\varepsilon -(C+Di)c, \\
F(\varepsilon ,c)& =-\frac{B}{C+iD}\varepsilon -\frac{f(\varepsilon ,c)}{C+iD%
}.
\end{align}%
Note that for each $\varepsilon <0$ fixed, a zero of $\Phi (\varepsilon
,\cdot )$ corresponds to a fixed point of the mapping $F(\varepsilon ,\cdot
) $. It is shown in \cite[pp. 338-339]{lin1} that for $\varepsilon \in
(\varepsilon _{0},0)$, where $|\varepsilon _{0}|$ is sufficiently small, the
mapping $F(\varepsilon ,\cdot )$ is contracting on $\Delta _{(R,b\left(
\varepsilon \right) )}$ for some $R>0$ and 
\begin{equation*}
b(\varepsilon )=-2DB(C^{2}+D^{2})^{-1}\varepsilon .
\end{equation*}%
So for each $\varepsilon \in (\varepsilon _{0},0)$ there exists a unique $%
c(\varepsilon )\in \Delta _{(R,b(\varepsilon ))}$ such that $F(\varepsilon
,c(\varepsilon ))=c(\varepsilon )$ and thus $\Phi (\varepsilon
,c(\varepsilon ))=0$. It can be also shown that $c(\varepsilon )$ is
differentiable in $\varepsilon $ in the interval $(\varepsilon _{0},0)$.
Since $c(\varepsilon )\in \Delta _{(R,b(\varepsilon ))}$, it is immediate
that (\ref{limit c}) holds. Finally, differentiation of $\Phi (\varepsilon
,c(\varepsilon ))=0$ yields 
\begin{equation*}
c^{\prime }(\varepsilon )=-\frac{\partial \Phi /\partial \varepsilon }{%
\partial \Phi /\partial c},
\end{equation*}%
which, in view of (\ref{i-conver 1}) and (\ref{i-conver2}), implies (\ref%
{limit rate}).

Case 2: $U(h)=U_{s}$. The proof is almost identical to that of Case 1, and
thus we only indicate some differences. Let us define 
\begin{equation*}
\Phi(\varepsilon,c)=g_{s}(U_{s}+c)\phi_{1}(0;\varepsilon,c)+\phi
_{2}(0;\varepsilon,c),
\end{equation*}
where $g_{s}$ is defined in (\ref{E:gs}). Note that $\phi_{s}(h)=%
\phi_{s}(0)=0$. Thereby, the Green's function is written as 
\begin{equation*}
G(y,y^{\prime };\varepsilon,c)=\phi_{1}(y;\varepsilon,c)\phi_{2}(y^{\prime
};\varepsilon,c) -\phi_{2}(y;\varepsilon,c)\phi_{1}(y^{\prime
};\varepsilon,c) .
\end{equation*}
The same computations as in Case 1 yield that 
\begin{equation*}
\frac{\partial\Phi}{\partial\varepsilon}\rightarrow-\frac{1}{%
\phi_{s}^{\prime }(0)}\int_{0}^{h}\phi_{s}^{2}dy
\end{equation*}%
and 
\begin{equation*}
\frac{\partial\Phi}{\partial c}(\varepsilon,c)\rightarrow-\frac{1}{\phi
_{s}^{\prime}(0)}\left( i\pi\sum_{j=1}^{m_{s}}\frac{K(y_{j})}{|U^{\prime
}(y_{j})|}\phi_{s}^{2}(y_{j})+p.v.\int_{0}^{h}\frac{K}{U-U_{s}}%
\phi_{s}^{2}dy\right)
\end{equation*}
uniformly as $\varepsilon\rightarrow0-$ and $c\rightarrow 0$ in $E_{(
R,b_{1},b_{2}) }$. This completes the proof.
\end{proof}

\begin{proof}[Proof of Theorem \protect\ref{class-k}]
Let $-\alpha _{N}^{2}<-\alpha _{N-1}^{2}<\cdots <-\alpha _{1}^{2}<0$ be
negative eigenvalues of the operator $-\frac{d^{2}}{dy^{2}}-K(y)$ on $(0,h)$
with boundary conditions (\ref{bc-sturm1})--(\ref{bc-sturm2}). That is, $%
\alpha _{N}=\alpha _{\mathrm{max}}$, where $-\alpha _{\mathrm{max}}^{2}$ is
defined either in (\ref{E:amax1}) or (\ref{E:amax2}). 
%where $\alpha_{m}=\alpha_{\max}$ is defined by (\ref{E:amax1} or (\ref{E:amax2}.
We deduce from Lemma \ref{basis} and Proposition \ref{theorem-unstable} that
to each $\alpha \in (0,\alpha _{N})$ with $\alpha \neq \alpha _{j}$ ($%
j=1,\dots ,N$) an unstable solution is associated. Our goal is to show the
instability at $\alpha =\alpha _{j}$ for each $j=1,\dots ,N-1$.

%For each $\alpha\in(\alpha_{j},\alpha_{j+1})$, there exists an unstable wave speed
%$c(\alpha)$ with $\im c(\alpha)>0$. Below we prove that
%$\im c(\alpha)>\delta>0$ as $\alpha_{k}\to \alpha_{j}+$;
%namely, $c_i$ is bounded from below.
%Assuming $\left(  \ref{lower bound}\right)  $, we easily conclude the
%existence of \ an unstable eigenvalue at $\alpha_{i}$.
Case 1: $U(h)\neq U_{s}$. Let $\{(\phi _{k},\alpha
_{k},c_{k})\}_{k=1}^{\infty }$ be a sequence of unstable solutions such that 
$\alpha _{k}\rightarrow \alpha _{j}+$ as $k\rightarrow \infty $. After
normalization, we may assume $\phi _{k}(h)=1$. Note that $\phi _{k}$
satisfies 
\begin{equation}
\phi _{k}^{\prime \prime }-\alpha _{k}^{2}\phi _{k}-\frac{U^{\prime \prime }%
}{U-c_{k}}\phi _{k}=0\qquad \text{for}\quad y\in (0,h)  \label{inter3}
\end{equation}%
and $\phi _{k}^{\prime }(h)=g_{r}(c_{k})$, $\phi _{k}(0)=0$. Below we will
prove that ${\rm Im}c_{k}\geq \delta >0$, where $\delta $ is independent of 
$k$. Since the coefficients of (\ref{inter3}) and $g_{r}(c_{k})$ are bounded
uniformly for $k$, the solutions $\phi _{k}$ of the above Rayleigh equations
are uniformly bounded in $C^{2}$, and subsequently, $\phi _{k}$ converges in 
$C^{2}$ as $k\rightarrow \infty $, say to $\phi _{\infty }$. The semicircle
theorem (\ref{semicircle}) ensures that $c_{k}\rightarrow c_{\infty }$ as $%
k\rightarrow \infty $. Note that $\mathrm{Im}\,c_{\infty }\geq \delta >0$.
By continuity, $\phi _{\infty }$ satisfies 
\begin{equation*}
\phi _{\infty }^{\prime \prime }-\alpha _{j}^{2}\phi _{\infty }-\frac{%
U^{\prime \prime }}{U-c_{\infty }}\phi _{\infty }=0\qquad \text{for}\quad
y\in (0,h),
\end{equation*}%
with $\phi _{\infty }(h)=1$, $\phi _{\infty }^{\prime }(h)=g_{r}(c_{\infty
}) $ and $\phi _{\infty }(0)=0$. That is, $(\phi _{\infty },\alpha
_{j},c_{\infty })$ is an unstable solution of (\ref{rayleigh})--(\ref%
{bc-rayleigh}).

It remains to show that $\{\mathrm{Im}\,c_{k}\}$ has a positive lower bound.
Suppose on the contrary that $\mathrm{Im}\,c_{k}\rightarrow0$ as $%
k\rightarrow\infty$. 
%Theorem \ref{neutral mode} dictates that $\phi_\infty$ must be a limiting neutral
%; then, there would exists a sequence of unstable solutions
%$\{(c_{k},\alpha_{k},\phi_{k})\}  _{k=1}^{\infty}$ to (\ref{rayleigh})--(\ref{bc-rayleigh})
%and $\alpha_{k}\downarrow \alpha_{i}$ and $\operatorname{Im}c_{k}\downarrow 0.$
%Again, the proof is divided into two cases.

%Case 1: $U_{s}\neq U(h)$. %We normalize $\phi_k$ so that $\phi_{k}(h)=1$.
%Subsequently, $\phi_{k}'(h)=g_r(c_{k})$ and $\phi_{k}(0)  =0$.
We claim that $\Vert \phi _{k}\Vert _{L^{2}}\leq C$, where $C>0$ is
independent of $k$. Otherwise, $\Vert \phi _{k}\Vert _{L^{2}}\rightarrow
\infty $ as $k\rightarrow \infty $. Let $\varphi _{k}=\phi _{k}/\Vert \phi
_{k}\Vert _{L^{2}}$, then $\Vert \varphi _{k}\Vert _{L^{2}}=1$. 
%$\{(  c_{k}, \alpha_{k}, \tilde{\phi}_{k})\}  _{k=1}^{\infty}$ is a sequence of unstable solution
%of (\ref{rayleigh})--(\ref{bc-rayleigh}).
Lemma \ref{lemma-h2-bound} then dictates that $\Vert \varphi _{k}\Vert
_{H^{2}}\leq C$ independently of $k$. Subsequently, Proposition \ref%
{prop-neutral-mode} ensures that $(\varphi _{k},\alpha _{k},c_{k})$
converges to a neutral limiting mode $(\varphi _{s},\alpha _{s},U_{s})$. By
continuity, $\Vert \varphi _{s}\Vert _{L^{2}}=1$ and 
\begin{equation*}
\varphi _{s}^{\prime \prime }-\alpha _{j}^{2}\varphi _{s}+K(y)\varphi
_{s}=0\qquad \text{for}\quad y\in (0,h).
\end{equation*}%
%
%
%
%
%
%
%
%
%
%
%
%
%
%with (\ref{bc-sturm1})--(\ref{bc-sturm2}.
On the other hand, $\varphi _{s}(h)=\varphi _{s}^{\prime }(h)=0$ since $%
\varphi _{k}(h)=1/\Vert \phi _{k}\Vert _{L^{2}}\rightarrow 0$ and $\varphi
_{k}^{\prime }(h)=g_{r}(c_{k})/\Vert \phi _{k}\Vert _{L^{2}}\rightarrow 0$
as $k\rightarrow \infty $. Correspondingly, $\varphi _{s}\equiv 0$ on $[0,h]$%
. A contradiction proves the claim.

Since $\Vert \phi _{k}\Vert _{L^{2}}$ is bounded uniformly for $k$, Lemma %
\ref{lemma-h2-bound} and Proposition \ref{prop-neutral-mode} apply and $%
\Vert \phi _{k}\Vert _{H^{2}}\leq C$, $c_{k}\rightarrow U_{s}$ and $\phi
_{k}\rightarrow \phi _{s}$ in $C^{1}$, where $\phi _{s}$ satisfies 
\begin{equation*}
\phi _{s}^{\prime \prime }-\alpha _{j}^{2}\phi _{s}-\frac{U^{\prime \prime }%
}{U-U_{s}}\phi _{s}=0\qquad \text{for}\quad y\in (0,h)
\end{equation*}%
with $\phi _{s}(h)=1$, $\phi _{s}^{\prime }(h)=g_{r}(U_{s})$ and $\phi
_{s}(0)=0$. %Multiplication of (\ref{inter4}) by $\phi_{k}$ and
%subtraction by $\phi_{s}$ times $(\ref{inter3})$ then integrating from $0$ to $d$, \ we get%
An integration by parts yields that 
\begin{align*}
0=& \int_{0}^{h}\left( \phi _{s}(\phi _{k}^{\prime \prime }-\alpha
_{k}^{2}\phi _{k}-\frac{U^{\prime \prime }}{U-c_{k}}\phi _{k})-\phi
_{k}(\phi _{s}^{\prime \prime }-\alpha _{j}^{2}\phi _{s}-\frac{U^{\prime
\prime }}{U-U_{s}}\phi _{s})\right) dy \\
=& (\alpha _{j}^{2}-\alpha _{k}^{2})\int_{0}^{h}\phi _{s}\phi
_{k}\,dy-(c_{k}-U_{s})\int_{0}^{h}\frac{U^{\prime \prime }}{%
(U-c_{k})(U-U_{s})}\phi _{s}\phi _{k}dy+g_{r}(c_{k})-g_{r}(U_{s}).
\end{align*}%
Let us denote 
\begin{align*}
B_{k}& =\int_{0}^{h}\phi _{s}\phi _{k}dy, \\
D_{k}& =-\int_{0}^{h}\frac{U^{\prime \prime }}{(U-c_{k})(U-U_{s})}\phi
_{s}\phi _{k}dy+\frac{g_{r}(c_{k})-g_{r}(U_{s})}{c_{k}-U_{s}}.
\end{align*}%
It is immediate that $\lim_{k\rightarrow \infty }B_{k}=\int_{0}^{h}|\phi
_{s}|^{2}dy$. We shall show in the appendix that 
\begin{equation}
\lim_{k\rightarrow \infty }D_{k}=A+i\pi \sum_{j=1}^{m_{s}}\frac{K(y_{j})}{%
|U(y_{j})|}\phi _{s}^{2}(y_{j}),  \label{integral-limit 2}
\end{equation}%
where $A$ is defined by (\ref{definition-A}) and $a_{j}$ $\left( j=1,2,\dots
,m_{s}\right) $ are inflection points corresponding to the inflection value $%
U_{s}$. Since $\mathrm{Im}\,(\lim_{k\rightarrow \infty }D_{k})>0$ (see the
proof of Proposition \ref{theorem-unstable}) it follows that 
%\[\lim_{k\rightarrow\infty}\frac{D_{k}}{B_{k}}=\frac{\int_{0}^h|\phi_{s}|^{2}dy}{
%A+i\pi\sum_{k=1}^{l} | U'(a_k)|^{-1}K(a_k)\phi_{s}^{2}(a_k)}\]%
\begin{equation*}
{\rm Im}c_{k}=(\alpha _{k}^{2}-\alpha _{j}^{2}){\rm Im}(B_{k}/D_{k})<0
\end{equation*}%
for $k$ large. A contradiction proves that $\{\mathrm{Im}\,c_{k}\}$ has a
positive lower bound, uniformly for $k$.

Case 2: $U(h)=U_{s}$. We normalize $\phi_{k}$ so that $\phi_{k}^{%
\prime}(h)=1 $ and $\phi_{k}(h) =g_{s}(c_{k})$. The proof is identically the
same as that of Case 1 except that 
\begin{equation*}
D_{k}=-\int_{0}^{h}\frac{U^{\prime\prime}}{(U-c_{k})(U-U_{s})}\phi_{s}
\phi_{k} \, dy +\frac{g_{s}(c_{k})}{c_{k}-U_{s}}.
\end{equation*}
We shall show in the appendix that 
%Similarly, we can show $c_{k}\rightarrow U_{s}$ and $\phi_{k}\rightarrow\phi_{s}$ in
%$C^{1}\left[  0,d\right]  ,$ where $\phi_{s}$ satisfies
%\[-\frac{d^{2}}{dy^{2}}\phi_{s}+\frac{U^{\prime\prime}}{U-U_{s}}\phi_{s}%
%=-\alpha_{i}^{2}\phi_{s}%\]
%with $\phi_{s}^{\prime}\left(  d\right)  =1$, $\phi_{s}\left(  d\right)=\phi_{s}\left(  0\right)  =0$. Then we have
%\[\frac{c_{k}-U_{s}}{\alpha_{k}^{2}-\alpha_{i}^{2}}=\frac{D_{k}}{B_{k}}%\] with
%\[D_{k}=\int_{0}^{d}\phi_{s}\phi_{k}dy,.\]
%We have\[\lim_{k\rightarrow\infty}D_{k}\rightarrow\int_{0}^{d}\left\vert \phi_{s}\right\vert ^{2}dy.\]
%and as in appendix%
\begin{equation}
\lim_{k\rightarrow\infty}D_{k}=i\pi\sum_{j=1}^{m_{s}}\frac{K(y_{j})}{%
|U^{\prime}(y_{j})|}\phi_{s}^{2}(y_{j}) +p.v.\int_{0}^{h}\frac{K}{U-U_{s}}
\phi_{s}^{2} dy.  \label{formula-Bk-limit-Ud}
\end{equation}
%which implies the contradiction as in Case 1. So $\left(  \ref{lower bound}\right)  $ is proved.
This proves that there exists an unstable solution for each $\alpha\in(0,
\alpha_{\mathrm{max}})$.

It remains to prove linear stability in case either the operator $-\frac{%
d^{2}}{dy^{2}}-K(y)$ on $y\in (0,h)$ with (\ref{bc-sturm1})--(\ref{bc-sturm2}%
) is nonnegative or $\alpha \geq \alpha _{\max }$. Suppose otherwise, there
exists an unstable mode at a wave number $\alpha \geq \alpha _{\max }$. By
Lemma \ref{basis}, we can continuate this unstable mode for wave numbers
larger than $\alpha $ until the growth rate becomes zero. By Lemma \ref%
{lemma-bound-wave number}, this continuation must stop at a wave number $%
\alpha _{s}>\alpha $, where there is a neutral limiting mode. Then by
Proposition \ref{prop-neutral-mode}, $-\alpha _{s}^{2}$ is a negative
eigenvalue of $-\frac{d^{2}}{dy^{2}}-K\left( y\right) \ $with (\ref%
{bc-sturm1})-(\ref{bc-sturm2}). But $-\alpha _{s}^{2}<-\alpha _{\max }^{2}$
which is a contradiction to the fact that $-\alpha _{\max }^{2}$ is the
lowest eigenvalue of $-\frac{d^{2}}{dy^{2}}-K\left( y\right) \ $with (\ref%
{bc-sturm1})-(\ref{bc-sturm2}). This completes the proof.
\end{proof}

\begin{remark}
\label{R:destable} For $U\in\mathcal{K}^{+}$, let $-\alpha_{d}^{2}$ be the
lowest eigenvalue of $-\frac{d^{2}}{dy^{2}}-K(y)$ on $y\in(0,h)$ with the
Dirichlet boundary conditions $\phi(h)=0=\phi(0)$. If $U(h)=U_{s}$ then $%
\alpha_{\mathrm{max}}=\alpha_{d}$. We claim that if $U(h)\neq U_{s}$ then $%
\alpha_{\mathrm{max}}>\alpha_{d}$. To see this, let $\phi_{d}$ be the
eigenfunction of $-\frac{d^{2}}{dy^{2}}-K(y)$ on $y\in(0,h)$ with the
Dirichlet boundary conditions corresponding to $-\alpha_{d}^{2}$ and let $%
\phi_{m}$ be the eigenfunction of $-\frac{d^{2}}{dy^{2}}-K(y)$ with the
boundary conditions (\ref{bc-sturm1})--(\ref{bc-sturm2}) corresponding to $%
-\alpha_{\mathrm{max}}^{2}$. By Sturm's theory, we can assume $%
\phi_{d},\phi_{m}>0$ on $y\in(0,h)$. An integration by parts yields that 
\begin{align*}
0 & =\int_{0}^{h}\left( \phi_{d}(\phi_{m}^{\prime\prime}-\alpha _{\mathrm{max%
}}^{2}\phi_{m}+K(y)\phi_{m})-\phi_{m}(\phi_{d}^{\prime\prime
}-\alpha_{d}^{2}\phi_{d}+K(y)\phi_{d})\right) dy \\
& =-\phi_{m}(h)\phi_{d}^{\prime}(h)+(\alpha_{d}^{2}-\alpha_{\mathrm{max}%
}^{2})\int_{0}^{h}\phi_{d}\phi_{m}dy.
\end{align*}
Since $\phi_{d}^{\prime}(h)<0$ the claim follows.

For flows in class $\mathcal{K}^{+}$, the lowest eigenvalue of (\ref{E:SL})
measures the range of instability, for both the free surface and the rigid
wall cases. That means, $(0,\alpha_{\mathrm{max}})$ is the interval of
unstable wave numbers in the free-surface setting (Theorem \ref{class-k})
and $(0,\alpha_{d})$ is the interval of unstable wave numbers in the
rigid-wall setting \cite[Theorem 1.2]{lin1}. The fact that $\alpha_{\mathrm{%
max}}>\alpha_{d}$ thus indicates that the free surface has a \emph{%
destabilizing} effect.
\end{remark}

\subsection{Monotone unstable shear flows}

In general, for a given shear flow profile in the class $\mathcal{K}^{+}$,
one might show the existence of a neutral limiting mode and thus the
existence of growing modes, by numerically computing the negativity of (\ref%
{E:amax}). For monotone flows with one inflection point, however, the
existence of neutral limiting modes can be shown (\cite{yih72}) by a
comparison argument.

\begin{lemma}
\label{lemma-mode-monotone} For any monotone shear flow $U\left( y\right) \ $%
with exactly one inflection point $y_{s}$ in the interior, there exists a
neutral limiting mode. That is, (\ref{rayleigh})--(\ref{bc-rayleigh}) has a
nontrivial solution for which $c=U(y_{s})$ and $\alpha >0$.
\end{lemma}

\begin{proof}
This result is given in Theorem 4 of \cite{yih72}. Here we present a
detailed proof for completeness and also for clarification of some arguments
in \cite{yih72}.

We consider an increasing flow $U(y)$ only. A decreasing flow can be treated
in the same way. Let $U_{s}=U(y_{s})$ be the inflection value. Denoted by $%
\phi _{\alpha }$ the solution of the Rayleigh equation 
\begin{equation*}
\phi _{\alpha }^{\prime \prime }+\left( \frac{U^{\prime \prime }}{U_{s}-U}%
-\alpha ^{2}\right) \phi _{\alpha }=0\qquad \text{for}\quad y\in (0,h)
\end{equation*}%
with $\phi _{\alpha }(0)=0$ and $\phi _{\alpha }^{\prime }(0)=1$. As in the
proof of Lemma \ref{lemma-bifur}, an integration of the above over $(0,h)$
yields that 
\begin{equation*}
(U(h)-U_{s})\phi _{\alpha }^{\prime }(h)-(U(0)-U_{s})-\phi _{\alpha
}(h)U^{\prime }(h)-\alpha ^{2}\int_{0}^{h}(U-U_{s})\phi _{\alpha }dy=0,
\end{equation*}%
and thus 
\begin{equation*}
\frac{\phi _{\alpha }^{\prime }(h)}{\phi _{\alpha }(h)}=\frac{U(0)-U_{s}}{%
(U(h)-U_{s})\phi _{\alpha }(h)}+\frac{U^{\prime }(h)}{U(h)-U_{s}}+\frac{%
\alpha ^{2}}{(U(h)-U_{s})\phi _{\alpha }(h)}\int_{0}^{h}(U-U_{s})\phi
_{\alpha }dy.
\end{equation*}%
It is straightforward to see that the boundary condition (\ref{bc-rayleigh})
is satisfied if and only if the function 
\begin{equation*}
f(\alpha )=\frac{U(0)-U_{s}}{(U(h)-U_{s})\phi _{\alpha }(h)}+\frac{\alpha
^{2}}{(U(h)-U_{s})\phi _{\alpha }(h)}\int_{0}^{h}(U-U_{s})\phi _{\alpha }dy-%
\frac{g}{(U_{s}-U(h))^{2}}.
\end{equation*}%
vanishes at some $\alpha >0$.

We claim that $\phi _{\alpha }(y)>0$ on $y\in (0,h]$ for any $\alpha \geq 0$%
. Suppose otherwise, let $y_{\alpha }\in (0,h]$ to be the first zero of $%
\phi _{\alpha }$ other than $0$, that is, $\phi _{\alpha }(y_{\alpha })=0$
and $\phi _{\alpha }(y)>0$ for $y\in (0,y_{\alpha })$. Then, $y_{\alpha
}>y_{s}$ must hold. Indeed, if $y_{\alpha }\leq y_{s}$ were to be true, then
Sturm's first comparison theorem would apply to $\phi _{\alpha }$ and $%
U-U_{s}$ on $[0,y_{\alpha }]$ to assert that $U-U_{s}$ must vanish somewhere
in $(0,y_{\alpha })\subset (0,y_{s})$. This uses that 
\begin{equation*}
(U-U_{s})^{\prime \prime }+\frac{U^{\prime \prime }}{U_{s}-U}(U-U_{s})=0.
\end{equation*}%
A contradiction then asserts that $y_{\alpha }>y_{s}$. Correspondingly, $%
\phi _{\alpha }$ and $U-U_{s}$ have exactly one zero in $[0,y_{\alpha }]$.
On the other hand, by Sturm's second comparison theorem \cite{hartman}, it
follows that 
\begin{equation*}
\frac{\phi _{\alpha }^{\prime }(y_{\alpha })}{\phi _{\alpha }(y_{\alpha })}%
\geq \frac{U^{\prime }(y_{\alpha })}{U(y_{\alpha })-U_{s}}.
\end{equation*}%
This contradicts since $\phi _{\alpha }(y_{\alpha })=0$ and the left hand
side is $-\infty $. Therefore, $\phi _{\alpha }(y)>0$ for $y\in (0,h]$ and
for any $\alpha \geq 0$. In particular, $\phi _{\alpha }(h)>0$ for any $%
\alpha \geq 0$. Consequently, $f$ is a continuous function of $\alpha $ and $%
f(0)<0$.

It remains to show that $f(\alpha )>0$ for $\alpha >0$ big enough. Thereby,
by continuity $f$ vanishes at some $\alpha >0$. Let $\left\vert \frac{%
U^{\prime \prime }}{U_{s}-U}\right\vert \leq M$ and $\alpha ^{2}>M$. Let us
denote by $\phi _{1}$ and $\phi _{2}$ the solutions of 
\begin{equation*}
\phi _{1}^{\prime \prime }+(M-\alpha ^{2})\phi _{1}=0\quad \text{and}\quad
\phi _{2}^{\prime \prime }+(-M-\alpha ^{2})\phi _{2}=0\qquad \text{for}\quad
y\in (0,h),
\end{equation*}%
respectively, with $\phi _{i}(0)=0$ and $\phi _{i}^{\prime }(0)=1$. It is
straightforward that 
\begin{equation*}
\phi _{1}(y)=\frac{1}{\sqrt{\alpha ^{2}-M}}\sinh \sqrt{\alpha ^{2}-M}y\quad 
\text{and}\quad \phi _{2}(y)=\frac{1}{\sqrt{\alpha ^{2}+M}}\sinh \sqrt{%
\alpha ^{2}+M}y.
\end{equation*}%
As in the proof of Lemma \ref{lemma-dispersion-property} (a), Sturm's second
comparison theorem \cite{hartman} implies that 
\begin{equation*}
\frac{1}{\sqrt{\alpha ^{2}-M}}\sinh \sqrt{\alpha ^{2}-M}y\leq \phi _{\alpha
}(y)\leq \frac{1}{\sqrt{\alpha ^{2}+M}}\sinh \sqrt{\alpha ^{2}+M}y.
\end{equation*}%
This together with the monotone property of $U$ establishes that $f(\alpha
)\geq C_{1}\alpha -C_{2}$ for some constants $C_{1},C_{2}>0$. For details we
refer to \cite[Theorem 4]{yih72}. Thus, $f(\alpha )>0$ if $\alpha >0$ is
sufficiently large. This completes the proof .
\end{proof}

Since a monotone flow with one inflection value is in class $\mathcal{K}^{+}$%
, the above lemma combined with Theorem \ref{class-k} asserts its
instability.

\begin{corollary}
\label{cor-unsta-mono}Any monotone shear flow with exactly one inflection
point in the interior is unstable in the free-surface setting, for wave
numbers in an interval $(0,\alpha _{\max })$ with $\alpha _{\max }>0$.
\end{corollary}

\begin{remark}
In the free-surface setting, there are two different kinds of neutral modes,
which are solutions to the Rayleigh system (\ref{rayleigh})--(\ref%
{bc-rayleigh}) with ${\rm Im}c=0$. Neutral limiting modes have their phase
speed in the range of the shear profile. For flows in class $\mathcal{K}^{+}$%
, moreover, the phase speed of a neutral limiting mode must be the
inflection value of the shear profile (Proposition \ref{prop-neutral-mode})
and it is contiguous to unstable modes (Proposition \ref{theorem-unstable}).
On the other hand, Lemma \ref{lemma-bifur} shows that under the conditions $%
U^{\prime \prime }(h)<0$ and $U(h)>U(y)$ for $h\neq y,$ a neutral mode
exists with the phase speed $c>\max U$. Such a neutral mode is used in the
local bifurcation of nontrivial periodic waves in Theorems \ref{T:smallE}
and \ref{T:bifur-fixed U}. In view of the semicircle theorem (\ref%
{semicircle}), however, such a neutral mode is not contiguous to unstable
modes. This implies that neutral modes governing the stability property are
different from those governing the bifurcation of nontrivial waves. This is
an important difference in the free-surface setting. Since in the rigid wall
setting, for any possible neutral modes the phase speed must lie in the
range of $U$, which follows easily from Sturm's first comparison theorem.
For class $\mathcal{K}^{+}$ flows, such neutral modes are contiguous to
unstable modes (\cite{lin1}) and the bifurcation of nontrivial waves from
these neutral modes can also be shown (\cite{bd}). Thus, in the rigid-wall
setting, the same neutral modes govern both stability and bifurcation.
\end{remark}

\begin{remark}
\label{remark-neutral} In \cite{cost1}, the $\mathcal{J}$-formal stability
was introduced via a quadratic form, which is related to the local
bifurcation of nontrivial waves in the transformed variables (see also
Introduction), and it was concluded that this formal stability of the
trivial solutions switches exactly at the bifurcation point. Below we
discuss two examples for which the linear stability property does not change
along the line of trivial solutions passing the bifurcation point, which
indicates that the $\mathcal{J}$-formal stability in \cite{cost1} is
unrelated to linear stability of the physical water wave problem. However,
the $\mathcal{J}$-formal stability results \cite{cost1} do give more
information about the structure of the periodic water wave branch.

Let us consider a monotone increasing flow $U(y)$ on $y\in \lbrack 0,h]$
with one inflection point $y_{s}\in (0,h)$, for example, $U(y)=a\sin
b(y-h/2) $ on $y\in \lbrack 0,h]$ for which $y_{s}=h/2$. By lemma \ref%
{lemma-mode-monotone} and Corollary \ref{cor-unsta-mono}, such a shear flow
is equipped with a neutral limiting mode with $c=U(y_{s})$ and $\alpha
=\alpha _{\max }>0$, and it is linearly unstable for any wave number $\alpha
\in (0,\alpha _{\max })$. In addition, by lemma \ref{lemma-bifur} and
Theorem \ref{T:bifur-fixed U} for any wave number $\alpha \in (0,\alpha
_{\max })$ this shear flow has a neutral mode with $c(\alpha )>U(h)=\max U$.
Moreover, such a neutral mode is a nontrivial solution to the bifurcation
equation (\ref{rayleigh-disper})--(\ref{bc-disper}), and thus there exists a
local curve of bifurcation of nontrivial waves with a wave speed $c(\alpha )$
and period $2\pi /\alpha $. Let $p_{0}$ and $\gamma $ be the flux and
vorticity relation determined by $U\left( y\right) ,\ c(\alpha )$ and $h$
via (\ref{eqn-flux-vorticity}). Consider the trivial solutions with shear
flows $U(y;\mu )$ defined in Lemma \ref{L:trivial}, with above $p_{0},\
\gamma $ and the parameter $\mu $. The bifurcation point $U(y;\mu
_{0})=U\left( y\right) -c(\alpha )\ $corresponds to $\mu _{0}=(U(h)-c(\alpha
))^{2}$. The instability of $U(y;\mu _{0})$ at the wave number $\alpha $ is
continuated to shear flows $U(y;\mu )\ $with $\mu $ near $\mu _{0}$, which
can be shown by a similar argument as in the proof of Lemma \ref{basis}. So
at the bifurcation point $\mu _{0}$, there is \emph{NO }switch of stability
of trivial solutions.

Let us consider $U\in C^{2}([0,h])$ satisfying that $U^{\prime }(y)>0$ and $%
U^{\prime \prime }(y)<0 $ in $y\in \lbrack 0,h]$. For such a shear flow
Lemma \ref{lemma-bifur} and Theorem \ref{T:bifur-fixed U} applies as well
and there exists a local curve of bifurcation of nontrivial waves for any
wave number $\alpha $ which travel at the speed $c(\alpha )>\max U$, where $%
c(\alpha )$ is chosen so that the bifurcation equation (\ref{rayleigh-disper}%
)--(\ref{bc-disper}) is solvable. For this bifurcation flow $U\left(
y\right) -c(\alpha )$, the vorticity relation $\gamma \ $determined via (\ref%
{eqn-flux-vorticity}) is monotone since $U^{\prime \prime }$ does not change
sign. Consequently, any shear flows $U(y;\mu )$ defined in Lemma \ref%
{L:trivial} with the same $\gamma $ has no inflection points, as also
remarked at the end of Section 3. Therefore, by Theorem \ref{T:stable} all
trivial solutions corresponding to these shear flows $U(y;\mu )$ are stable.
This again shows that the bifurcation of nontrivial periodic waves does not
involve the switch of stability of trivial solutions.
\end{remark}

\section{Linear instability of periodic water waves with free surface}

\label{S:instability0}We now turn to investigating the linear instability of
periodic traveling waves near an unstable background shear flow. Suppose
that a shear flow $(U(y),0)$ with $U\in C^{2+\beta }([0,h_{0}])$ and $U\in 
\mathcal{K}^{+}$ has an unstable wave number $\alpha >0$, that is, for such
a wave number $\alpha >0$ the Rayleigh system (\ref{rayleigh})--(\ref%
{bc-rayleigh}) has a nontrivial solution with ${\rm Im}c>0$. Suppose
moreover that for the unstable wave number $\alpha >0$ the bifurcation
equation (\ref{rayleigh-disper})--(\ref{bc-disper}) is solvable with some $%
c(\alpha )>\max U$. Then Remark \ref{remark-neutral-bifur} and Theorem \ref%
{T:smallE} apply to state that there exists a one-parameter curve of
small-amplitude traveling-wave solutions $(\eta _{\epsilon }(x),\psi
_{\epsilon }(x,y))$ satisfying (\ref{stream}) with the period $2\pi /\alpha $
and the wave speed $c(\alpha )$, where $\epsilon \geq 0$ is the amplitude
parameter. A natural question is: are these small-amplitude nontrivial
periodic waves generated over the unstable shear flow also unstable? The
answer is YES under some technical assumptions, which is the subject of the
forthcoming investigation.

\subsection{The main theorem and examples}

We prove the linear instability of the steady periodic water-waves $(\eta
_{\epsilon }(x),\psi _{\epsilon }(x,y))$ by finding a growing-mode solution
to the linearized water-wave problem. As in Section \ref{S:formulation}, let 
\begin{equation*}
\mathcal{D}_{\epsilon }=\{(x,y):0<x<2\pi /\alpha ,\,0<y<\eta _{\epsilon
}(x)\}\quad \text{and}\quad \mathcal{S}_{\epsilon }=\{(x,\eta _{\epsilon
}(x)):0<x<2\pi /\alpha \}
\end{equation*}%
denote, respectively, the fluid domain of the steady wave $(\eta _{\epsilon
}(x),\psi _{\epsilon }(x,y))$ of one period and the steady surface. The
growing-mode problem (\ref{growing mode}) of the linearized periodic
water-wave problem around $(\eta _{\epsilon }(x),\psi _{\epsilon }(x,y))$
reduces to 
\begin{subequations}
\label{E:growing}
\begin{gather}
\Delta \psi +\gamma ^{\prime }(\psi _{\epsilon })\psi -\gamma ^{\prime
}(\psi _{\epsilon })\int_{-\infty }^{0}\lambda e^{\lambda s}\psi
(X_{\epsilon }(s),Y_{\epsilon }(s))ds=0\quad \text{in $\mathcal{D}_{\epsilon
}$};  \label{g-vor} \\
\lambda \eta (x)+\frac{d}{dx}\big(\psi _{\epsilon y}(x,\eta _{\epsilon
}(x))\eta (x)\big)=-\frac{d}{dx}\psi (x,\eta _{\epsilon }(x));  \label{g-eta}
\\
P(x,\eta _{\epsilon }(x))+P_{\epsilon y}(x)\eta (x)=0;  \label{g-P} \\
\lambda \psi _{n}(x)+\frac{d}{dx}\big(\psi _{\epsilon y}(x,\eta _{\epsilon
}(x))\psi _{n}(x)\big)=-\frac{d}{dx}P(x,\eta _{\epsilon }(x))-\Omega \frac{d%
}{dx}\psi (x,\eta _{\epsilon }(x));  \label{g-phi-n} \\
\psi (x,0)=0.  \label{g-Bottom}
\end{gather}%
Here and in sequel, let us abuse notation and denote that $P_{\epsilon
y}(x)=P_{\epsilon y}(x,\eta _{\epsilon }(x))$, that is, the restriction of $%
P_{\epsilon y}(x,y)$ on the steady wave-profile $y=\eta _{\epsilon }(x)$. By
Theorem \ref{T:smallE}, $P_{\epsilon y}(x)=P_{\epsilon y}(x,\eta _{\epsilon
}(x))=-g+O(\epsilon )$. Recall that 
\end{subequations}
\begin{equation*}
\psi _{n}(x)=\partial _{y}\psi (x,\eta _{\epsilon }(x))-\eta _{\epsilon
x}(x)\partial _{x}\psi (x,\eta _{\epsilon }(x))
\end{equation*}%
is the derivative of $\psi (x,\eta _{\epsilon }(x))$ in the direction normal
to the free surface $(x,\eta _{\epsilon }(x))$ and that $\Omega =\gamma (0)$
is the vorticity of the steady flow of $\psi _{\epsilon }(x,y)$ on the
steady wave-profile $y=\eta _{\epsilon }(x)$. Note that $\Omega $ is a
constant independent of $\epsilon $.

\begin{theorem}[Linear instability of small-amplitude periodic water-waves]
\label{T:unstable} Let the shear flow $U\left( y\right) \in C^{2+\beta
}([0,h_{0}])$, $\beta \in (0,1)$, be in class $\mathcal{K}^{+}$. Suppose
that $U(h_{0})\neq U_{s}$, where $U_{s}$ is the inflection value of $U$, and
that $\alpha _{\mathrm{max}}$ defined by (\ref{E:amax1}) is positive, as
such Theorem \ref{class-k} applies to find the interval of unstable wave
numbers $(0,\alpha _{\mathrm{max}})$. Suppose moreover that for some $\alpha
\in (\alpha _{\mathrm{max}}/2,\alpha _{\mathrm{max}})$ there exists $%
c(\alpha )>\max U$ such that the bifurcation equation (\ref{rayleigh-disper}%
)-- (\ref{bc-disper}) has a nontrivial solution. Let us denote by $(\eta
_{\epsilon }(x),\psi _{\epsilon }(x,y))$ the family of nontrivial waves with
the period $2\pi /\alpha $ and the wave speed $c(\alpha )$, bifurcating from
the trivial solution $\eta _{0}(x)\equiv h_{0}$ and $(\psi _{0y}(y),-\psi
_{0x}(y))=U(y)-c(\alpha ),0)$, where $\epsilon \geq 0$ is the amplitude
parameter. Provided that 
\begin{equation}
g+U^{\prime }(h_{0})(U(h_{0})-U_{s})>0,  \label{C:technical}
\end{equation}%
then for each $\epsilon >0$ sufficiently small, there exists an
exponentially growing solution $(e^{\lambda t}\eta (x),e^{\lambda t}\psi
(x,y))$ of the linearized system (\ref{eqn-L}), where ${\rm Re}\lambda >0$,
with the regularity property 
\begin{equation*}
(\eta (x),\psi (x,y))\in C^{2+\beta }([0,2\pi /\alpha ])\times C^{2+\beta }(%
\mathcal{\bar{D}}_{\epsilon }).
\end{equation*}
\end{theorem}

\begin{remark}[Examples]
\label{R:example} As is discussed in Remark \ref{remark-neutral}, any
increasing flow shear flow $U\in C^{2+\beta }([0,h_{0}])$, $\beta \in (0,1)$%
, with exactly one inflection point in $y\in (0,h_{0})$ satisfies $\alpha _{%
\mathrm{max}}>0$. Moreover, Lemma \ref{lemma-bifur} applies and
small-amplitude periodic waves bifurcate at any wave number $\alpha >0$.
Since (\ref{C:technical}) holds true, therefore, by Theorem \ref{T:unstable}
small-amplitude periodic waves bifurcating from such a shear flow at any
wave number $\alpha \in (\alpha _{\mathrm{max}}/2,\alpha _{\mathrm{max}})$
are unstable. Below, we discuss in details such an example: 
\begin{equation}
U(y)=a\sin b(y-h_{0}/2)\qquad \text{for}\quad y\in \lbrack 0,h_{0}],
\label{E:example}
\end{equation}%
where $h_{0},b>0$ satisfy $h_{0}b\leq \pi $ and $a>0$ is arbitrary.

(1) The shear flow in (\ref{E:example}) is unstable under periodic
perturbations of a wave number $\alpha \in (0,\alpha _{\mathrm{max}})$,
where $\alpha _{\mathrm{max}}>0$. Note that in the rigid-wall setting \cite%
{lin1}, the same shear flow is stable under perturbations of any wave
number. This indicates that the free surface has a destabilizing effect.
This serves as an example of Remark \ref{R:destable} since $\alpha _{\max
}>0 $ and $\alpha _{d}=0$.

(2) The amplitude $a$ and the depth $h_{0}$ in (\ref{E:example}) may be
chosen arbitrarily small, and the shear flow as well as the nontrivial
periodic waves near the shear flow are unstable for any wave number $\alpha
\in (0,\alpha _{\mathrm{max}})$, which contrasts with the result in \cite%
{cost1} that small-amplitude rotational periodic water-waves are $\mathcal{J}
$-formally stable if the vorticity strength and the depth are sufficiently
small. Thus, as also commented in Remark (\ref{remark-neutral}), the $%
\mathcal{J}$-formal stability in \cite{cost1} is not directly related to the
linear stability of water waves. Indeed, while $\partial \mathcal{J}(\eta
,\psi )=0$ gives the equations for steady steady waves, the linearized
water-wave problem is not in the form 
\begin{equation*}
\partial _{t}(\eta ,\psi )=(\partial ^{2}\mathcal{J})(\eta ,\psi ),
\end{equation*}%
which is implicitly required in \cite{cost1} in order to apply the
Crandall-Rabanowitz theory \cite{cr-ra-stability} of the exchange of
stability.

(3) Our example (\ref{E:example}) also indicates that adding an arbitrarily
small vorticity to the irrotational water wave system of an arbitrary depth
may induce instability. That means, although small irrotational periodic
waves are found to be stable under perturbations of the same period \cite%
{longet-78-super}, \cite{tanaka83}, they are not structurally stable;
Vorticity has a subtle influence on the stability of water waves.
\end{remark}

The proof of Theorem \ref{T:unstable} uses a perturbation argument. At $%
\epsilon =0$ the trivial solution $(\eta _{0}(x),\psi _{0}(x,y))$
corresponds to the shear flow $(U(y)-c(\alpha ),0)$ under the flat surface $%
\{y=h_{0}\}$. To simplify notations, in the remainder of this section, we
write $U(y)$ for $U(y)-c(\alpha )$, as is done in Section \ref%
{S:instability0}. Thereby, $U(y)<0$. Since $\alpha $ is an unstable wave
number of $U(y)$, there exist an unstable solution $\phi _{\alpha }$ to the
Rayleigh system (\ref{rayleigh})--(\ref{bc-rayleigh}) and an unstable phase
speed $c_{\alpha }$. That is, $\phi _{\alpha }\not\equiv 0$, $\mathrm{Im}%
\,c_{\alpha }>0$ and 
\begin{equation}
\begin{split}
& \phi _{\alpha }^{\prime \prime }-\alpha ^{2}\phi _{\alpha }+\frac{%
U^{\prime \prime }}{U-c_{\alpha }}\phi _{\alpha }=0\qquad \text{for}\quad
y\in (0,h_{0}), \\
& \phi _{\alpha }^{\prime }(h_{0})=\left( \frac{g}{(U(h_{0})-c_{\alpha })^{2}%
}+\frac{U^{\prime }(h_{0})}{U(h_{0})-c_{\alpha }}\right) \phi _{\alpha
}(h_{0}), \\
& \phi _{\alpha }(0)=0.
\end{split}
\label{E:alpha}
\end{equation}%
This corresponds to a growing mode solution satisfying (\ref{E:growing}) at $%
\epsilon =0$, where $\lambda _{0}=-i\alpha c_{\alpha }$ has a positive real
part. Our goal is to show that for $\epsilon >0$ sufficiently small, there
exists $\lambda _{\epsilon }$ near $\lambda _{0}$ such that the growing-mode
problem (\ref{E:growing}) at $(\eta _{\epsilon }(x),\psi _{\epsilon }(x,y))$
is solvable. First, the system (\ref{E:growing}) is reduced to an operator
equation defined in a function space independent of $\epsilon $. Then, by
showing the continuity of this operator with respect to the small-amplitude
parameter $\epsilon $, the continuation of the unstable mode follows from
the eigenvalue perturbation theory of operators.

\subsection{Reduction to an operator equation}

The purpose of this subsection is to reduce the growing mode system (\ref%
{E:growing}) to an operator equation on $L^{2}_{\text{per}}(\mathcal{S}%
_{\epsilon})$. Here and elsewhere the subscript \emph{per} denotes the
periodicity in the $x$-variable. The idea is to express $\eta(x)$ on $%
\mathcal{S}_{\epsilon}$ and $\psi(x,y)$ in $\mathcal{D}_{\epsilon}$ (and
hence $P(x, \eta_{\epsilon}(x))$) in terms of $\psi(x,\eta_{\epsilon}(x))$.

Our first task is to relate $\eta(x)$ with $\psi(x,\eta_{\epsilon}(x))$.

\begin{lemma}
\label{lemma-c-lb}For $|\lambda -\lambda _{0}|\leq ({\rm Re}\lambda _{0})/2$%
, where $\lambda _{0}=-i\alpha c_{\alpha }$, let us define the operator $%
\mathcal{C}^{\lambda }:L_{\text{per}}^{2}(\mathcal{S}_{\epsilon
})\rightarrow L_{\text{per}}^{2}(\mathcal{S}_{\epsilon })$ by%
\begin{equation}
\begin{split}
\mathcal{C}^{\lambda }\phi (x)& =-\frac{1}{\psi _{\epsilon y}(x)}\phi (x)+%
\frac{1}{\psi _{\epsilon y}(x)e^{\lambda a(x)}}\int_{0}^{x}\lambda
e^{\lambda a(x^{\prime })}\psi _{\epsilon y}^{-1}(x^{\prime })\phi
(x^{\prime })dx^{\prime } \\
& -\frac{\lambda }{\psi _{\epsilon y}(x)e^{\lambda a(x)}\left( 1-e^{\lambda
a(2\pi /\alpha )}\right) }\int_{0}^{2\pi /\alpha }e^{\lambda a(x^{\prime
})}\psi _{\epsilon y}^{-1}(x^{\prime })\phi (x^{\prime })dx^{\prime },
\end{split}
\label{defn-c-lambda}
\end{equation}%
where $a(x)=\int_{0}^{x}\psi _{\epsilon y}^{-1}(x^{\prime },\eta _{\epsilon
}(x^{\prime }))dx^{\prime }$. For simplicity, here and in the sequel we
identify $\psi _{ey}(x)$ with $\psi _{ey}(x,\eta _{\epsilon }(x))$ and $\phi
(x)$ with $\phi (x,\eta _{\epsilon }(x))$, etc. Then,

\textrm{{(a)} The operator $\mathcal{C}^{\lambda}$ is analytic in $\lambda$
for $|\lambda - \lambda_{0}| \leq ({\rm Re}\lambda_{0})/2$, and the
estimate 
\begin{equation*}
\left\Vert \mathcal{C}^{\lambda}\right\Vert _{L^{2}_{\text{per}}(\mathcal{S}%
_{\epsilon })\rightarrow L^{2}_{\text{per}}(\mathcal{S}_{\epsilon})}\leq K
\end{equation*}
holds, where $K>0$ is independent of $\lambda$ and $\epsilon$. }

\textrm{{(b)} For any $\phi\in L^{2}_{\text{per}}(\mathcal{S}_{\epsilon})$,
the function $\varphi=\mathcal{C}^{\lambda}\phi$ is the unique $L^{2}_{\text{%
per}}(\mathcal{S}_{\epsilon})$- weak solution of the first-order ordinary
differential equation 
\begin{equation}
\lambda\varphi+\frac{d}{dx}\left(\psi_{\epsilon y}(x)\varphi\right) =-\frac{d%
}{dx}\phi.  \label{eqn-inhomo}
\end{equation}
If, in addition, $\phi\in C^{1}_{\text{per}}(\mathcal{S}_{\epsilon})$ then $%
\varphi\in C^{1}_{\text{per}}(\mathcal{S}_{\epsilon})$ is the unique
classical solution of (\ref{eqn-inhomo}). }
\end{lemma}

\begin{proof}
Assertions of (a) follow immediately since $\psi _{\epsilon y}(x)<0$ and
thus $a(x)<0$ and since ${\rm Re}\lambda \geq {\rm Re}\lambda _{0}/2>0$.

(b) First, we consider the case $\phi\in C^{1}_{\text{per}}(\mathcal{S}%
_{\epsilon})$ to motivate the definition of $\mathcal{C}^{\lambda}$.

Let us write (\ref{eqn-inhomo}) as the first-order ordinary differential
equation 
\begin{equation*}
\frac{d}{dx}\varphi +\frac{1}{\psi _{\epsilon y}}\left( \lambda +\frac{d}{dx}%
\psi _{\epsilon y}\right) \varphi =-\frac{1}{\psi _{\epsilon y}}\frac{d}{dx}%
\phi ,
\end{equation*}%
which has an unique $2\pi /\alpha $-periodic solution 
\begin{align*}
\varphi (x)=-\frac{1}{\psi _{\epsilon y}(x)e^{\lambda a(x)}}\Big(&
\int_{0}^{x}e^{\lambda a(x^{\prime })}\frac{d}{dx}\phi (x^{\prime
})dx^{\prime } \\
& -\frac{1}{1-e^{\lambda a(2\pi /\alpha )}}\int_{0}^{2\pi /\alpha
}e^{\lambda a(x^{\prime })}\frac{d}{dx}\phi (x^{\prime })dx^{\prime }\Big).
\end{align*}%
An integration by parts of the above formula yields that $\varphi (x)=%
\mathcal{C}^{\lambda }\phi $ is as defined in (\ref{defn-c-lambda}).

In case $\phi \in L_{\text{per}}^{2}(\mathcal{S}_{\epsilon })$, the integral
representation (\ref{defn-c-lambda}) makes sense and solves (\ref{eqn-inhomo}%
) in the weak sense. Indeed, an integration by parts shows that $\varphi $
defined by (\ref{defn-c-lambda}) satisfies the weak form of equation (\ref%
{eqn-inhomo}) 
\begin{equation*}
\int_{0}^{2\pi /\alpha }\left( \lambda \varphi (x)h(x)-\psi _{\epsilon
y}(x)\varphi (x)\frac{d}{dx}h(x)-\phi (x)\frac{d}{dx}h(x)\right) dx=0
\end{equation*}%
for any $2\pi /\alpha $-periodic function $h\in H_{\text{per}}^{1}([0,2\pi
/\alpha ])$. In order to show the uniqueness, suppose that $\tilde{\varphi}%
\in L_{\text{per}}^{2}(\mathcal{S}_{\epsilon })$ is another weak solution of
(\ref{eqn-inhomo}). Let $\varphi _{1}=\varphi -\tilde{\varphi}$. Then, $%
\varphi _{1}\in L_{\text{per}}^{2}(\mathcal{S}_{\epsilon })$ is a weak
solution of the homogeneous differential equation 
\begin{equation*}
\lambda \varphi _{1}+\frac{d}{dx}\left( \psi _{\epsilon y}(x)\varphi
_{1}\right) =0.
\end{equation*}%
It is readily seen that $\int_{0}^{2\pi /\alpha }\varphi _{1}(x)dx=0$. Note
that $h_{1}(x)=\int_{0}^{x}\varphi _{1}(x^{\prime })dx^{\prime }$ defines a $%
2\pi /\alpha $-periodic function in $H_{\text{per}}^{1}([0,2\pi /\alpha ])$.
Multiplication of the above homogeneous equation by $(\lambda h_{1})^{\ast }$
and an integration by parts then yield that 
\begin{align*}
0& ={\rm Re}\int_{0}^{2\pi /\alpha }\left( \left\vert \lambda \right\vert
^{2}\varphi _{1}(x)h_{1}^{\ast }(x)-\psi _{\epsilon y}(x)\lambda ^{\ast
}\varphi _{1}(x)\left( \frac{d}{dx}h_{1}(x)\right) ^{\ast }\right) dx \\
& =\left\vert \lambda \right\vert ^{2}\int_{0}^{2\pi /\alpha }\frac{d}{dx}%
\left( \tfrac{1}{2}\left\vert h_{1}\right\vert ^{2}\right) dx-{\rm Re}%
\lambda \int_{0}^{2\pi /\alpha }\psi _{\epsilon y}(x)\left\vert \varphi
_{1}\right\vert ^{2}dx \\
& =-{\rm Re}\lambda \int_{0}^{2\pi /\alpha }\psi _{\epsilon y}(x)\left\vert
\varphi _{1}\right\vert ^{2}dx.
\end{align*}%
Here and elsewhere, the asterisk denotes the complex conjugation. Since $%
{\rm Re}\lambda >0$ and $\psi _{\epsilon y}(x)<0$, it follows that $\varphi
_{1}\equiv 0$, and in turn, $\varphi \equiv \tilde{\varphi}$. This completes
the proof.
\end{proof}

Formally, the operator $\mathcal{C}^{\lambda }$ can be written as%
\begin{equation*}
\mathcal{C}^{\lambda }\phi (x)=-\left( \lambda +\frac{d}{dx}\left( \psi
_{\epsilon y}(x)\phi (x)\right) \right) ^{-1}\frac{d}{dx}\phi (x).
\end{equation*}

Let us denote $f(x)=\psi(x,\eta_{\epsilon}(x))$. With the use of $\mathcal{C}%
^{\lambda}$ then the boundary conditions (\ref{E:growing}b)--(\ref{E:growing}%
d) are written in terms of $f$ as 
\begin{gather*}
\eta(x)=\mathcal{C}^{\lambda}f(x), \\
P(x,\eta_{\epsilon}(x))=-P_{\epsilon y}(x)\mathcal{C}^{\lambda}f(x), \\
\psi_{n}(x)=-\mathcal{C}^{\lambda}(P_{\epsilon y}(x)\mathcal{C}^{\lambda
}+\Omega \,id)f(x),
\end{gather*}
where $id:L^{2}(\mathcal{S}_{\epsilon})\rightarrow L^{2}(\mathcal{S}%
_{\epsilon })$ is the identity operator.

Our next task is to relate $\psi (x,y)$ in $\mathcal{D}_{\epsilon }$ with $%
f(x)=\psi (x,\eta _{\epsilon }(x))$. Given $b\in L_{\text{per}}^{2}(\mathcal{%
S}_{\epsilon })$, let $\psi _{b}\in H^{1}(\mathcal{D}_{\epsilon })$ be a
weak solution of the elliptic partial differential equation 
\begin{subequations}
\label{E:psib}
\begin{gather}
\Delta \psi +\gamma ^{\prime }(\psi _{\epsilon })\psi -\gamma ^{\prime
}(\psi _{\epsilon })\int_{-\infty }^{0}\lambda e^{\lambda s}\psi
(X_{\epsilon }(s),Y_{\epsilon }(s))ds=0\quad \text{in }\mathcal{D}_{\epsilon
}  \label{E:psib-a} \\
\psi _{n}(x):=\partial _{y}\psi (x,\eta _{\epsilon }(x))-\eta _{\epsilon
x}(x)\partial _{x}\psi (x,\eta _{\epsilon }(x))=b(x)\quad \text{on }\mathcal{%
S}_{\epsilon },  \label{E:psib-b} \\
\psi (x,0)=0  \label{E:psib-c}
\end{gather}%
such that $\psi _{b}$ is $2\pi /\alpha $-periodic in the $x$-variable. Lemma %
\ref{L:solv} below proves that the boundary value problem (\ref{E:psib}) is
uniquely solvable and $\psi _{b}\in H^{1}(\mathcal{D}_{\epsilon })$ provided
that $|\lambda -\lambda _{0}|\leq ({\rm Re}\lambda _{0})/2$ and $(\eta
_{\epsilon }(x),\psi _{\epsilon }(x,y))$ is near the trivial solution with
the flat surface $y=h_{0}$ and the unstable shear flow $(U(y),0)$ given in
Theorem \ref{T:unstable}. This, together with the trace theorem, allows us
to define an operator $\mathcal{T}_{\epsilon }:L_{\text{per}}^{2}(\mathcal{S}%
_{\epsilon })\rightarrow L_{\text{per}}^{2}(\mathcal{S}_{\epsilon })$ by 
\end{subequations}
\begin{equation}
\mathcal{T}_{\epsilon }b(x)=\psi _{b}(x,\eta _{\epsilon }(x)),  \label{E:T}
\end{equation}%
which is the unique solution $\psi _{b}$ of (\ref{E:psib}) restricted on the
steady surface $\mathcal{S}_{\epsilon }$. 
%Note that $\te :L^2(\se) \to L^2(\se)$.

Be definition, it follows that 
\begin{equation*}
f(x)=\psi(x,\eta_{\epsilon}(x))=\mathcal{T}_{\epsilon} \psi_{n}(x).
\end{equation*}
This, together with the boundary conditions written in terms of $f$ as above
yields that 
\begin{equation}
f=-\mathcal{T}_{\epsilon}\mathcal{C}^{\lambda}(P_{\epsilon y}(x)\mathcal{C}%
^{\lambda}+\Omega \, id)f.  \label{E:operator}
\end{equation}
The growing-mode problem (\ref{E:growing}) is thus reduced to find a
nontrivial solution $f(x)=\psi(x,\eta_{\epsilon}(x))\in L^{2}_{\text{per}}(%
\mathcal{S}_{\epsilon})$ of the equation (\ref{E:operator}), or
equivalently, to show that the operator 
\begin{equation*}
id +\mathcal{T}_{\epsilon}\mathcal{C}^{\lambda}(P_{\epsilon y}(x)\mathcal{C}%
^{\lambda}+\Omega \, id)
\end{equation*}
has a nontrivial kernel for some $\lambda\in\mathbb{C}$ with ${\rm Re}%
\,\lambda>0$.

The remainder of this subsection concerns with the unique solvability of (%
\ref{E:psib}). Our first task is to compare (\ref{E:psib}) at $\epsilon=0$
with the Rayleigh system, which is useful in later consideration.

\begin{lemma}
\label{L:alpha} For $U\left( y\right) $ in class $\mathcal{K}^{+}$, $y\in
\lbrack 0,h_{0}]$ and $U(h_{0})\neq U_{s}$, let $-\alpha _{\mathrm{max}}^{2}$
be the lowest eigenvalue of $-\frac{d^{2}}{dy^{2}}-K(y)$ for $y\in (0,h_{0})$
subject to the boundary conditions 
\begin{equation}
\phi ^{\prime }(h_{0})=\left( \frac{g}{(U(h_{0})-U_{s})^{2}}+\frac{U^{\prime
}(h_{0})}{U(h_{0})-U_{s}}\right) \phi (h_{0})\quad \text{and}\quad \phi
(0)=0,  \label{BC:g}
\end{equation}%
where $K$ is defined in (\ref{E:K}). Let $-\alpha _{n}^{2}$ be the lowest
eigenvalue of $-\frac{d^{2}}{dy^{2}}-K(y)$ for $y\in (0,h_{0})$ subject to
the boundary conditions 
\begin{equation}
\phi ^{\prime }(h_{0})=0\quad \text{and}\quad \phi (0)=0.  \label{BC:n}
\end{equation}%
If $g+U^{\prime }(h_{0})(U(h_{0})-U_{s})>0$, then $\alpha _{\mathrm{max}%
}>\alpha _{n}$.
\end{lemma}

\begin{proof}
The argument is nearly identical to that in Remark \ref{R:destable}. Let us
denote by $\phi _{m}$ the eigenfunction of $-\frac{d^{2}}{dy^{2}}-K(y)$ on $%
y\in (0,h_{0})$ with (\ref{BC:g}) corresponding to the eigenvalue $-\alpha _{%
\mathrm{max}}^{2}$ and by $\phi _{n}$ the eigenfunction of $-\frac{d^{2}}{%
dy^{2}}-K(y)$ on $y\in (0,h_{0})$ with (\ref{BC:n}) corresponding to the
eigenvalue $-\alpha _{n}^{2}$. By the standard theory of Sturm-Liouville
operators we can assume $\phi _{m}>0$ and $\phi _{n}>0$ on $y\in (0,h_{0})$.
An integration by parts yields that 
\begin{align*}
0=& \int_{0}^{h_{0}}\Big(\phi _{n}(\phi _{m}^{\prime \prime }-\alpha _{%
\mathrm{max}}^{2}\phi _{m}+K(y)\phi _{m})-\phi _{m}(\phi _{n}^{\prime \prime
}-\alpha _{n}^{2}\phi _{n}+K(y)\phi _{n})\Big)dy \\
=& (\alpha _{n}^{2}-\alpha _{\mathrm{max}}^{2})\int_{0}^{h_{0}}\phi _{n}\phi
_{m}dy+\frac{g+U^{\prime }(h_{0})(U(h_{0})-U_{s})}{(U(h_{0})-U_{s})^{2}}\phi
_{m}(h_{0})\phi _{n}(h_{0}).
\end{align*}%
The assumption $g+U^{\prime }(h_{0})(U(h_{0})-U_{s})>0$ then proves the
assertion.
\end{proof}

Our next task is the unique solvability of the homogeneous problem of (\ref%
{E:psib}).

\begin{lemma}
\label{L:solv0} Assume that $U(h_{0})\neq U_{s}$ and $g+U^{\prime
}(h_{0})(U(h_{0})-U_{s})>0$. For $\epsilon >0$ sufficiently small and $%
|\lambda -\lambda _{0}|\leq ({\rm Re}\lambda _{0})/2$, the following
elliptic partial differential equation 
\begin{subequations}
\label{E:psi0}
\begin{equation}
\Delta \psi +\gamma ^{\prime }(\psi _{\epsilon })\psi -\gamma ^{\prime
}(\psi _{\epsilon })\int_{-\infty }^{0}\lambda e^{\lambda s}\psi
(X_{\epsilon }(s),Y_{\epsilon }(s))ds=0\quad \text{in }\mathcal{D}_{\epsilon
}  \label{E:psi0-a}
\end{equation}%
subject to 
\begin{equation}
\psi (x,0)=0  \label{E:psi0-b}
\end{equation}%
and the Neumann boundary condition 
\begin{equation}
\psi _{n}=0\qquad \text{on}\quad \mathcal{S}_{\epsilon }  \label{E:psi0-c}
\end{equation}\end{subequations}
admits only the trivial solution $\psi \equiv 0$.
\end{lemma}

The main difficulty in the proof of Lemma \ref{L:solv0} is that the domain $%
D_{\epsilon }$ depends on the small amplitude parameter $\epsilon >0$
whereas the statement of Lemma \ref{L:solv0} calls for an estimate of
solutions of the system (\ref{E:psi0}) uniform for $\epsilon >0$. In order
to compare (\ref{E:psi0}) for different values of $\epsilon $, we employ the
action-angle variables, which map the domain $D_{\epsilon }$ into a common
domain independent of $\epsilon $. For any $(x,y)\in D_{\epsilon }$, let us
denote by $\{(x^{\prime },y^{\prime }):\psi _{\epsilon }(x^{\prime
},y^{\prime })=\psi _{\epsilon }(x,y)=p\}$ the streamline containing $(x,y)$%
, by $\sigma $ the arc-length variable on the streamline $\{\psi _{\epsilon
}(x^{\prime },y^{\prime })=p\}$, and by $\sigma (x,y)$ the value of $\sigma $
corresponding to the point $(x,y)$ along the streamline. Let us define the
normalized action-angle variables as

\begin{equation}
I=\frac{\alpha }{2\pi }\frac{h_{0}}{h_{\epsilon }}\iint_{\{\psi _{\epsilon
}(x^{\prime },y^{\prime })<p\}}dy^{\prime }dx^{\prime },\quad \text{and}%
\quad \theta =\upsilon _{\epsilon }(I)\int_{0}^{\sigma (x,y)}\frac{1}{%
|\nabla \psi _{\epsilon }|}\Big|_{\{\psi _{\epsilon }(x^{\prime }y^{\prime
})=p\}}d\sigma ^{\prime },  \label{E:aa}
\end{equation}%
where $h_{0}$ and $h_{\epsilon }$ are the mean water depth at the parameter
values $0$ and $\epsilon $, respectively, and 
\begin{equation*}
\upsilon _{\epsilon }(I)=\frac{2\pi }{\alpha }\left( \oint_{\{\psi
_{\epsilon }(x^{\prime },y^{\prime })=p\}}\frac{1}{|\nabla \psi _{\epsilon }|%
}\right) ^{-1}.
\end{equation*}%
The action variable $I$ represents the (normalized) area in the phase space
under the streamline $\{(x^{\prime },y^{\prime }):\psi _{\epsilon
}(x^{\prime },y^{\prime })=\psi _{\epsilon }(x,y)=p\}$ and the angle
variable $\theta $ represents the position along the streamline of $\psi
_{\epsilon }(x,y)$. The assumption of no stagnation, i.e. $\psi _{\epsilon
y}(x,y)<0$ throughout $\mathcal{D}_{\epsilon }$, implies that all stream
lines are non-closed. For, otherwise, the horizontal velocity $\psi
_{\epsilon y}$ must change signs on a closed streamline. Moreover, for $%
\epsilon >0$ sufficiently small, all streamlines are close to those of the
trivial flow, that is, they almost horizontal. Therefore, the action-angle
variable $(\theta ,I)$ is defined globally in $\mathcal{D}_{\epsilon }$. The
mean-zero property (\ref{mean-zero}$)$ of the wave profile $\eta _{\epsilon
}(x)$ implies that the area of the (steady) fluid region $\mathcal{D}%
_{\epsilon }$ is $(2\pi /\alpha )h_{\epsilon }$. Accordingly, by the
definition in (\ref{E:aa}) it follows that $0<I<h_{0}$ and $0<\theta <2\pi
/\alpha $, independently of $\epsilon $.

Let us define the mapping by the action-angle variables by $\mathcal{A}%
_{\epsilon}(x,y)=(\theta,I)$. From the above discussions follows that $%
\mathcal{A}_{\epsilon}$ is bijective and maps $\mathcal{D}_{\epsilon}$ to 
\begin{equation*}
D=\{(\theta,I):0<\theta<2\pi/\alpha,\,0<I<h_{0}\}.
\end{equation*}
At $\epsilon=0$, the action-angle mapping reduces to the identity mapping on 
$\mathcal{D}_{0}=D$. For $\epsilon>0$, the mapping has a scaling effect.
More precisely,\qquad\ 
\begin{equation*}
\iint_{D}f(\mathcal{A}_{\epsilon}^{-1}(\theta,I))d\theta dI=\frac{h_{0}}{%
h_{\epsilon}}\iint_{\mathcal{D}_{\epsilon}}f(x,y)dydx
\end{equation*}
for any function $f$ defined in $\mathcal{D}_{\epsilon}$.

Another motivation to employ the action-angle variables comes from that they
simplify the equation on the particle trajectory\footnote{%
In the irrotational setting, i.e., $\gamma \equiv 0$, a qualitative
description of particle trajectories is obtained \cite{con} by studying a
specific nonlinear boundary value problem for harmonic functions.}. In the
action-angle variables $\left( \theta ,I\right) $, the characteristic
equation (\ref{E:char}) becomes (\cite[Section 50]{arnold}, \cite[p. 94]%
{lin2}) 
\begin{equation*}
\begin{cases}
\dot{\theta}=-\upsilon _{\epsilon }(I) \\ 
\dot{I}=0,%
\end{cases}%
\end{equation*}%
where the dot above a variable denotes the differentiation in the $\sigma $%
-variable. This observation is very useful for future considerations.

\begin{proof}[Proof of Lemma \protect\ref{L:solv0}]
Suppose on the contrary that there would exist sequences $\epsilon
_{k}\rightarrow 0+$, $\lambda _{k}\rightarrow \lambda _{0}$ as $k\rightarrow
\infty $ and $\psi _{k}\in H^{1}(\mathcal{D}_{\epsilon _{k}})$ such that $%
\psi _{k}\not\equiv 0$ is a solution of (\ref{E:psi0}) with $\epsilon
=\epsilon _{k}$. After normalization, $\Vert \psi _{k}\Vert _{L^{2}(\mathcal{%
D}_{\epsilon _{k}})}=1$. We claim that 
\begin{equation}
\Vert \psi _{k}\Vert _{H^{2}(\mathcal{D}_{\epsilon _{k}})}\leq C,
\label{E:elliptic1}
\end{equation}%
where $C>0$ is independent of $k$. 
%Since the mapping $(x,y) \mapsto (X^\epsilon, Y^\epsilon)$ is measure-preserving,
Indeed, by Minkovski's inequality it follows that 
\begin{align}
& \left\Vert \gamma ^{\prime }(\psi _{\epsilon _{k}})\psi _{k}-\gamma
^{\prime }(\psi _{\epsilon _{k}})\int_{-\infty }^{0}\lambda e^{\lambda
s}\psi _{k}(X_{\epsilon _{k}}(s),Y_{\epsilon _{k}}(s))ds\right\Vert _{L^{2}(%
\mathcal{D}_{\epsilon _{k}})}  \notag \\
& \leq \Vert \gamma ^{\prime }(\psi _{\epsilon _{k}})\Vert _{L^{\infty
}}\left( \Vert \psi _{k}\Vert _{L^{2}(\mathcal{D}_{\epsilon
_{k}})}+\int_{-\infty }^{0}\left\vert \lambda \right\vert e^{{\rm Re}%
\lambda s}\left\Vert \psi _{k}(X_{\epsilon _{k}}(s),Y_{\epsilon
_{k}}(s))\right\Vert _{L^{2}(\mathcal{D}_{\epsilon _{k}})}ds\right)
\label{estimate-k-lb} \\
& =\Vert \gamma ^{\prime }(\psi _{\epsilon _{k}})\Vert _{L^{\infty }}\Vert
\psi _{k}\Vert _{L^{2}(\mathcal{D}_{\epsilon _{k}})}\left( 1+\frac{%
\left\vert \lambda \right\vert }{{\rm Re}\lambda }\right) \leq \Vert \gamma
^{\prime }(\psi _{\epsilon _{k}})\Vert _{L^{\infty }}\left( 1+\frac{%
\left\vert \lambda _{0}\right\vert +\delta }{\delta }\right) .  \notag
\end{align}%
This uses the fact that the mapping $(x,y)\rightarrow (X_{\epsilon
_{k}}(s),Y_{\epsilon _{k}}(s))$ is measure preserving and that $|\lambda
-\lambda _{0}|<{\rm Re}\lambda _{0}/2$. The standard elliptic regularity
theory \cite{ag-book} for a Neumann problem adapted for (\ref{E:psi0}) then
proves the estimate (\ref{E:elliptic1}). This proves the claim.

In order to study the convergence of $\{\psi_{k}\}$, we perform the mapping
by the action-angle variables (\ref{E:aa}) and write $\mathcal{A}%
_{\epsilon_{k}}(x,y)=(\theta,I)$. Note that the image of $\mathcal{D}%
_{\epsilon_{k}}$ under the mapping $\mathcal{A}_{\epsilon_{k}}$ is 
\begin{equation*}
D=\{(\theta,I):0<\theta<2\pi/\alpha\,,\,0<I<h_{0}\},
\end{equation*}
which is independent of $k$. Let us denote $A\psi_{k}(\theta,I)=\psi _{k}(%
\mathcal{A}_{\epsilon_{k}}^{-1}(\theta,I))$. It is immediate to see that $%
A\psi_{k}\in H^{2}(D)$.

Since $h_{\epsilon}=h_{0}+O(\epsilon)$ it follows that 
\begin{equation*}
\Vert
A\psi_{k}\Vert_{L^{2}(D)}=(h_{0}/h_{\epsilon})\Vert\psi_{k}\Vert_{L^{2}(%
\mathcal{D}_{\epsilon_{k}})}=1+O(\epsilon).
\end{equation*}
This, together with (\ref{E:elliptic1}), implies that 
\begin{alignat*}{2}
A\psi_{k} & \rightarrow\psi_{\infty}\quad\text{weakly in $H^{2}(D) $}\quad & 
\text{as $k\rightarrow\infty$}, \\
A\psi_{k} & \rightarrow\psi_{\infty}\quad\text{strongly in $L^{2}(D)$}\quad
& \text{ as $k\rightarrow\infty$}
\end{alignat*}
for some $\psi_{\infty}$. By continuity, $\Vert\psi_{\infty}%
\Vert_{L^{2}(D_{0})}=1$. Our goal is to show that $\psi_{\infty}\equiv0$ and
thus prove the assertion by contradiction.

At the limit as $\epsilon_{k}\rightarrow0,$ the limiting mapping $\mathcal{A}%
_{0}$ is the identity mapping on $D=\mathcal{D}_{0}$ and the limit function $%
\psi_{\infty}$ satisfies 
\begin{gather}
\Delta\psi_{\infty}+\gamma^{\prime}(\psi_{0})\psi_{\infty}-\gamma^{\prime
}(\psi_{0})\int_{-\infty}^{0}\lambda_{0}e^{\lambda_{0}s}\psi_{%
\infty}(X_{0}(s),Y_{0}(s))ds=0\qquad\text{in}\quad D;  \notag \\
\partial_{y}\psi_{\infty}(x,h_{0})=0;  \label{E:infty} \\
\psi_{\infty}(x,0)=0.  \notag
\end{gather}
Here, $\lambda_{0}=-i\alpha c_{\alpha}$ and $(X_{0}(s),Y_{0}(s))=(x+U(y)s,y)$%
.

Since the above equation and the boundary conditions are separable in the $x$
and $y$ variables, $\psi_{\infty}$ can be written as 
\begin{equation*}
\psi_{\infty}(x,y)=\sum_{l=0}^{\infty} e^{il\alpha x}\phi_{l}(y).
\end{equation*}

In case $l=0$, the boundary value problem (\ref{E:infty}) reduces to 
\begin{equation*}
\begin{cases}
\phi_{0}^{\prime\prime}=0\qquad\text{for}\quad y\in(0,h_{0}) \\ 
\phi_{0}^{\prime}(h_{0})=0,\ \ \ \phi_{0}(0)=0,%
\end{cases}%
\end{equation*}
and thus, $\phi_{0}\equiv0$.

Next, consider the solution $\phi _{1}$ of (\ref{E:infty}) when $l=1$: 
\begin{equation*}
\begin{cases}
\phi _{1}^{\prime \prime }-\alpha ^{2}\phi _{1}-{\frac{U^{\prime \prime }}{%
U-c_{\alpha }}}\phi _{1}=0\qquad \text{for}\quad y\in (0,h_{0}) \\ 
\phi _{1}^{\prime }(h_{0})=0,\ \ \ \phi _{1}(0)=0.%
\end{cases}%
\end{equation*}%
This uses that $\gamma ^{\prime }(\psi _{0})=-U^{\prime \prime }/U$. Recall
that for the unstable wave number $\alpha $ and the unstable wave speed $%
c_{\alpha }$, there exists an unstable solution $\phi _{\alpha }$ of the
Rayleigh system (\ref{E:alpha}). As is done in the proof of Lemma \ref%
{L:alpha}, an integration by parts yields that 
\begin{align*}
0=& \int_{0}^{h_{0}}\left( \phi _{1}\left( \phi _{\alpha }^{\prime \prime
}-\alpha ^{2}\phi _{\alpha }-\frac{U^{\prime \prime }}{U-c_{\alpha }}\phi
_{\alpha }\right) -\phi _{\alpha }\left( \phi _{1}^{\prime \prime }-\alpha
^{2}\phi _{1}-\frac{U^{\prime \prime }}{U-c_{\alpha }}\phi _{1}\right)
\right) dy \\
=& \left( \frac{g}{(U(h_{0})-c_{\alpha })^{2}}+\frac{U^{\prime }(h_{0})}{%
U(h_{0})-c_{\alpha }}\right) \phi _{\alpha }(h_{0})\phi _{1}(h_{0}).
\end{align*}%
%
%
%
%
%
%
%
%
%
%
%
%
%
%
%Since $$\frac{g}{(U(h)-c)^2}+\frac{U'(h)}{U(h)-c}=\frac{g+U'(h)(U(h)-c)}{(U(h)-c)^2} \neq 0,$$
Since $\phi _{\alpha }(h_{0})\neq 0$, we have $\phi _{1}(h_{0})=0$, which
together with $\phi _{1}^{\prime }(h_{0})=0$, implies that $\phi _{1}\equiv
0 $.

For $l\geq 2$, the solution $\phi _{l}$ of (\ref{E:infty}) ought to satisfy 
\begin{equation*}
\begin{cases}
\phi _{l}^{\prime \prime }-l^{2}\alpha ^{2}\phi _{l}-\frac{U^{\prime \prime }%
}{U-c_{\alpha }/l}\phi _{l}=0\qquad \text{for}\quad y\in (0,h_{0}) \\ 
\phi _{l}^{\prime }(h_{0})=0,\ \ \phi _{l}(0)=0.%
\end{cases}%
\end{equation*}%
%
%
%
%
%
%
%
%
%
%
%
%
%
%
%where $c_k=\tfrac{c}{k}-\tfrac{k+1}{k}c$. Suppose $\phi_l \not\equiv 0$.
An integration by parts yields that 
\begin{equation*}
\int_{0}^{h_{0}}\Big(\left\vert \phi _{l}^{\prime }\right\vert
^{2}+l^{2}\alpha ^{2}|\phi _{l}|^{2}+\frac{U^{\prime \prime }}{U-c_{\alpha
}/l}|\phi _{l}|^{2}\Big)dy=0,
\end{equation*}%
and subsequently, for any $q$ real it follows that (see the proof of Lemma %
\ref{lemma-h2-bound}) 
\begin{equation*}
\int_{0}^{h_{0}}\Big(\left\vert \phi _{l}^{\prime }\right\vert
^{2}+l^{2}\alpha ^{2}|\phi _{l}|^{2}+\frac{U^{\prime \prime }(U-q)}{%
|U-c_{\alpha }/l|^{2}}|\phi _{l}|^{2}\Big)dy=0.
\end{equation*}%
%
%
%
%
%
%
%
%
%
%
%
%
%
%
%whose real and imaginary parts would read
%\begin{gather*}
%\intertext{and}
%\im \ca/k \int^{h_0}_0\frac{U''}{|U- \ca/k|^2}|\phi_k|^2dy=0.
%\end{gather*}
The same argument as in proving Lemma \ref{lemma-h2-bound} applies to assert
that 
\begin{equation*}
\int_{0}^{h_{0}}(\left\vert \phi _{l}^{\prime }\right\vert ^{2}+l^{2}\alpha
^{2}|\phi _{l}|^{2})dy\leq \int_{0}^{h_{0}}K(y)|\phi _{l}|^{2}dy.
\end{equation*}%
Recall that $K(y)=-U^{\prime \prime }(y)/(U(y)-U_{s})>0$.

Let $-\alpha_{n}^{2}$ be as in Lemma \ref{L:alpha} the lowest eigenvalue of $%
\frac{d^{2}}{dy^{2}}-K(y)$ on $y\in(0,h_{0})$ with the boundary conditions $%
\phi^{\prime}(h_{0})=0$ and $\phi(0)=0$. By Lemma \ref{L:alpha}, $\alpha_{%
\mathrm{max}}>\alpha_{n}$. On the other hand, the variational
characterization of $-\alpha_{n}^{2}$ asserts that 
\begin{equation*}
\int_{0}^{h}(|\phi_{l}^{\prime2}-K(y)|\phi_{l}|^{2})dy\geq-\alpha_{n}^{2}%
\int_{0}^{h}|\phi_{l}|^{2}dy.
\end{equation*}
Accordingly, 
\begin{equation*}
\int_{0}^{h_{0}}(l^{2}\alpha^{2}-\alpha_{n}^{2})|\phi_{l}|^{2}dy\leq0
\end{equation*}
must hold. Since $\alpha>\alpha_{\mathrm{max}}/2$ and $l\geq2$, it follows
that $l^{2}\alpha^{2}-\alpha_{n}>2\alpha_{\mathrm{max}}^{2}-\alpha_{n}^{2}>0$%
. Consequently, $\phi_{l}\equiv0$. 
%A contradiction then  $\phi_k \equiv 0$ for $k\geq 2$.

Therefore, $\psi_{\infty}\equiv0$, which contradicts since $\|\phi_{\infty
}\|_{L^{2}}=1$. This completes the proof.
\end{proof}

\begin{lemma}
\label{L:solv} Under the assumption of \textrm{Lemma \ref{L:solv0}}, for any 
$b\in L^{2}(\mathcal{S}_{\epsilon })$, there exists an unique solution $\psi
_{b}$ to (\ref{E:psib}). Moreover, the estimate 
\begin{equation}
\Vert \psi _{b}\Vert _{H^{1}(\mathcal{D}_{\epsilon })}\leq C\Vert b\Vert
_{L^{2}(\mathcal{S}_{\epsilon })}  \label{E:elliptic}
\end{equation}%
holds, where $C>0$ is independent of $\epsilon ,b$.
\end{lemma}

\begin{proof}
The proof uses the theory of Fredholm alternative as adapted to usual
elliptic problems \cite[Section 6.2]{evans}. Let us introduce the Hilbert
space 
\begin{equation}
H\left( \mathcal{D}_{\epsilon }\right) =\{\psi \in H^{1}(\mathcal{D}%
_{\epsilon }):\psi (x,0)=0\}  \label{defn-spaceH}
\end{equation}%
and a bilinear form $B_{z}:H\times H\rightarrow \mathbb{R}$, defined as 
\begin{equation*}
B_{z}[\phi ,\psi ]=\iint_{\mathcal{D}_{\epsilon }}\left( \nabla \phi \cdot
\nabla \psi ^{\ast }+\phi (K^{\lambda }\psi )^{\ast }\right) dydx+z\left(
\phi ,\psi \right) .
\end{equation*}%
Here, 
\begin{equation*}
K^{\lambda }\psi =-\gamma (\psi _{\epsilon })\psi +\gamma ^{\prime }(\psi
_{\epsilon })\int_{-\infty }^{0}\lambda e^{\lambda s}\psi (X_{\epsilon
}(s),Y_{\epsilon }(s))ds,
\end{equation*}%
$z\in \mathbb{R}$ and $(\cdot ,\cdot )$ denotes the $L^{2}\left( \mathcal{D}%
_{\epsilon }\right) $ inner product. By the estimate (\ref{estimate-k-lb})
it follows that 
\begin{equation}
|B_{z}[\phi ,\psi ]|\leq (C(\gamma ,\delta )+z)\Vert \phi \Vert
_{H^{1}}\Vert \psi \Vert _{H^{1}}  \label{E:energy1}
\end{equation}%
and 
\begin{equation}
B_{z}[\phi ,\phi ]\geq c_{0}\Vert \phi \Vert _{H^{1}}^{2}  \label{E:energy2}
\end{equation}%
for $z>2C(\gamma ,\delta )$, where 
\begin{equation*}
C(\gamma ,\delta )=\Vert \gamma ^{\prime }(\psi _{\epsilon })\Vert
_{L^{\infty }}\left( 1+\frac{\left\vert \lambda _{0}\right\vert +\delta }{%
\delta }\right)
\end{equation*}%
and $c_{0}=\min (1,C(\gamma ,\delta ))>0$. Then, by the Lax-Milgram theorem
there exists a bounded operator $L_{z}:H^{\ast }\rightarrow H$ such that $%
B_{z}[L_{z}f,\phi ]=\left\langle f,\phi \right\rangle $ for any $f\in
H^{\ast }$ and $\phi \in H$, where $H^{\ast }$ denotes the dual space of $H$
and $\left\langle \cdot ,\cdot \right\rangle $ is the duality pairing. For $%
b\in L^{2}(\mathcal{S}_{\epsilon })$, the trace theorem \cite[Section 5.5]%
{evans} permits us to define $b^{\ast }\in H^{\ast }$ by 
\begin{equation*}
<b^{\ast },\psi >=\oint_{\mathcal{S}_{\epsilon }}b(x)\psi ^{\ast }(x,\eta
_{\epsilon }(x))dx\qquad \text{for}\quad \psi \in H.
\end{equation*}%
Note that $\psi _{b}\in H$ is a weak solution of (\ref{E:psib}) if and only
if 
\begin{equation*}
B_{z}(\psi _{b},\phi )=<z\psi _{b}+b^{\ast },\phi >\qquad \text{for all}%
\quad \phi \in H.
\end{equation*}%
That is to say, $\psi _{b}=L_{z}(z\psi _{b}+b^{\ast })$, or equivalently 
\begin{equation}
(id-zL_{z})\psi _{b}=L_{z}b^{\ast }.  \label{eqn-fredolm}
\end{equation}

The operator $L_{z}:L^{2}(\mathcal{D}_{\epsilon})\rightarrow L^{2}(\mathcal{D%
}_{\epsilon})$ is compact. Indeed, 
\begin{equation*}
\left\Vert L_{z}\phi\right\Vert _{H^{1}\left( \mathcal{D}_{\epsilon}\right)
}\leq\left\Vert L_{z}\right\Vert _{H^{\ast}\rightarrow H}\left\Vert
\phi\right\Vert _{H^{\ast}}\leq C\left\Vert \phi\right\Vert _{L^{2}(\mathcal{%
D}_{\epsilon})}
\end{equation*}
for any $\phi\in L^{2}(\mathcal{D}_{\epsilon})$. Moreover, the result of
Lemma \ref{L:solv0} states that $\ker( I-zL_{z}) =\{ 0\} $. Thus, by the
Fredholm alternative theory for compact operators, the equation (\ref%
{eqn-fredolm}) is uniquely solvable for any $b\in L^{2}(\mathcal{S}%
_{\epsilon})$ and 
\begin{equation}
\psi_{b}=\left( id-zL_{z}\right) ^{-1}L_{z}b^{\ast}.  \label{exp-phib}
\end{equation}

Next is the proof of (\ref{E:elliptic}). From (\ref{exp-phib}) it follows
that 
\begin{equation*}
\Vert \psi _{b}\Vert _{L^{2}(\mathcal{D}_{\epsilon })}\leq \left\Vert
(id-zL_{z})^{-1}\right\Vert _{L^{2}\rightarrow L^{2}}\left\Vert
L_{z}\right\Vert _{H^{\ast }\rightarrow H}\left\Vert b^{\ast }\right\Vert
_{H^{\ast }}\leq C\Vert b\Vert _{L^{2}(\mathcal{S}_{\epsilon })},
\end{equation*}%
where $C>0$ is independent of $b$ and $\epsilon $. Then, by (\ref%
{eqn-fredolm}) it follows that 
\begin{equation*}
\left\Vert \psi _{b}\right\Vert _{H^{1}(\mathcal{D}_{\epsilon })}=\left\Vert
L_{z}(z\psi _{b})+L_{z}b^{\ast }\right\Vert _{H}\leq C(\Vert \psi _{b}\Vert
_{L^{2}(\mathcal{D}_{\epsilon })}+\Vert b\Vert _{L^{2}(\mathcal{S}_{\epsilon
})})\leq C^{\prime }\Vert b\Vert _{L^{2}(\mathcal{S}_{\epsilon })},
\end{equation*}%
where $C,C^{\prime }>0$ are independent of $b$ and $\epsilon $. This
completes the proof.
\end{proof}

Similar considerations to the above proves the unique solvability of the
inhomogeneous problem.

\begin{corollary}
\label{lemma-solv-f}Under the assumption of \textrm{Lemma \ref{L:solv0}},
for any $f\in H^{\ast}(\mathcal{D}_{\epsilon})$ there exists an unique weak
solution $\psi_{f}\in H^{1}(\mathcal{D}_{\epsilon})$ to the elliptic problem 
\begin{gather*}
\Delta\psi+\gamma^{\prime}(\psi_{\epsilon})\psi-\gamma^{\prime}(\psi
_{\epsilon})\int_{-\infty}^{0}\lambda e^{\lambda s}\psi(X_{\epsilon
}(s),Y_{\epsilon}(s))ds=f\quad\text{in }\mathcal{D}_{\epsilon}, \\
\psi_{n}=0\qquad\text{on}\quad\mathcal{S}_{\epsilon}, \\
\psi(x,0)=0.
\end{gather*}
The estimate 
\begin{equation}
\Vert\psi_{f}\Vert_{H^{1}(\mathcal{D}_{\epsilon})}\leq C\Vert
f\Vert_{H^{\ast }(\mathcal{D}_{\epsilon})}  \label{estimate-H1-f}
\end{equation}
holds, where $C>0$ is independent of $f$ and $\epsilon$. Moreover, if $f\in
L^{2}(\mathcal{D}_{\epsilon})$ then an improved estimate 
\begin{equation}
\Vert\psi_{f}\Vert_{H^{2}(\mathcal{D}_{\epsilon})}\leq C\Vert f\Vert _{L^{2}(%
\mathcal{D}_{\epsilon})}  \label{estimate-H2-f}
\end{equation}
holds, where $C>0$ is independent of $f$ and $\epsilon$.
\end{corollary}

\begin{proof}
As in the proof of Lemma \ref{L:solv}, the function $\psi_{f}$ is a solution
to the above boundary value problem if and only if 
\begin{equation*}
(id-zL_{z}) \psi_{f}=L_{z}f.
\end{equation*}
The unique solvability and the $H^{1}$-estimate (\ref{estimate-H1-f}) are
identically the same as those in Lemma \ref{L:solv}. When $f\in L^{2}(%
\mathcal{D}_{\epsilon})$ the elliptic regularity theory \cite{ag-book} for
Neumann boundary condition implies (\ref{estimate-H2-f}).
\end{proof}

For any $b\in L^{2}_{\text{per}}(\mathcal{S}_{\epsilon})$ let us define the
operator 
\begin{equation}
\mathcal{T}_{\epsilon}b=\psi_{b}|_{\mathcal{S}_{\epsilon}}=\psi_{b}\left(
x,\eta_{\epsilon}(x)\right) ,  \label{defn-T-ebs}
\end{equation}
where $\psi_{b}$ is the unique solution of (\ref{E:psib}) in Lemma \ref%
{L:solv} with periodicity. Then, the elliptic estimate (\ref{E:elliptic})
and the trace theorem \cite[Section 5.5]{evans} implies that 
\begin{equation}
\Vert\mathcal{T}_{\epsilon}b\Vert_{H^{1/2}\left( \mathcal{S}_{\epsilon
}\right) }\leq C\Vert\psi_{b}\Vert_{H^{1}\left( \mathcal{D}_{\epsilon
}\right) }\leq C^{\prime}\Vert b\Vert_{L^{2}};  \label{E:t-cpt}
\end{equation}
Therefore, $\mathcal{T}_{\epsilon}:L^{2}_{\text{per}}(\mathcal{S}%
_{\epsilon}) \to L^{2}_{\text{per}}(\mathcal{S}_{\epsilon})$ is a compact
operator.

\subsection{Proof of Theorem \protect\ref{T:unstable}}

This subsection is devoted to the proof of Theorem \ref{T:unstable}
pertaining to the linear instability of small-amplitude periodic traveling
waves over an unstable shear flow.

For $\epsilon \geq 0$ sufficiently small and $|\lambda - \lambda_{0}| \leq (%
{\rm Re}\lambda_{0})/2$, let us denote 
\begin{equation}
\mathcal{F}(\lambda,\epsilon)=\mathcal{T}\left( \lambda,\epsilon\right) 
\mathcal{C}\left( \lambda,\epsilon\right) (P_{\epsilon y}\mathcal{C}\left(
\lambda,\epsilon\right) +\Omega \, id):L^{2}_{\text{per}}(\mathcal{S}%
_{\epsilon})\rightarrow L^{2}_{\text{per}}(\mathcal{S}_{\epsilon}),
\label{E:operator-Fc}
\end{equation}
where $\mathcal{C}(\lambda,\epsilon)=\mathcal{C}^{\lambda}$ and $\mathcal{T}%
(\lambda,\epsilon)=\mathcal{T}_{\epsilon}$ are defined in (\ref%
{defn-c-lambda}) and (\ref{defn-T-ebs}), respectively. In the light of the
discussion in the previous subsection, it suffices to show that for each
small parameter $\epsilon \geq 0$ there exists $\lambda(\epsilon)$ with $%
|\lambda(\epsilon) - \lambda_{0}| \leq ({\rm Re}\lambda_{0})/2$ such that
the operator $id +\mathcal{F}(\lambda (\epsilon),\epsilon)$ has a nontrivial
null space. The result of Theorem \ref{class-k} states that there exists $%
\lambda_{0}=-i\alpha c_{\alpha}$ with $\mathrm{Im}\,c_{\alpha}>0$ such that $%
id+\mathcal{F}(\lambda_{0},0)$ has a nontrivial null space. The proof for $%
\epsilon>0$ uses a perturbation argument, based on the following lemma due
to Steinberg \cite{steinberg}.

\begin{lemma}
\label{L:Steinberg} Let $F(\lambda,\epsilon)$ be a family of compact
operators on a Banach space, analytic in $\lambda$ in a region $\Lambda$ in
the complex plane and jointly continuous in $(\lambda,\epsilon)$ for each $%
(\lambda, \epsilon) \in \Lambda \times\mathbb{R}$. Suppose that $%
id-F(\lambda_{0},\epsilon)$ is invertible for some $\lambda_{0}\in \Lambda$
and all $\epsilon \in\mathbb{R}$. Then, $R(\lambda,
\epsilon)=(id-F(\lambda,\epsilon))^{-1}$ is meromorphic in $\Lambda$ for
each $\epsilon \in\mathbb{R}$ and jointly continuous at $(z_{0},
\epsilon_{0})$ if $\lambda_{0}$ is not a pole of $R(\lambda ,\epsilon)$; its
poles depend continuously on $\epsilon$ and can appear or disappear only at
the boundary of $\Lambda$.
\end{lemma}

In order to apply Lemma \ref{L:Steinberg} to our situation, we need to
transform the operator (\ref{E:operator-Fc}) to one on a function space
independent of the parameter $\epsilon$. This calls for the employment of
the action-angle mapping $\mathcal{A}_{\epsilon}$, as is done in the proof
of Lemma \ref{L:solv0}: 
\begin{equation*}
\mathcal{A}_{\epsilon}:\mathcal{D}_{\epsilon}\rightarrow D \quad \text{and}%
\quad \mathcal{A}_{\epsilon}(x,y)=(\theta,I),
\end{equation*}
where the action-angle variables $(\theta,I)$ are defined in (\ref{E:aa})
and 
\begin{equation*}
D=\{(\theta,I):0<\theta<2\pi/\alpha\,,\,0<I<h_{0}\}.
\end{equation*}
Note that $\mathcal{A}_{\epsilon}$ maps $\mathcal{S}_{\epsilon}$ bijectively
to $\{(\theta,h_{0}):0<\theta<2\pi/\alpha\}$. The latter may be identified
with $(2\pi/\alpha,h_{0})$. This naturally induces an homeomorphism%
\begin{equation*}
\mathcal{B}_{\epsilon}:L^{2}_{\text{per}}(\mathcal{S}_{\epsilon})\rightarrow
L_{\text{per}}^{2}([0,2\pi/\alpha])
\end{equation*}
by 
\begin{equation*}
(\mathcal{B}_{\epsilon}f)(\theta)=f(\mathcal{A}_{\epsilon}^{-1}\left(
\theta,h_{0}\right) ).
\end{equation*}
Let us denote the following operators from $L_{\text{per}}^{2}\left(
[0,2\pi/\alpha]\right)$ to itself: 
\begin{equation*}
\mathcal{\tilde{T}}\left( \lambda,\epsilon\right) =\mathcal{B}_{\epsilon }%
\mathcal{T}\left( \left( \lambda,\epsilon\right) \right) (\mathcal{B}%
_{\epsilon})^{-1},\newline
\end{equation*}%
\begin{equation*}
\mathcal{\tilde{C}}\left( \lambda,\epsilon\right) =\mathcal{B}_{\epsilon }%
\mathcal{C}\left( \lambda,\epsilon\right) (\mathcal{B}_{\epsilon})^{-1},
\end{equation*}%
\begin{equation*}
\mathcal{\tilde{P}}_{\epsilon}=\mathcal{B}_{\epsilon}P_{\epsilon y}(\mathcal{%
B}_{\epsilon})^{-1},
\end{equation*}
and 
\begin{equation}
\mathcal{\tilde{F}}(\lambda,\epsilon)=\mathcal{B}_{\epsilon}\mathcal{F}%
(\lambda,\epsilon)(\mathcal{B}_{\epsilon})^{-1}=\mathcal{\tilde{T}}\left(
\lambda,\epsilon\right) \mathcal{\tilde{C}}( \lambda,\epsilon) (\mathcal{%
\tilde{P}}_{\epsilon}\mathcal{\tilde{C}}\left( \lambda ,\epsilon\right)
+\Omega id).  \label{defn-cal-F}
\end{equation}

Since $\mathcal{T}(\lambda ,\epsilon )$ is compact and $\mathcal{T}(\lambda
,\epsilon )$ and $\mathcal{C}(\lambda ,\epsilon )$ are analytic in $\lambda $
with $|\lambda -\lambda _{0}|\leq ({\rm Re}\lambda _{0})/2$, the operator $%
\mathcal{F}(\lambda ,\epsilon )$ is compact and analytic in $\lambda $.
Subsequently, $\mathcal{\tilde{F}}(\lambda ,\epsilon )$ is compact and
analytic in $\lambda $. Clearly, $P_{\epsilon y}(x)$, $\mathcal{C}(\lambda
,\epsilon )$ and $\mathcal{B}_{\epsilon }$ are continuous in $\epsilon $,
and in turn, $\mathcal{\tilde{C}}(\lambda ,\epsilon )$ and $\mathcal{\tilde{P%
}}_{\epsilon }$ are continuous in $\epsilon $. The key technical lemma is to
show the continuity of $\mathcal{\tilde{T}}(\lambda ,\epsilon )$ in $%
\epsilon $ and thus obtain the the continuity of $\mathcal{\tilde{F}}%
(\lambda ,\epsilon )$ in $\epsilon $.

\begin{lemma}
\label{L:cont} For $\epsilon \geq 0$ sufficiently small and $|\lambda -
\lambda_{0}| \leq ({\rm Re}\lambda_{0})/2$, the operator $\mathcal{\tilde{T}%
}( \lambda,\epsilon) $ satisfies the estimate 
\begin{equation}
\Vert\mathcal{\tilde{T}}( \lambda,\epsilon_{1}) -\mathcal{\tilde {T}}(
\lambda,\epsilon_{2}) \Vert_{L_{\text{per}}^{2}([0,2\pi/\alpha])\rightarrow
L_{\text{per}}^{2}([0,2\pi/\alpha])}\leq C|\epsilon_{1}-\epsilon_{2}|,
\label{E:cont}
\end{equation}
where $C>0$ is independent of $\lambda$ and $\epsilon$.
\end{lemma}

\begin{proof}
For a given $b\in L_{\text{per}}^{2}([0,2\pi/\alpha])$, let us denote $%
b_{\epsilon }=\mathcal{B}_{\epsilon}^{-1}b\in L^{2}(\mathcal{S}_{\epsilon}) $%
. By definition 
\begin{equation*}
\mathcal{\tilde{T}}\left( \lambda,\epsilon\right) b=\mathcal{B}_{\epsilon }%
\mathcal{T}_{\epsilon}b_{\epsilon}=\mathcal{B}_{\epsilon}(\psi_{b,\epsilon
}|_{\mathcal{S}_{\epsilon}}),
\end{equation*}
where $\psi_{b,\epsilon}\in H^{1}\left( \mathcal{D}_{\epsilon}\right) $ is
the unique weak solution of 
\begin{subequations}
\label{eqn-phi-b-ebs}
\begin{gather}
\Delta\psi_{b,\epsilon}+\gamma^{\prime}(\psi_{\epsilon})\psi_{b,\epsilon
}-\gamma^{\prime}(\psi_{\epsilon})\int_{-\infty}^{0}\lambda e^{\lambda
s}\psi_{b,\epsilon}(X_{\epsilon}(s),Y_{\epsilon}(s))ds=0\quad\text{in }%
\mathcal{D}_{\epsilon}; \\
(\psi_{b,\epsilon})_{n}=b_{\epsilon}\quad\text{on }\,\mathcal{S}_{\epsilon};
\\
\psi_{b,\epsilon}(x,0)=0.
\end{gather}

Our goal is to estimate the $L^{2}$-operator norm of $\mathcal{\tilde{T}}%
\left( \lambda ,\epsilon _{1}\right) -\mathcal{\tilde{T}}\left( \lambda
,\epsilon _{2}\right) $ in terms of $|\epsilon _{1}-\epsilon _{2}|$. Since
the domain of the boundary value problem (\ref{eqn-phi-b-ebs}) depends on $%
\epsilon $, we use the action-angle mapping $\mathcal{A}_{\epsilon _{j}}$ ($%
j=1,2$) to transform functions and the Laplacian operator in $\mathcal{D}%
_{\epsilon _{j}}$ ($j=1,2$) to those in the fixed domain $D$. To simplify
notations, we use $\mathcal{A}_{\epsilon _{j}}$ ($j=1,2$) to denote the
induced transformations for functions and operators. Let $\psi _{j}=\mathcal{%
A}_{\epsilon _{j}}(\psi _{b,\epsilon _{j}})$, which are $H^{1}(D)$%
-functions, and let 
\end{subequations}
\begin{equation*}
\Delta _{j}=\mathcal{A}_{\epsilon _{j}}(\Delta ),\qquad \gamma _{j}(I)=%
\mathcal{A}_{\epsilon _{j}}(\gamma ^{\prime }(\psi _{\epsilon _{j}})),
\end{equation*}%
where $j=1,2$. By definition, $\mathcal{A}_{\epsilon _{j}}(b_{\epsilon
_{j}})=\mathcal{B}_{\epsilon _{j}}(b_{\epsilon _{j}})=b$. Note that the
characteristic equation (\ref{E:char}) in the action-angle variables $%
(\theta ,I)$ becomes 
\begin{equation*}
\begin{cases}
\dot{\theta}=\upsilon _{j}(I) \\ 
\dot{I}=0%
\end{cases}%
\end{equation*}%
for $j=1,2$, where $\upsilon _{j}(I)=\upsilon _{\epsilon _{j}}(I)$. That
means, the trajectory $(X_{\epsilon _{j}}(s),Y_{\epsilon _{j}}(s))$ in the
phase space transforms under the mapping $\mathcal{A}_{\epsilon _{j}}$ into $%
(\theta +\upsilon _{j}(I)s,I)$. Since $\psi _{j}$ $(j=1,2)$ are $2\pi
/\alpha $-periodic in $\theta $, we have the Fourier expansions 
\begin{equation*}
\psi _{j}(\theta ,I)=\sum_{l}e^{il\theta }\psi _{j,l}(I),\ j=1,2.
\end{equation*}%
Under the action-angle mapping $\mathcal{A}_{\epsilon _{j}}\left(
j=1,2\right) $ the left side of (\ref{eqn-phi-b-ebs}) becomes 
\begin{align}
& \mathcal{A}_{\epsilon _{j}}\left( \gamma ^{\prime }(\psi _{\epsilon })\psi
_{b,\epsilon }-\gamma ^{\prime }(\psi _{\epsilon })\int_{-\infty
}^{0}\lambda e^{\lambda s}\psi _{b,\epsilon }(X_{\epsilon }(s),Y_{\epsilon
}(s))ds\right)  \notag \\
& =\gamma _{j}(I)\left( \sum_{l}e^{il\theta }\psi _{j,l}(I)-\int_{-\infty
}^{0}\lambda e^{\lambda s}\sum_{l}e^{il(\theta +\upsilon _{j}(I)s)}\psi
_{j,l}(I)ds\right)  \label{eqn-Klb-A} \\
& =\gamma _{j}(I)\sum_{l}\frac{il\upsilon _{j}(I)}{\lambda +il\upsilon
_{j}(I)}\psi _{j,l}(I)e^{il\theta },  \notag
\end{align}%
and thus the system (\ref{eqn-phi-b-ebs}) becomes 
\begin{gather*}
\Delta _{j}\psi _{j}+\gamma _{j}(I)\sum_{l}\frac{il\upsilon _{j}(I)}{\lambda
+il\upsilon _{j}(I)}\psi _{j,l}(I)e^{il\theta }=0\quad \text{in }D, \\
\partial _{I}\psi _{j}(\theta ,h_{0})=b(\theta ), \\
\psi _{j}(\theta ,0)=0.
\end{gather*}%
Accordingly, the difference $\psi _{1}-\psi _{2}$ is a weak solution of the
partial differential equation 
\begin{equation}
\begin{split}
\Delta _{1}(\psi _{1}& -\psi _{2})+(\Delta _{1}-\Delta _{2})\psi _{2}+\gamma
_{1}\sum_{l}\frac{il\upsilon _{1}}{\lambda +il\upsilon _{1}}(\psi
_{1,l}(I)-\psi _{2,l}(I))e^{il\theta } \\
& +\gamma _{1}\sum_{l}\left( \frac{il\upsilon _{1}}{\lambda +il\upsilon _{1}}%
-\frac{il\upsilon _{2}}{\lambda +il\upsilon _{2}}\right) \psi
_{2,l}(I)e^{il\theta }+(\gamma _{1}-\gamma _{2})\sum_{l}\frac{il\upsilon _{2}%
}{\lambda +il\upsilon _{2}}\psi _{2,l}(I)e^{il\theta }=0
\end{split}
\label{E:psi-difference}
\end{equation}%
with the boundary conditions 
\begin{subequations}
\begin{gather}
\partial _{I}(\psi _{1}-\psi _{2})(\theta ,h_{0})=0,  \label{b1-phi-12} \\
(\psi _{1}-\psi _{2})(\theta ,0)=0.  \label{b2-phi12}
\end{gather}%
Let us write (\ref{E:psi-difference}) as 
\end{subequations}
\begin{equation}
\Delta _{1}(\psi _{1}-\psi _{2})+\gamma _{1}\sum_{l}\frac{il\upsilon _{1}(I)%
}{\lambda -ik\upsilon _{1}(I)}(\psi _{1,l}(I)-\psi _{2,l}(I))e^{il\theta }=f,
\label{eqn-phi-1-2}
\end{equation}%
where $f=f_{1}+f_{2}+f_{3}$ with 
\begin{align*}
f_{1}& =-(\Delta _{1}-\Delta _{2})\psi _{2}, \\
f_{2}& =-\gamma _{1}\sum_{l}\frac{i\lambda l(\upsilon _{1}-\upsilon _{2})}{%
(\lambda +il\upsilon _{1})(\lambda +il\upsilon _{2})}\psi
_{2,l}(I)e^{il\theta }, \\
f_{3}& =-(\gamma _{1}-\gamma _{2})\sum_{l}\frac{il\upsilon _{2}}{\lambda
+il\upsilon _{2}}\psi _{2,l}(I)e^{il\theta }.
\end{align*}%
We estimate $f_{1},f_{2},f_{3}$ separately. To simplify notations, $C>0$ in
the estimates below denotes a generic constant independent of $\epsilon $
and $\lambda $.

First, we claim that $f_{1}\in H^{\ast }(D)$ with the estimates 
\begin{equation}
\left\Vert f_{1}\right\Vert _{H^{\ast }\left( D_{0}\right) }\leq C|\epsilon
_{1}-\epsilon _{2}|\Vert b\Vert _{L^{2}([0,2\pi /\alpha ])},
\label{estimate-f1}
\end{equation}%
where $H^{\ast }(D)$ is the dual space of 
\begin{equation*}
H(D)=\{\psi \in H^{1}(D):\psi (\theta ,I)=\psi (\theta +2\pi /\alpha ,I),\
\psi (\theta ,0)=0\}.
\end{equation*}%
Let us write 
\begin{equation*}
\Delta _{j}=a_{II}^{j}\partial _{II}+a_{I\theta }^{j}\partial _{I\theta
}+a_{\theta \theta }^{j}\partial _{\theta \theta }+b_{I}^{j}\partial
_{I}+b_{\theta }^{j}\partial _{\theta }\qquad \text{for}\quad j=1,2,
\end{equation*}%
and the difference of the coefficients as 
\begin{gather*}
\bar{a}_{II}=a_{II}^{1}-a_{II}^{2},\ \ \bar{a}_{I\theta }=a_{I\theta
}^{1}-a_{I\theta }^{2},\ \ \ \bar{a}_{\theta \theta }=a_{\theta \theta
}^{1}-a_{\theta \theta }^{2}, \\
\bar{b}_{I}=b_{I}^{1}-b_{I}^{2},\quad \bar{b}_{\theta }=b_{\theta
}^{1}-b_{\theta }^{2}.
\end{gather*}%
Then formally, for any $\phi \in H(D)\cap C^{2}(\bar{D})$ it follows that%
\begin{eqnarray}
&&\int_{D_{0}}f_{1}\phi \ dId\theta  \label{comp-forma} \\
&=&\int_{D_{0}}-\phi \left( \bar{a}_{II}\partial _{II}+\bar{a}_{I\theta
}\partial _{I\theta }+\bar{a}_{\theta \theta }\partial _{\theta \theta }+%
\bar{b}_{\theta }\partial _{\theta }+\bar{b}_{I}\partial _{I}\right) \psi
_{2}dId\theta  \notag \\
&=&\int_{D_{0}}\left[ \partial _{I}\psi _{2}\partial _{I}\left( \bar{a}%
_{II}\phi \right) +\partial _{I}\psi _{2}\partial _{\theta }\left( \bar{a}%
_{I\theta }\phi \right) +\partial _{\theta }\psi _{2}\partial _{\theta
}\left( \bar{a}_{\theta \theta }\phi \right) \right] dId\theta  \notag \\
&&-\int_{D_{0}}\phi \left( \bar{b}_{I}\partial _{I}\psi _{2}+\bar{b}_{\theta
}\partial _{\theta }\psi _{2}\right) dId\theta -\int_{\left\{
I=h_{0}\right\} }\bar{a}_{II}\phi b\left( \theta \right) d\theta ,  \notag
\end{eqnarray}%
This uses that $\psi _{2}$ and $\phi $ are periodic in the $\theta $%
-variable and that 
\begin{equation*}
\partial _{I}\psi _{2}(\theta ,h_{0})=b(\theta ),\qquad \phi (\theta ,0)=0.
\end{equation*}%
Note that the elliptic estimate (\ref{E:elliptic}) and the equivalence of
norms under the transformation $\mathcal{A}_{\epsilon _{2}}^{-1}$ assert
that 
\begin{align*}
\left\Vert \psi _{2}\right\Vert _{H^{1}(D)}& \leq C\left\Vert \mathcal{A}%
_{\epsilon _{2}}^{-1}\psi _{2}\right\Vert _{H^{1}(\mathcal{D}_{\epsilon
_{2}})}=C\left\Vert \psi _{b,\epsilon _{2}}\right\Vert _{H^{1}(\mathcal{D}%
_{\epsilon _{2}})} \\
& \leq C\left\Vert b_{\epsilon _{2}}\right\Vert _{L^{2}(\mathcal{S}%
_{\epsilon _{2}})}\leq C\left\Vert b\right\Vert _{L^{2}([0,2\pi /\alpha ])}.
\end{align*}%
Since 
\begin{equation*}
\left\vert \bar{a}_{II}\right\vert _{C^{1}}+\left\vert \bar{a}_{I\theta
}\right\vert _{C^{1}}+\ \left\vert \bar{a}_{\theta \theta }\right\vert
_{C^{1}}+\left\vert \bar{b}_{\theta }\right\vert _{C^{1}}+\left\vert \bar{b}%
_{I}\right\vert _{C^{1}}=O(|\epsilon _{1}-\epsilon _{2}|),
\end{equation*}%
by using the trace theorem it follow from (\ref{comp-forma}) the estimate 
\begin{align*}
\left\vert \int_{D}f_{1}\phi \ dId\theta \right\vert & \leq C\left\vert
\epsilon _{1}-\epsilon _{2}\right\vert \left( \left\Vert \psi
_{2}\right\Vert _{H^{1}(D)}+\left\Vert b\right\Vert _{L^{2}([0,2\pi /\alpha
])}\right) \left\Vert \phi \right\Vert _{H^{1}(D)} \\
& \leq C\left\vert \epsilon _{1}-\epsilon _{2}\right\vert \left\Vert
b\right\Vert _{L^{2}([0,2\pi /\alpha ])}\left\Vert \phi \right\Vert
_{H^{1}(D)}.
\end{align*}%
This proves the estimate (\ref{estimate-f1}).

We claim that if $b$ is smooth then the formal manipulations in (\ref%
{comp-forma}) are valid and $\psi _{2}\in C^{2}(\bar{D})$. Note that Theorem %
\ref{T:smallE} ensures that the steady state $(\eta _{\epsilon _{2}}(x),\psi
_{\epsilon _{2}}(x,y))$ is in $C^{3+\beta }$ class, where $\beta \in (0,1)$.
Since $b$ is smooth it follows that $b_{\epsilon _{2}}=\mathcal{B}_{\epsilon
_{2}}^{-1}b$ is at least in $H^{2}\left( \mathcal{S}_{\epsilon _{2}}\right) $%
. Then, the similar argument as in the regularity proof of Theorem \ref%
{T:unstable} below asserts that $\psi _{b,\epsilon _{2}}\in H^{7/2}(\mathcal{%
D}_{\epsilon _{2}})\subset C^{2}(\mathcal{\bar{D}}_{\epsilon _{2}})$. Since
the definition of the action-angle variables guarantees that the mapping $%
\mathcal{A}_{\epsilon _{2}}$ is at least of $C^{2}$, subsequently, $\psi
_{2}=\mathcal{A}_{\epsilon _{2}}\psi _{b,\epsilon _{2}}\in C^{2}(\bar{D})$.
This proves the claim. If $b\in L^{2}$, an approximation of $b$ by smooth
functions establishes (\ref{estimate-f1}).

Next, since 
\begin{equation*}
\left\vert \frac{1}{\lambda+il\upsilon_{j}}\right\vert =\frac{1}{(|{\rm Re}%
\,\lambda|^{2}+|\mathrm{Im}\,\lambda-l\upsilon_{j}|^{2})^{1/2}}\leq\frac{1}{|%
{\rm Re}\,\lambda|} \leq\frac{2}{{\rm Re}\,\lambda_{0}},
\end{equation*}
by the estimate (\ref{E:elliptic}) it follows that 
\begin{equation}
\Vert f_{2}\Vert_{L^{2}(D)}\leq C|\epsilon_{1}-\epsilon_{2}|\Vert\psi
_{2}\Vert_{L^{2}(D)}\leq C|\epsilon_{1}-\epsilon_{2}|\Vert b\Vert
_{L^{2}([0,2\pi/\alpha])}.  \label{E:II}
\end{equation}
Similarly,%
\begin{equation}
\left\Vert f_{3}\right\Vert _{L^{2}\left( D_{0}\right) }\leq C|\epsilon
_{1}-\epsilon_{2}|\Vert b\Vert_{L^{2}([0,2\pi/\alpha])}.  \label{E:III}
\end{equation}
Combining the estimates (\ref{estimate-f1}), (\ref{E:II}) and (\ref{E:III})
asserts that $f\in H^{\ast}(D)$ and 
\begin{equation*}
\left\Vert f\right\Vert _{H^{\ast}(D)}\leq C|\epsilon
_{1}-\epsilon_{2}|\Vert b\Vert_{L^{2}([0,2\pi/\alpha])}.
\end{equation*}

Let $\psi=\psi_{1}-\psi_{2}\in H^{1}(D)$ and $\phi=\mathcal{A}%
_{\epsilon_{1}}^{-1}\psi\in H^{1}\left( \mathcal{D}_{\epsilon_{1}}\right) $.
It remains to transform back to the physical space of the boundary value
problem for $\psi$ and to compute the operator norm of $\mathcal{\tilde{T}}%
\left( \lambda,\epsilon_{1}\right) -\mathcal{\tilde{T}}\left(
\lambda,\epsilon_{2}\right)$. Under the transformation $\mathcal{A}%
_{\epsilon_{1}}^{-1}$, the equations (\ref{eqn-phi-1-2}), (\ref{b1-phi-12}),
(\ref{b2-phi12}) become 
\begin{gather*}
\Delta\phi+\gamma^{\prime}(\psi_{\epsilon_{1}})\phi
-\gamma^{\prime}(\psi_{\epsilon_{1}})\int_{-\infty}^{0}\lambda e^{\lambda
s}\phi(X_{\epsilon_{1}}(s),Y_{\epsilon_{1}}(s))ds=\mathcal{A}%
_{\epsilon_{1}}^{-1}f\quad\text{in }\mathcal{D}_{\epsilon_{1}}; \\
(\phi^{\epsilon_{1}})_{n}=0\quad\text{on }\,\mathcal{S}_{\epsilon_{1}}; \\
\phi^{\epsilon_{1}}(x,0)=0,
\end{gather*}
Then, $\mathcal{A}_{\epsilon_{1}}^{-1}f\in H^{\ast}(\mathcal{D}%
_{\epsilon_{1}})$ and 
\begin{equation*}
\left\Vert \mathcal{A}_{\epsilon_{1}}^{-1}f\right\Vert _{H^{\ast}(\mathcal{D}%
_{\epsilon_{1}})}\leq C\left\Vert f\right\Vert _{H^{\ast}(D) }\leq
C|\epsilon_{1}-\epsilon_{2}|\Vert b\Vert_{L^{2}([0,2\pi/\alpha])}.
\end{equation*}
Corollary \ref{lemma-solv-f} thus applies to assert that 
\begin{align*}
\left\Vert \psi\right\Vert _{H^{1}(D)} & \leq\left\Vert \phi\right\Vert
_{H^{1}( \mathcal{D}_{\epsilon_{1}}) }\leq C\left\Vert \mathcal{A}%
_{\epsilon_{1}}^{-1}f\right\Vert _{H^{\ast}(\mathcal{D}_{\epsilon_{1}})} \\
& \leq C|\epsilon_{1}-\epsilon_{2}|\Vert b\Vert_{L^{2}([0,2\pi/\alpha])}.
\end{align*}
Finally, by the trace theorem it follows 
\begin{align*}
\Vert\mathcal{\tilde{T}}\left( \lambda,\epsilon_{1}\right) b-\mathcal{\tilde{%
T}}\left( \lambda,\epsilon_{2}\right) b\Vert _{L_{\text{per}%
}^{2}([0,2\pi/\alpha])}&=\left\Vert \left( \psi_{1}-\psi _{2}\right) \left(
\theta,h_{0}\right) \right\Vert _{L_{\text{per}}^{2}([0,2\pi/\alpha])} \\
& \leq C\left\Vert \psi\right\Vert _{H^{1}(D) }\leq
C|\epsilon_{1}-\epsilon_{2}|\Vert b\Vert_{L^{2}_{\text{per}%
}([0,2\pi/\alpha])}.
\end{align*}
This completes the proof.
\end{proof}

We are now in a position to prove our main theorem.

\begin{proof}[Proof of Theorem \protect\ref{T:unstable}]
For $|\lambda -\lambda _{0}|\leq ({\rm Re}\lambda _{0})/2$, where $\lambda
_{0}=-i\alpha c_{\alpha }$, and $\epsilon \geq 0$ small, consider the family
of operators $\mathcal{\tilde{F}}(\lambda ,\epsilon )$ on $L_{\text{per}%
}^{2}([0,2\pi /\alpha ])$, defined by (\ref{defn-cal-F}). The discussions
following (\ref{defn-cal-F}) and Lemma \ref{L:cont} assert that $\mathcal{%
\tilde{F}}(\lambda ,\epsilon )$ is compact, analytic in $\lambda $ and
continuous in $\epsilon $.

By our assumption, ${\rm Im}c_{\alpha }>0$ and $c_{\alpha }$ is an unstable
eigenvalue of the Rayleigh system (\ref{rayleigh})-(\ref{bc-rayleigh}) which
corresponds to $\epsilon =0$. In other words, $\lambda _{0}$ is a pole of $%
(id+\mathcal{F}(\lambda ,0))^{-1}$. Subsequently, it is a pole of $(id+%
\mathcal{\tilde{F}}(\lambda ,0))^{-1}$. Since $\lambda _{0}$ is an isolated
pole, we may choose $\delta >0$ small enough so that the operator $id+%
\mathcal{\tilde{F}}(\lambda ,0)$ is invertible on $|\lambda -\lambda
_{0}|=\delta $. By the continuity of $\mathcal{\tilde{F}}(\lambda ,\epsilon
) $ in $\epsilon $, the following estimate 
\begin{equation*}
\Vert \mathcal{\tilde{F}}(\lambda ,\epsilon )-\mathcal{\tilde{F}}(\lambda
,0)\Vert _{L_{\text{per}}^{2}([0,2\pi /\alpha ])\rightarrow L_{\text{per}%
}^{2}([0,2\pi /\alpha ])}\leq C\epsilon
\end{equation*}%
holds. Hence, $id+\mathcal{\tilde{F}}(\lambda ,\epsilon )$ is invertible on $%
|\lambda -\lambda _{0}|=\delta $ for $\epsilon \geq 0$ sufficiently small.
Then by Lemma \ref{L:Steinberg}, the poles of $(id+\mathcal{\tilde{F}}%
(\lambda ,\epsilon ))^{-1}$ are continuous in $\epsilon $ and can only
appear or disappear in the boundary of $\{\lambda :|\lambda -\lambda
_{0}<\delta \}$. Therefore, for each $\epsilon \geq 0$, there exists a pole $%
\lambda (\epsilon )$ of $\left( id+\mathcal{\tilde{F}}(\lambda ,\epsilon
)\right) ^{-1}$ in $|\lambda (\epsilon )-\lambda _{0}|<\delta $. Thus, $%
{\rm Re}\,\lambda (\epsilon )>0$ and there exists a nonzero function $%
\tilde{f}\in L_{\text{per}}^{2}([0,2\pi /\alpha ])$ such that 
\begin{equation}
(id+\mathcal{\tilde{F}}(\lambda ,\epsilon ))\tilde{f}=0.  \label{eqn-F^t-d}
\end{equation}

Below we construct an exponentially growing solution to the linearized
system (\ref{eqn-L-vor})-(\ref{eqn-L-bottom}). Define 
\begin{equation*}
f=(\mathcal{B}_{\epsilon })^{-1}\tilde{f}\in L^{2}(\mathcal{S}_{\epsilon }),
\end{equation*}%
then 
\begin{equation}
\left( I+\mathcal{F}(\lambda ,\epsilon )\right) f=0  \label{eqn-F-f}
\end{equation}%
by (\ref{eqn-F^t-d}) and the definition of $\mathcal{\tilde{F}}$. Let $\psi
\left( x,y\right) \in $ $H^{1}\left( \mathcal{D}_{\epsilon }\right) $ to be
the unique solution of (\ref{E:psib}) with $\lambda =\lambda (\epsilon )$
and 
\begin{equation}
\psi _{n}(x)=b=-\mathcal{C}^{\lambda }(P_{\epsilon y}(x)\mathcal{C}^{\lambda
}+\Omega I)f(x).  \label{eqn-b-f}
\end{equation}%
By (\ref{eqn-F-f}), (\ref{eqn-b-f}) and the definition of $\mathcal{F}$, we
have 
\begin{equation*}
f=\mathcal{T}_{\epsilon }b=\psi (x,\eta _{\epsilon }(x))
\end{equation*}%
and thus 
\begin{equation*}
\psi _{n}(x)=-\mathcal{C}^{\lambda }(P_{\epsilon y}(x)\mathcal{C}^{\lambda
}+\Omega I)\psi (x,\eta _{\epsilon }(x)).
\end{equation*}%
Define 
\begin{equation}
\eta (x)=\mathcal{C}^{\lambda }\left[ \psi (x,\eta _{\epsilon }(x))\right]
\in L^{2}(\mathcal{S}_{\epsilon }),  \label{eqn-W-eta}
\end{equation}%
and 
\begin{equation*}
P(x,\eta _{\epsilon }(x))=-P_{\epsilon y}(x)\eta (x),
\end{equation*}%
then (\ref{eqn-b-f}) becomes 
\begin{equation}
\psi _{n}(x)=-\mathcal{C}^{\lambda }\left( P(x,\eta _{\epsilon }(x))+\Omega
\psi (x,\eta _{\epsilon }(x))\right) .  \label{eqn-W-phi-n}
\end{equation}

Now we show that $\left[ e^{\lambda (\epsilon )t}\psi \left( x,y\right)
,e^{\lambda (\epsilon )t}\eta (x)\right] \ $satisfies the linearized system (%
\ref{eqn-L-vor})-(\ref{eqn-L-bottom}). The bottom boundary condition (\ref%
{eqn-L-bottom}) is satisfied since $\psi \left( x,0\right) =0.$ The equation
(\ref{eqn-L-P-eta}) is automatic. By Lemma \ref{lemma-c-lb} and equations (%
\ref{eqn-W-eta}), (\ref{eqn-W-phi-n}), $\eta \left( x\right) $ and $\psi
_{n}(x)$ satisfy the equations (\ref{g-eta}) and (\ref{g-phi-n}) weakly.
Equivalently, the equations (\ref{eqn-L-eta}) and (\ref{eqn-L-phi-P}) are
satisfied weakly. Since $\psi \left( x,y\right) $ satisfies the equation (%
\ref{E:psib}a), we have

\begin{equation}
\omega =-\Delta \psi =\gamma ^{\prime }(\psi _{\epsilon })\psi -\gamma
^{\prime }(\psi _{\epsilon })\int_{-\infty }^{0}\lambda e^{\lambda s}\psi
(X_{\epsilon }(s),Y_{\epsilon }(s))ds.  \label{eqn-I-omega}
\end{equation}%
As shown in \cite{lin3}, above equation implies that the vorticity $\omega $
satisfies the equation (\ref{E:vorticityG}) weakly. Equivalently, the
equation (\ref{eqn-L-vor}) is satisfied weakly. In summary, $\left[
e^{\lambda (\epsilon )t}\psi \left( x,y\right) ,e^{\lambda (\epsilon )t}\eta
(x)\right] $ is a weak solution of the linearized system (\ref{eqn-L-vor})-(%
\ref{eqn-L-bottom}).

Our last step of the proof is to get the regularity of the growing-mode $%
(e^{\lambda (\epsilon )t}\eta (x)$, $e^{\lambda (\epsilon )t}\psi (x,y))$
and thus show that it is a classical solution of (\ref{eqn-L-vor})-(\ref%
{eqn-L-bottom}). By (\ref{eqn-b-f}) and Lemma \ref{L:solv}, it follows that $%
\psi \in H^{1}\left( \mathcal{D}_{\epsilon }\right) $. We claim that $\psi
\in H^{2}\left( \mathcal{D}_{\epsilon }\right) $. Indeed, by the trace
theorem, $\psi \in H^{1}\left( \mathcal{D}_{\epsilon }\right) $ implies that 
$\psi (x,\eta _{\epsilon }(x))\in H^{\frac{1}{2}}\left( \mathcal{S}%
_{\epsilon }\right) $. Since the operator $\mathcal{C}^{\lambda }$ is
regularity preserving, by (\ref{eqn-W-phi-n}) $\psi _{n}(x)\in H^{\frac{1}{2}%
}\left( \mathcal{S}_{\epsilon }\right) $. This, together with the facts that 
$\omega =-\Delta \psi \in L^{2}\left( \mathcal{D}_{\epsilon }\right) $ and
that the steady state 
\begin{equation*}
\left( \eta _{\epsilon }(x),\psi _{\epsilon }\left( x,y\right) \right) \in
C^{3+\alpha },\alpha \in \left( 0,1\right)
\end{equation*}%
$\,$(see \cite{cost} or Theorem \ref{T:smallE})$,$ implies that $\psi \in
H^{2}\left( \mathcal{D}_{\epsilon }\right) $ by the regularity theory (\cite%
{ag-book}) of elliptic boundary problems. Then, by using the trace theorem
and (\ref{eqn-W-phi-n}) again, we get $\psi (x,\eta _{\epsilon }(x))$ and $%
\psi _{n}(x)\in H^{\frac{3}{2}}\left( \mathcal{S}_{\epsilon }\right) $.

In order to obtain the higher regularity for $\psi $, we need to show that $%
\omega \in H^{1}(\mathcal{D}_{\epsilon })$. The argument presented below is
a simpler version of that in \cite{lin4}. Taking the gradient of (\ref%
{eqn-I-omega}) yields that 
\begin{equation}
\begin{split}
\nabla \omega =& \nabla (\gamma ^{\prime }(\psi _{\epsilon }))\psi +\gamma
^{\prime }(\psi _{\epsilon })\nabla \psi -\nabla (\gamma ^{\prime }(\psi
_{\epsilon }))\int_{-\infty }^{0}\lambda e^{\lambda s}\psi (X_{\epsilon
}(s),Y_{\epsilon }(s))ds \\
& -\gamma ^{\prime }(\psi _{\epsilon })\int_{-\infty }^{0}\lambda e^{\lambda
s}\nabla \psi (X_{\epsilon }(s),Y_{\epsilon }(s))\frac{\partial (X_{\epsilon
}(s),Y_{\epsilon }(s))}{\partial (x,y)}ds.
\end{split}
\label{eqn-gradient-omega}
\end{equation}%
Note that the particle trajectory is written in the action-angle variables $%
(\theta ,I)=\mathcal{A}_{\epsilon }(x,y)$ as 
\begin{equation*}
\left( X_{\epsilon }(s;x,y),Y_{\epsilon }(s;x,y)\right) =\mathcal{A}%
_{\epsilon }^{-1}((\theta +\upsilon _{\epsilon }(I)s,I)).
\end{equation*}%
This relies on that the action-angle mapping $\mathcal{A}_{\epsilon }$ is
globally defined, as a consequence of the fact that the steady flow has no
stagnation. With the use of the above description of the trajectory the
estimate of the Jacobi matrix 
\begin{equation}
\left\vert \frac{\partial \left( X_{\epsilon }(s;x,y),Y_{\epsilon
}(s;x,y)\right) }{\partial \left( x,y\right) }\right\vert \leq
C_{1}\left\vert s\right\vert +C_{2}  \label{estimate-J-linear}
\end{equation}%
follows, where $C_{1},C_{2}>0$ are independent of $s$.

It is straightforward to see that by calculations as in proving (\ref%
{estimate-k-lb}), the $L^{2}$-norm of the first three terms of (\ref%
{eqn-gradient-omega}) is bounded by the $H^{1}$ norm of $\psi $. The last
term in (\ref{eqn-gradient-omega}) is treated as 
\begin{align*}
\Big\|\gamma ^{\prime }(\psi _{\epsilon })& \int_{-\infty }^{0}\lambda
e^{\lambda s}\nabla \psi (X_{\epsilon }(s),Y_{\epsilon }(s))\frac{\partial
(X_{\epsilon }(s),Y_{\epsilon }(s))}{\partial (x,y)}ds\Big\|_{L^{2}} \\
& \leq ||\gamma ^{\prime }(\psi _{\epsilon })||_{L^{\infty }}\int_{-\infty
}^{0}|\lambda |e^{{\rm Re}\lambda s}(C_{1}\left\vert s\right\vert
+C_{2})||\nabla \psi (X_{\epsilon }(s),Y_{\epsilon }(s))||_{L^{2}\left( 
\mathcal{D}_{\epsilon }\right) }ds \\
& \leq C\left\Vert \psi \right\Vert _{H^{1}(\mathcal{D}_{\epsilon })}.
\end{align*}%
This uses (\ref{estimate-J-linear}), ${\rm Re}\lambda \geq \delta >0$ and
the fact that the mapping $(x,y)\mapsto (X_{\epsilon }(s),Y_{\epsilon }(s))$
is measure-preserving. Therefore, 
\begin{equation*}
\Vert \nabla \omega \Vert _{L^{2}(\mathcal{D}_{\epsilon })}\leq C\left\Vert
\psi \right\Vert _{H^{1}(\mathcal{D}_{\epsilon })}.
\end{equation*}%
In turn, $\omega \in H^{1}(\mathcal{D}_{\epsilon })$. Since $\psi _{n}(x)\in
H^{3/2}(\mathcal{S}_{\epsilon })$, by the elliptic regularity theorem \cite%
{ag-book} it follows that that $\psi \in H^{3}\left( \mathcal{D}_{\epsilon
}\right) $. In view of the trace theorem this implies $\psi _{n}\in H^{5/2}(%
\mathcal{S}_{\epsilon })$.

We repeat the process again. Taking the gradient of (\ref{eqn-gradient-omega}%
) and using the linear stretching property (\ref{linear-streching}) of the
trajectory, it follows that $\omega \in H^{2}\left( \mathcal{D}_{\epsilon
}\right) $. The elliptic regularity applies to assert that $\psi \in H^{4}(%
\mathcal{D}_{\epsilon })\subset C^{2+\beta }(\mathcal{\bar{D}}_{\epsilon })$%
, where $\beta \in (0,1)$. By the trace theorem then it follows that $\psi
(x,\eta _{\epsilon }(x))\in H^{7/2}(\mathcal{S}_{\epsilon })$. On account of
(\ref{eqn-W-eta}) this implies that $\eta \in H^{7/2}\left( \mathcal{S}%
_{\epsilon }\right) \subset C^{2+\beta }([0,2\pi /\alpha ])$. Therefore, $%
(e^{\lambda (\epsilon )t}\eta (x),e^{\lambda (\epsilon )t}\psi (x,y))$ is a
classical solution of (\ref{eqn-L}). This completes the proof.
\end{proof}

\section{Instability of general shear flows}

Linear instability of free-surface shear flows is of independent interests.
This section extends our instability result in Theorem \ref{class-k} to a
more general class of shear flows. The following class of flows was
introduced in \cite{lin1} and \cite{lin3} in the rigid-wall setting.

\begin{definition}
\label{classF} A function $U\in C^{2}([0,h])$ is said to be in the class $%
\mathcal{F}$ if $U^{\prime \prime }$ takes the same sign at all points such
that $U(y)=c$, where $c$ is in the range of $U$ but not an inflection value
of $U$.
\end{definition}

Examples of the class-$\mathcal{F}$ flows include all monotone flows and
symmetric flows with a monotone half. Moreover, if $U^{\prime \prime
}(y)=f(U(y))k(y)$ for $f$ continuous and $k(y)>0$, then $U$ is in class $%
\mathcal{F}$. All flows in class $\mathcal{K}^{+}$ are in class $\mathcal{F}$%
.

The lemma below shows that for a flow in class $\mathcal{F}$ a neutral
limiting wave speed must be an inflection value. The main difference of the
proof from that in the class $\mathcal{K}^{+}$ case (Proposition \ref%
{prop-neutral-mode}) is the lack of an uniform $H^{2}$-bound for the
unstable mode sequence. 
%In contrast, in class $\mathcal{K}^+$ Lemma \ref{lemma-h2-bound} establishes such a uniform bound.

\begin{lemma}
\label{lemma-inflection-F} For $U\in \mathcal{F}$, let $\{(\phi _{k},\alpha
_{k},c_{k})\}_{k=1}^{\infty }$ with ${\rm Im}c_{k}>0$ be a sequence of
unstable solutions satisfying (\ref{rayleigh})--(\ref{bc-rayleigh}). If $%
(\alpha _{k},c_{k})$ converges to $(\alpha _{s},c_{s})$ as $k\rightarrow
\infty $ with $\alpha _{s}>0$ and $c_{s}$ is in the range of $U$, then $%
c_{s} $ must be an inflection value of $U$.
\end{lemma}

\begin{proof}
Suppose on the contrary that $c_{s}$ is not an inflection value. Let $%
y_{1},y_{2},\ldots ,y_{m}$ be in the pre-image of $c_{s}$ so that $%
U(y_{j})=c_{s}$, and let $S_{0}$ be the complement of the set of points $%
\{y_{1},y_{2},\dots ,y_{m}\}$ in the interval $[0,h]$. Since $c_{s}$ is not
an inflection value, Definition \ref{classF} asserts that $U^{\prime \prime
}(y_{j})$ takes the same sign for $j=1,2,\dots ,m$, say positive. As in the
proof of Proposition \ref{prop-neutral-mode}, let $E_{\delta }=\{y\in
\lbrack 0,h]:|y-y_{j}|<\delta \text{ for some $j$, where $j=1,2,\cdots ,m$}%
\} $. It is readily seen that $E_{\delta }^{c}\subset S_{0}$. Note that $%
U^{\prime \prime }(y)>0$ for $y\in E_{\delta }$ if $\delta >0$ small enough.
We normalize the sequence by setting $\Vert \phi _{k}\Vert _{L^{2}}=1$. The
result of Lemma 3.6 in \cite{lin1} implies that $\phi _{k}$ converges
uniformly to $\phi _{s}$ on any compact subset of $S_{0}$. Moreover, $\phi
_{s}^{\prime \prime }$ exists on $S_{0}$ and $\phi _{s}$ satisfies 
\begin{equation*}
\phi _{s}^{\prime \prime }-\alpha _{s}^{2}\phi _{s}-\frac{U^{\prime \prime }%
}{U-c_{s}}\phi _{s}=0\qquad \text{for}\quad y\in (0,h).
\end{equation*}

Our first task is to show that $\phi_{s}$ is not identically zero. Suppose
otherwise. The proof is again divided into two cases.

Case 1: $U(h)\neq c_{s}$. In this case, $[h-\delta _{1},h]\subset S_{0}$ for
some $\delta _{1}>0$. As is done in the proof of \ Proposition \ref%
{prop-neutral-mode}, for any $q$ real, it follows that 
\begin{align*}
\ \int_{0}^{h}& \left( \left\vert \phi _{k}^{\prime }\right\vert ^{2}+\alpha
_{k}^{2}|\phi _{k}|^{2}+\frac{U^{\prime \prime }(U-q)}{|U-c_{k}|^{2}}|\phi
_{k}|^{2}\right) dy \\
& \geq \int_{0}^{h}\left\vert \phi _{k}^{\prime }\right\vert ^{2}dy+\alpha
_{k}^{2}+\int_{E_{\delta }^{c}}\frac{U^{\prime \prime }(U-q)}{|U-c_{k}|^{2}}%
|\phi _{k}|^{2}dy+\int_{E_{\delta }}\frac{U^{\prime \prime }(U-q)}{%
|U-c_{k}|^{2}}|\phi _{k}|^{2}dy \\
& \geq \int_{0}^{h}\left\vert \phi _{k}^{\prime }\right\vert ^{2}dy+\alpha
_{k}^{2}-\sup_{E_{\delta }^{c}}\frac{|U^{\prime \prime }(U-q)|}{|U-c_{k}|^{2}%
}\int_{E_{\delta }^{c}}|\phi _{k}|^{2}dy.
\end{align*}%
We choose $q=U_{\min }-1$, then by (\ref{estimate-Jq-f}) 
\begin{align}
\int_{0}^{h}& \left( \left\vert \phi _{k}^{\prime }\right\vert ^{2}+\alpha
_{k}^{2}|\phi _{k}|^{2}+\frac{U^{\prime \prime }(U-U_{\min }+1)}{%
|U-c_{k}|^{2}}|\phi _{k}|^{2}\right) dy  \label{inter5} \\
& =\left( {\rm Re}g_{r}(c_{k})+({\rm Re}c_{k}-U_{\min }+1)\frac{{\rm Im}%
g_{r}(c_{k})}{{\rm Im}c_{k}}\right) |\phi _{k}(h)|^{2}  \notag \\
& \leq C|\phi _{k}(h)|^{2}\leq C_{1}\left( \varepsilon \int_{h-\delta
_{1}}^{h}\left\vert \phi _{k}^{\prime }\right\vert ^{2}dy+\frac{1}{%
\varepsilon }\int_{h-\delta _{1}}^{h}|\phi _{k}|^{2}dy\right) .  \notag
\end{align}%
If $\varepsilon $ is chosen to be small then the above two inequalities lead
to 
\begin{equation}
0\geq \alpha _{k}^{2}-\sup_{E_{\delta }^{c}}\frac{|U^{\prime \prime
}(U-U_{\min }+1)|}{|U-c_{k}|^{2}}\int_{E_{\delta }^{c}}|\phi
_{k}|^{2}dy-C_{\varepsilon }\int_{h-\delta _{1}}^{h}|\phi _{k}|^{2}dy.
\label{inter8}
\end{equation}%
Since $\phi _{k}$ converges to $\phi _{s}\equiv 0$ uniformly on $E_{\delta
}^{c}$ and $[h-\delta _{1},h]$, this implies $0\geq \alpha _{s}^{2}/2$ when $%
k$ is large enough. A contradiction proves that $\phi _{s}$ is not
identically zero.

Case 2: $U(h)=c_{s}$. From (\ref{estimate-2nd-gs}), we have 
\begin{equation*}
\int_{0}^{h}\left( \left\vert \phi _{k}^{\prime \prime }\right\vert
^{2}+2\alpha _{k}^{2}\left\vert \phi _{k}^{\prime }\right\vert ^{2}+\alpha
_{k}^{4}\left\vert \phi _{k}\right\vert ^{2}\right) dy=2\alpha _{k}^{2}
{\rm Re}g_{s}(c_{k})\left\vert \phi _{k}^{\prime }\left( h\right) \right\vert
^{2}+\int_{0}^{h}\frac{\left( U^{\prime \prime }\right) ^{2}\left\vert \phi
_{k}\right\vert ^{2}}{\left\vert U-c_{k}\right\vert ^{2}}dy.
\end{equation*}%
The imaginary part of (\ref{zweo-g-complex}) yields 
\begin{equation}
{\rm Im}c_{k}\int_{0}^{h}\frac{U^{\prime \prime }\left\vert \phi
_{k}\right\vert ^{2}}{\left\vert U-c_{k}\right\vert ^{2}}dy=-{\rm Im}%
g_{s}(c_{k})\left\vert \phi _{k}^{\prime }\left( h\right) \right\vert ^{2}.
\label{inter10}
\end{equation}%
Denote $U_{\max }^{\prime \prime }=\max_{[0,h]}U^{\prime \prime }(y)$.
Combining the above two identities, we have 
\begin{align}
& \int_{0}^{h}\left( \left\vert \phi _{k}^{\prime \prime }\right\vert
^{2}+2\alpha _{k}^{2}\left\vert \phi _{k}^{\prime }\right\vert ^{2}+\alpha
_{k}^{4}\left\vert \phi _{k}\right\vert ^{2}\right) dy+\int_{0}^{h}\frac{%
U^{\prime \prime }\left( U_{\max }^{\prime \prime }+1-U^{\prime \prime
}\right) \left\vert \phi _{k}\right\vert ^{2}}{\left\vert U-c_{k}\right\vert
^{2}}dy  \label{estimate-2nd-g} \\
& =\left( 2\alpha _{k}^{2}{\rm Re}g_{s}(c_{k})-\frac{{\rm Im}g_{s}(c_{k})}{%
{\rm Im}c_{k}}\left( U_{\max }^{\prime \prime }+1\right) \right) \left\vert
\phi _{k}^{\prime }\left( h\right) \right\vert ^{2}  \notag \\
& \leq C^{\prime }d\left( c_{k},U(h)\right) \left\vert \phi _{k}^{\prime
}\left( h\right) \right\vert ^{2}\leq Cd\left( c_{k},U(h)\right) \left\Vert
\phi _{k}\right\Vert _{H^{2}}^{2},  \notag
\end{align}%
where we use (\ref{bound-gs}). Since 
\begin{equation*}
d(c_{k},U(h))=|{\rm Re}\,c_{k}-U(h)|+(\mathrm{Im}\,c_{k})^{2}\rightarrow 0,
\end{equation*}%
so for $k$ large enough we have 
\begin{equation*}
0\geq \frac{\alpha _{k}^{4}}{2}-\sup_{E_{\delta }^{c}}\frac{\left\vert
U^{\prime \prime }\left( U_{\max }^{\prime \prime }+1-U^{\prime \prime
}\right) \right\vert }{\left\vert U-c_{k}\right\vert ^{2}}\int_{E_{\delta
}^{c}}\left\vert \phi _{k}\right\vert ^{2}dy\geq \frac{\alpha _{s}^{4}}{4},
\end{equation*}%
which is a contradiction. This proves that $\phi _{s}$ is not identically
zero. Subsequently, Lemma \ref{vanish} asserts that $\phi _{s}(y_{j})\neq 0$
for some $y_{j}$.

Below, we get a contradiction from the assumption that $c_{s}$ is not an
inflection value. In Case 1 when $U(h)\neq c_{s}$, it is straightforward to
see that 
\begin{equation*}
\int_{E_{\delta }}\frac{U^{\prime \prime }(U-U_{\min }+1)}{|U-c_{s}|^{2}}%
|\phi _{s}|^{2}dy\geq \int_{|y-y_{j}|<\delta }\frac{U^{\prime \prime }}{%
|U-c_{s}|^{2}}|\phi _{s}|^{2}dy=\infty ,
\end{equation*}%
since $\phi _{s}(y_{j})\neq 0$ and by our assumption $U^{\prime \prime }>0$
on $\left\{ |y-y_{j}|<\delta \right\} $. Fatou's lemma then states that 
\begin{equation*}
\liminf_{k\rightarrow \infty }\int_{E_{\delta }}\frac{U^{\prime \prime
}(U-U_{\min }+1)}{|U-c_{k}|^{2}}|\phi _{k}|^{2}dy=\infty .
\end{equation*}%
Then similar to the estimate (\ref{inter5}) above, we have 
\begin{align*}
0=& \int_{0}^{h}\left( \left\vert \phi _{k}^{\prime }\right\vert ^{2}+\alpha
_{k}^{2}|\phi _{k}|^{2}+\frac{U^{\prime \prime }(U-U_{\min }+1)}{|U-c|^{2}}%
|\phi _{k}|^{2}\right) dy \\
& -\left( {\rm Re}g_{r}(c_{k})+({\rm Re}c_{k}-U_{\min }+1)\frac{{\rm Im}%
g_{r}(c_{k})}{{\rm Im}c_{k}}\right) |\phi _{k}(h)|^{2} \\
\geq & \int_{E_{\delta }}\frac{U^{\prime \prime }(U-U_{\min }+1)}{%
|U-c_{k}|^{2}}|\phi _{k}|^{2}dy-\sup_{E_{\delta }^{c}}\frac{|U^{\prime
\prime }(U-U_{\min }+1)|}{|U-c_{k}|^{2}}-C_{\varepsilon }>0,
\end{align*}%
for $k$ large. A contradiction asserts that $c_{s}$ is an inflection value.
For Case 2 when $U(h)=c_{s}$, similarly we have 
\begin{equation*}
\lim_{k\rightarrow \infty }\inf \int_{E_{\delta }}\frac{U^{\prime \prime
}\left( U_{\max }^{\prime \prime }+1-U^{\prime \prime }\right) }{\left\vert
U-c_{k}\right\vert ^{2}}\left\vert \phi _{k}\right\vert ^{2}dy=+\infty ,
\end{equation*}%
and from (\ref{estimate-2nd-g})%
\begin{equation*}
0\geq \int_{E_{\delta }}\frac{U^{\prime \prime }\left( U_{\max }^{\prime
\prime }+1-U^{\prime \prime }\right) }{\left\vert U-c_{k}\right\vert ^{2}}%
\left\vert \phi _{k}\right\vert ^{2}dy-\sup_{E_{\delta }^{c}}\frac{%
\left\vert U^{\prime \prime }\left( U_{\max }^{\prime \prime }+1-U^{\prime
\prime }\right) \right\vert }{\left\vert U-c_{k}\right\vert ^{2}}>0,
\end{equation*}%
when $k$ is large. Another contradiction completes the proof.
\end{proof}

The proof of above lemma indicates that a flow in class $\mathcal{F}$ is
linearly stable when the wave number is large.

\begin{lemma}
\label{lemma-bound-wave number}Assume $g\neq 1$. Then for any flow $U\left(
y\right) $ in class $\mathcal{F},$ there exists $\alpha _{\max }>0$ such
that when $\alpha \geq \alpha _{\max }$ there is no unstable solutions to (%
\ref{rayleigh})--(\ref{bc-rayleigh}).
\end{lemma}

\begin{proof}
Suppose otherwise. Then, there would exist a sequence of unstable solutions $%
\{(\phi _{k},\alpha _{k},c_{k})\}_{k=1}^{\infty }$ of (\ref{rayleigh})--(\ref%
{bc-rayleigh}) such that $\alpha _{k}\rightarrow \infty $ as $k\rightarrow
\infty $. After normalization, let $\Vert \phi _{k}\Vert _{L^{2}}=1$. First
we show that $\lim_{k\rightarrow \infty }{\rm Im}c_{k}=0$. If $\mathrm{Im}%
\,c_{k}\geq \delta >0$ for some $\delta $, then $1/\left\vert
U-c_{k}\right\vert $ and $|g_{r}(c_{k})|$ are uniformly bounded.
Accordingly, with $q={\rm Re}c_{k}$ in (\ref{estimate-Jq-f}) it follows
that 
\begin{align*}
0& =\int_{0}^{h}\left( \left\vert \phi _{k}^{\prime }\right\vert ^{2}+\alpha
_{k}^{2}\left\vert \phi _{k}\right\vert ^{2}+\frac{U^{\prime \prime }\left(
U-{\rm Re}c_{k}\right) }{\left\vert U-c_{k}\right\vert ^{2}}\left\vert \phi
_{k}\right\vert ^{2}\right) dy-{\rm Re}g_{r}(c_{k})\left\vert \phi
_{k}\left( h\right) \right\vert ^{2} \\
& \geq \alpha _{k}^{2}-\sup \frac{\left\vert U^{\prime \prime }\left( U-%
{\rm Re}c_{k}\right) \right\vert }{\left\vert U-c_{k}\right\vert ^{2}}%
+\int_{0}^{h}\left\vert \phi _{k}^{\prime }\right\vert ^{2}dy-C\left(
\varepsilon \int_{0}^{h}\left\vert \phi _{k}^{\prime }\right\vert ^{2}dy+%
\frac{1}{\varepsilon }\int_{0}^{h}\left\vert \phi _{k}\right\vert
^{2}dy\right) \\
& \geq \alpha _{k}^{2}-\sup \frac{\left\vert U^{\prime \prime }\left( U-%
{\rm Re}c_{k}\right) \right\vert }{\left\vert U-c_{k}\right\vert ^{2}}-%
\frac{C}{\varepsilon }>0,
\end{align*}%
when $k$ is big enough. This contradiction shows that $c_{k}\rightarrow
c_{s}\in \left[ U_{\min },U_{\max }\right] $ when $k\rightarrow \infty $.
The remainder of the proof is nearly identical to that of Lemma \ref%
{lemma-inflection-F} and hence is omitted.
\end{proof}

The following theorem gives a necessary condition for the free surface
instability that the flow profile should have an inflection point, which
generalize the classical result of Lord Rayleigh \cite{ray} in the rigid
wall case.

\begin{theorem}
\label{T:stable} A shear flow $U\left( y\right) $ without an inflection
point is linearly stable in the free surface setting.
\end{theorem}

\begin{proof}
Suppose otherwise; Then, there would exist an unstable solution $(\phi
,\alpha,c)$ to (\ref{rayleigh})--(\ref{bc-rayleigh}) with $\alpha>0$ and $%
{\rm Im}c>0$. Lemma \ref{basis} allows us to continue this unstable mode
for wave numbers to the right of $\alpha$ until the growth rate becomes
zero. Note that a flow without an inflection point is trivially in class $F$%
. So by Lemma \ref{lemma-bound-wave number}, this continuation must end at a
finite wave number $\alpha_{\max}$ and a neutral limiting mode therein. On
the other hand, Lemma \ref{lemma-inflection-F} asserts that the neutral
limiting wave speed $c_{s}$ corresponding to this neutral limiting mode must
be an inflection value. A contradiction proves the assertion.
\end{proof}

\begin{remark}
Our proof of the above no-inflection stability theorem is very different
from the rigid wall case. In the rigid-wall setting, where $\phi(h)=0$, the
identity (\ref{zero-imag}) reduces to%
\begin{equation*}
c_{i}\int_{0}^{h}\frac{U^{\prime\prime}}{|U-c|^{2}}|\phi|^{2}dy=0,
\end{equation*}
which immediately shows that if $U$ is unstable $\left( c_{i}>0\right) $
then $U^{\prime\prime}(y)=0$ at some point $y\in(0,h)$. The same argument
was adapted in \cite[Section 5]{yih72} for the free-surface setting,
however, it does not give linear stability for general flows with no
inflection points. More specifically, in the free-surface setting, (\ref%
{zero-imag}) becomes 
\begin{equation*}
c_{i}\int_{0}^{h}\frac{U^{\prime\prime}}{|U-c|^{2}}\left\vert
\phi\right\vert ^{2}dy=\left( \frac{2g(U(h)-c_{r})}{|U(h)-c|^{4}}+\frac{%
U^{\prime}(h)}{|U(h)-c|^{2}}\right) \left\vert \phi(h)\right\vert ^{2},
\end{equation*}
which only implies linear stability (\cite[Section 5]{yih72}) for special
flows satisfying $U^{\prime\prime}\left( y\right) <0,U^{\prime}\left(
y\right) \geq0$ or $U^{\prime\prime}\left( y\right) >0,U^{\prime}\left(
y\right) \leq0$. In the proof of Theorem \ref{T:stable}, we use the
characterization of neutral limiting modes and remove above additional
assumptions.
\end{remark}

Let us now consider a shear flow $U\in \mathcal{F}$ with multiple inflection
values $U_{1},U_{2},\dots ,U_{n}$. Lemma \ref{lemma-inflection-F} states
that a neutral limiting wave speed $c_{s}$ must be one of the inflection
values $U_{1},U_{2},\dots ,U_{n}$, say $c_{s}=U_{j}$. By localizing the
estimates in the proof of Lemma \ref{lemma-h2-bound} around inflection
points with the inflection values $U_{i}$, we can get an uniform $H^{2}$
bound for the unstable mode sequence. We skip the details, which are similar
to the case of rigid walls treated in \cite{lin3}. Thus, neutral limiting
modes for flows in class $\mathcal{F}$ are also characterized by inflection
values.

\begin{proposition}
\label{neutral mode-F} If $U\in \mathcal{F}$ has inflection values $%
U_{1},U_{2},\dots ,U_{n}$, then for a neutral limiting mode $(\phi
_{s},\alpha _{s},c_{s})$ with $\alpha _{s}>0$ the neutral limiting wave
speed must be one of the inflection values, that is, $c_{s}=U_{j}$ for some $%
j$. Moreover, $\phi _{s}$ must solve 
\begin{equation*}
\phi _{s}^{\prime \prime }-\alpha _{s}^{2}\phi _{s}+K_{j}(y)\phi _{s}=0\quad 
\text{for }\,y\in (0,h),
\end{equation*}%
with boundary conditions%
\begin{equation}
\begin{cases}
\phi _{s}^{\prime }(h)=g_{r}(U_{j}), & \phi _{s}(0)=0\qquad \text{if}\quad
U(h)\neq U_{j} \\ 
\phi _{s}(h)=0,\quad & \phi _{s}(0)=0\qquad \text{if}\quad U(h)=U_{j}%
\end{cases}%
,  \label{bc-sturm-F}
\end{equation}%
where $K_{j}(y)=-U^{\prime \prime }(y)/(U(y)-U_{j})$.
\end{proposition}

One may exploit the instability analysis of Theorem \ref{theorem-unstable}
for a flow in class $\mathcal{F}$ with possibly multiple inflection values.
The main difference of the analysis in class $\mathcal{F}$ from that in
class $\mathcal{K}^{+}$ is that unstable wave numbers in class $\mathcal{F}$
may bifurcate to the left and to the right of a neutral limiting wave
number, whereas unstable wave numbers in class $\mathcal{K}^{+}$ bifurcate
only to the left of a neutral limiting wave number. In the rigid-wall
setting, with an extension of the proof of \cite[Theorem 1.1]{lin1}, Lin 
\cite[Theorem 2.7]{lin3} analyzed this more complicated structure of the set
of unstable wave numbers. The remainder of this section establishes an
analogous result in the free-surface setting.

In order to study the structure of unstable wave numbers in class $\mathcal{F%
}$ with possibly multiple inflection values, we need several notations to
describe. %$(  \alpha_{s},U_{s}^{i},\phi_{s}\right)  $ can be both to
%the left and to the right of the neutral wave number $\alpha_{s}$.
%In the rigid-wall setting, this issue was indicated in \cite{lincc}
A flow $U \in\mathcal{F}$ is said to be in class $\mathcal{F}^{+}$ if each $%
K_{j}(y)=-U^{\prime\prime}(y)/(U(y)-U_{j})$ is nonzero, where $U_{j}$ for $%
j=1, \cdots, n$ are inflection values of $U$. It is readily seen that for
such a flow $K_{j}$ takes the same sign at all inflection points of $U_{j}$.
A neutral limiting mode $(\phi_{j}, \alpha_{j}, U_{j})$ is said to be
positive if the sign of $K_{j}$ is positive at inflection points of $U_{j}$,
and negative if the sign of $K_{j}$ is negative. Proposition \ref{neutral
mode-F} asserts that $-\alpha_{s}^{2}$ is a negative eigenvalue of $-\frac{%
d^{2}}{dy^{2}}-K_{j}(y)$ on $y \in(0,h)$ with boundary conditions (\ref%
{bc-sturm-F}). We employ the argument in the proof of Theorem \ref%
{theorem-unstable} to conclude that an unstable solution exists near a
positive (negative) neutral limiting mode if and only if the perturbed wave
number is slightly to the left (right) of the neutral limiting wave number.
Thus, the structure of the set of unstable wave numbers with multiple
inflection values is more intricate. We remark that a class-$\mathcal{K}^{+}$
flow has a unique positive neutral limiting mode and hence unstable
solutions bifurcate to the left of a neutral limiting wave number.

Let us list all neutral limiting wave numbers in the increasing order. If
the sequence contains more than one successive negative neutral limiting
wave numbers, then we pick the smallest (and discard others). If the
sequence contains more than one successive positive neutral limiting wave
numbers, then we pick the largest (and discard others). If the smallest
member in this sequence is a positive neutral limiting wave number, then we
add zero into the sequence. Thus, we obtain a new sequence of neutral
limiting wave numbers. Let us denote the resulting sequence by $%
\alpha^{-}_{0}<a^{+}_{0}<\cdots <\alpha^{-}_{N}<\alpha^{+}_{N}$, where, $%
\alpha^{-}_{0}$ (might be $0$),$\dots,\alpha^{-}_{N}$ are negative neutral
limiting wave numbers and $\alpha^{+}_{0},\dots,\alpha^{+}_{N}$ are positive
neutral limiting wave numbers. The largest member of the sequence must be a
positive neutral wave number since no unstable modes exist to its right.

\begin{theorem}
\label{theorem class-f}For $U\in\mathcal{F}^{+}$ with inflection values $%
U_{1},U_{2},\dots,U_{n}$, let $\alpha_{0}^{-}<a_{0}^{+}<\cdots<%
\alpha_{N}^{-}<\alpha_{N}^{+}$ be defined as above. For each $%
\alpha\in\cup_{j=0}^{N}(\alpha_{j}^{-},\alpha_{j}^{+})$, there exists an
unstable solution of (\ref{rayleigh})--(\ref{bc-rayleigh}. Moreover, the
flow is linear stable if either $\alpha\geq\alpha_{N}^{+}$ or all operators $%
-\frac{d^{2}}{dy^{2}}-K_{j}(y)$ ($j=1,2,\dots,n$) on $y\in(0,h)$ with (\ref%
{bc-sturm-F} are nonnegative.
\end{theorem}

Theorem \ref{theorem class-f} indicates that there might exist a gap in $%
(0,\alpha _{N}^{+})$ of stable wave numbers. Indeed, in the rigid-wall
setting, a numerical computation \cite{balm} demonstrates that for a certain
shear-flow profile the onset of the unstable wave numbers is away from zero,
that is $\alpha _{0}^{-}>0$.

\renewcommand{\epsilon}{\varepsilon}

\appendix

\section{Proofs of (\protect\ref{i-conver 1}), (\protect\ref{i-conver2}), (%
\protect\ref{integral-limit 2}) and (\protect\ref{formula-Bk-limit-Ud})}

Our first task is to show that the limit (\ref{i-conver 1}) holds as $%
\epsilon\rightarrow0-$ uniformly in $E_{(R,b_{1},b_{2})}$.

In the proof of Theorem \ref{theorem-unstable}, we have already established
that both $\bar{\phi}_{1}\left( y;\varepsilon ,c\right) $ and $\phi
_{0}\left( y;\varepsilon ,c\right) $ uniformly converge to $\phi _{s}$ in $%
C^{1}$, as $(\epsilon ,c)\rightarrow (0,0)$ in $E_{(R,b_{1},b_{2})}$.
Moreover, $\phi _{2}(0;\epsilon ,c)\rightarrow -\frac{1}{\phi _{s}^{\prime
}(0)}$ uniformly as $(\epsilon ,c)\rightarrow (0,0)$ in $E_{(R,b_{1},b_{2})}$%
. So the function 
\begin{equation*}
G(y,0;\epsilon ,c)\phi _{0}(y;\epsilon ,c)=\left( \bar{\phi}_{1}\left(
y;\varepsilon ,c\right) \phi _{2}\left( 0;\varepsilon ,c\right) -\phi
_{2}\left( y;\varepsilon ,c\right) \bar{\phi}_{1}\left( 0;\varepsilon
,c\right) \right) \phi _{0}(y;\epsilon ,c)
\end{equation*}%
converges uniformly to $-\phi _{s}^{2}\left( y\right) /\phi _{s}^{\prime
}\left( 0\right) $ in $C^{1}[0,d],$as $(\epsilon ,c)\rightarrow (0,0)$ in $%
E_{(R,b_{1},b_{2})}.$Then the uniform convergence of (\ref{i-conver 1})
follows from (\ref{differ1}).

The proof of (\ref{i-conver2}) uses the following lemma.

\begin{lemma}[\protect\cite{lin1}, Lemma 7.3]
\label{interlemma}Assume that a sequence of differentiable functions $%
\{\Gamma_{k}\}_{k=1}^{\infty}$ converges to $\Gamma_{\infty}$ in $C^{1}$ and
that $\{c_{k}\}_{k=1}^{\infty}$ converges to zero, where $\mathrm{Im}%
\,c_{k}>0$ and $|{\rm Re}\,c_{k}|\leq R\,\mathrm{Im}\,c_{k}$ for some $R>0$%
. Then, 
\begin{equation}
\lim_{k\rightarrow\infty}-\int_{0}^{h}\frac{U^{\prime\prime}}{%
(U-U_{s}-c_{k})^{2}}\Gamma_{k}dy=p.v.\int_{0}^{h}\frac{K(y)}{U-U_{s}}%
\Gamma_{\infty }dy+i\pi\sum_{i=1}^{m_{s}}\frac{K(a_{j})}{|U^{\prime}(a_{j})|}%
\phi_{s}(a_{j}),  \label{2-order limit}
\end{equation}
provided that $U^{\prime}(y)\neq0$ at each $a_{j}$. Here $a_{1},\dots
,a_{m_{s}}$ are roots of $U-U_{s}$.
\end{lemma}

We now prove (\ref{i-conver2}). That is, We shall show that (\ref{i-conver2}%
) holds uniformly in $E_{(R,b_{1},b_{2})}$. Suppose for some $\delta >0$ and
a sequence $\{(\varepsilon _{k},c_{k})\}_{k=1}^{\infty }$ in $%
E_{(R,b_{1},b_{2})}$with $\max (b_{1}^{k},b_{2}^{k})$ tending to zero 
\begin{equation*}
|\frac{\partial \Phi }{\partial c}(\varepsilon _{k},c_{k})-(C+iD)|>\delta
\end{equation*}%
holds, where $C$ and $D$ are defined in (\ref{E:BCD}). 
%\[C+iD=-\frac{1}{\phi_{s}^{\prime}\left(  0\right)  }\left(  i\pi\sum_{k=1}%
%^{l}\left(  \left\vert U^{\prime}\right\vert ^{-1}K\phi_{s}^{2}\right)
%|_{a_{k}}+P\int_{0}^{d}\left(  K\left(  y\right)  \phi_{s}^{2}\left(
%y\right)  \right)  /\left(  U\left(  y\right)  -U_{s}\right)  dy+A\right)\]
%with $A$ defined in (\ref{defn-A}). But
Let us write 
\begin{equation*}
\frac{\partial \Phi }{\partial c}(\varepsilon _{k},c_{k})=-\int_{0}^{h}\frac{%
U^{\prime \prime }}{(U-U_{s}-c_{k})^{2}}\Gamma _{k}dy+\frac{d}{dc}%
g_{r}(U_{s}+c)\phi _{2}(0;\varepsilon _{k},c_{k})=I+II,
\end{equation*}%
where 
\begin{equation*}
\Gamma _{k}(y)=G(y,0;\varepsilon _{k},c_{k})\phi _{0}(y;\varepsilon
_{k},c_{k})\rightarrow -\frac{1}{\phi _{s}^{\prime }(0)}\phi _{s}^{2}\quad 
\text{in }C^{1}.
\end{equation*}%
%
%
%
%
%
%
%
%
%
%
%
%
%
%
%&  =\left(  \bar{\phi}_{1}\left(  y;\varepsilon,c\right)  \phi_{2}\left(
%0;\varepsilon,c\right)  -\phi_{2}\left(  y;\varepsilon,c\right)  \bar{\phi
%}_{1}\left(  0;\varepsilon,c\right)  \right)  \tilde{\phi}\left(y;c_{k},\varepsilon_{k}\right)\end{align*}
%Since $\phi_0$ converges in $C^{1}$ to $\phi_{s}$ and $\phi_{2}(  0;\varepsilon,c)
%\rightarrow-\frac{1}{\phi_{s}^{\prime}(  0)  }$, we deduce that $\Gamma_{k}\left(
%y\right)  $ converges to $-\phi_{s}^{2}\left(  y\right)  /\phi_{s}^{\prime
%}\left(  0\right)  $ in $C^{1}[0,d]$ $.$Then we have
By Lemma \ref{interlemma}, it follows that 
\begin{equation*}
\lim_{k\rightarrow \infty }I=-\frac{1}{\phi _{s}^{\prime }(0)}\left(
p.v.\int_{0}^{h}\frac{K(y)}{U-U_{s}}\phi _{s}^{2}dy+i\pi \sum_{k=1}^{m_{s}}%
\frac{K(a_{j})}{|U^{\prime }(a_{j})}\phi _{s}^{2}(a_{j})\right) .
\end{equation*}%
A straightforward calculation yields that 
\begin{equation*}
\lim_{k\rightarrow \infty }II=-\frac{d}{dc}g_{r}(U_{s}+c)\frac{1}{\phi
_{s}^{\prime }(0)}=-\left( \frac{2g}{(U(h)-U_{s})^{3}}+\frac{U^{\prime }(h)}{%
(U(h)-U_{s})^{2}}\right) \frac{1}{\phi _{s}^{\prime }(0)}=-\frac{A}{\phi
_{s}^{\prime }(0)},
\end{equation*}%
where $A$ is given in (\ref{definition-A}). Therefore, 
\begin{equation*}
\lim_{k\rightarrow \infty }\frac{\partial \Phi }{\partial c}(\varepsilon
_{k},c_{k})=C+iD.
\end{equation*}%
A contradiction then proves the uniform convergence.

The proofs of (\ref{integral-limit 2}) and (\ref{formula-Bk-limit-Ud}) use
the following lemma.

\begin{lemma}[\protect\cite{lin1}, Lemma 7.1]
\label{singular limit} Assume that $\{\psi_{k}\}_{k=1}^{\infty}$ converges
to $\psi_{\infty}$ in $C^{1}([0,h])$ and that $\{c_{k}\}_{k=1}^{\infty}$
with $\mathrm{Im}\,c_{k}>0$ converges to zero. Let us denote $%
W_{k}(y)=U(y)-U_{s}-{\rm Re}\,c_{k}.$ Then, the limits 
\begin{align}
\lim_{k\rightarrow\infty}\int_{0}^{h}\frac{W_{k}}{W_{k}^{2}+\mathrm{Im}%
\,c_{k}^{2}}\psi_{k}dy & =p.v.\int_{0}^{h}\frac{\psi_{\infty}}{U-U_{s}}dy,
\label{auxially-ineg1} \\
\lim_{k\rightarrow\infty}\int_{0}^{h}\frac{\mathrm{Im}\,c_{k}}{W_{k}^{2}+%
\mathrm{Im}\,c_{k}^{2}}\psi_{k}dy & =\pi\sum_{j=1}^{m_{s}}\frac {%
\psi_{\infty}(a_{j})}{|U(a_{j})|}  \label{auxially-ineg3}
\end{align}
hold provided that $U^{\prime}(y)\neq0$ at each $a_{j}$. Here, $%
a_{1},\ldots,a_{m_{s}}$ satisfy $U(a_{j})=U_{s}$.
\end{lemma}

We now prove (\ref{integral-limit 2}). Since $K(y)\phi_{s}\phi_{k}%
\rightarrow K(y)\phi_{s}^{2}$ in $C^{1}([0,h])$, by (\ref{auxially-ineg1})
and (\ref{auxially-ineg3}) it follows that 
\begin{align}
-\int_{0}^{h} & \frac{U^{\prime\prime}}{(U-c_{k})(U-U_{s})}\phi_{s}\phi
_{k}dy  \notag \\
& =\int_{0}^{h}\frac{W_{k}}{W_{k}^{2}+\mathrm{Im}\,c_{k}^{2}}%
K(y)\phi_{s}\phi_{k}dy+i\int_{0}^{h}\frac{\mathrm{Im}\,c_{k}}{W_{k}^{2}+%
\mathrm{Im}\,c_{k}^{2}}K(y)\phi_{s}\phi_{k}dy  \label{limit-I} \\
& \rightarrow p.v.\int_{0}^{h}\frac{K}{(U-U_{s})}\phi_{s}^{2}dy+i\pi
\sum_{j=1}^{m_{s}}\frac{K(a_{j})}{|U(a_{j})|}\phi_{s}^{2}(a_{j})  \notag
\end{align}
as $k\rightarrow\infty$. Since $c_{k}\rightarrow U_{s}$ as $k\rightarrow
\infty$ in the proof of Theorem \ref{class-k}, it follows that 
\begin{equation}
\lim_{k\rightarrow\infty}\frac{g_{r}(c_{k})-g_{r}(U_{s})}{c_{k}-U_{s}}%
=g_{r}^{\prime}(U_{s})=\frac{2g}{(U(h)-U_{s})^{3}}+\frac{U^{\prime}(h)}{%
(U(h)-U_{s})^{2}}.  \label{limit-II}
\end{equation}
Addition of (\ref{limit-I}) and (\ref{limit-II}) proves (\ref{integral-limit
2}). In case $U(h)=U_{s}$ the same computations as above prove (\ref%
{formula-Bk-limit-Ud}). This uses that 
\begin{equation*}
\lim_{k\rightarrow\infty}\frac{g_{s}(c_{k})}{c_{k}-U_{s}}=-\lim_{k%
\rightarrow \infty}\frac{U(h)-c_{k}}{g+U^{\prime}(h)(U(h)-c_{k})}=0.
\end{equation*}

\subsection*{Acknowledgement}

The work of Vera Mikyoung Hur is supported partly by the NSF grant
DMS-0707647, and the work of Zhiwu Lin is supported partly by the NSF grants
DMS-0505460 and DMS-0707397. The authors thank the anonymous referee for
valuable comments and suggestions.

\end{document}